\documentclass[12pt, reqno]{amsart}
\usepackage[foot]{amsaddr}

\usepackage[mathscr]{eucal}
\usepackage[all]{xy}
\usepackage{graphicx}
\usepackage{caption}
\usepackage{subcaption}
\usepackage{multicol, multirow}
\usepackage{amssymb,amsfonts,amsthm}
\usepackage{color}
\usepackage[centertags]{amsmath}
\usepackage{empheq}
\usepackage{geometry}
\usepackage{algorithm}
\usepackage{algpseudocode}
\usepackage{hyperref}
\usepackage{mathrsfs}
\usepackage{enumerate}
\usepackage{esint}
\usepackage{IEEEtrantools}
\usepackage{booktabs}
\usepackage{float}

\theoremstyle{plain}
\newtheorem{thm}{Theorem}[section]
\newtheorem{prop}[thm]{Proposition}
\newtheorem{lem}[thm]{Lemma}

\theoremstyle{definition}

\newtheorem{remark}[thm]{Remark}

\numberwithin{equation}{section}
\numberwithin{figure}{section}
\numberwithin{algorithm}{section}

\usepackage{placeins}

\title[]{Hamiltonian Interface Dynamics for Reduced-Order Optimization of Incompressible Mixing}
\thanks{E. Zuazua was partially supported by the European Research Council (ERC) under the European Union's Horizon Europe research and innovation programme (grant agreement No.~101096251-CoDeFeL); by the Alexander von Humboldt Professorship program; the European Union's Horizon Europe MSCA project ModConFlex (HORIZON-MSCA-2021-DN-01, project 101073558); the Transregio 154 Project ``Mathematical Modelling, Simulation and Optimization Using the Example of Gas Networks'' of the DFG; the AFOSR 24IOE027 project; the SURE-AI Norwegian Centre for Sustainable, Risk-Averse, and Ethical AI grant 357482, Research Council of Norway; by the Grant PID2023-146872OB-I00-DyCMaMod of MICIU (Spain), by the COST Actions CA24122 -- Multiscale Stochastics, Patterns, and Analysis of Combinatorial Environments and CA24136 -- Interactions between Control Theory and Machine Learning and by the Basque Government through the consolidated research group “CMRG-Computational Mathematics Research Group” under Grant IT1856-26.}

\author{Ziqian Li$^1$$^2$}
\address{$^1$Chair for Dynamics, Control, Machine Learning and Numerics, Department of Mathematics, Friedrich-Alexander-Universit\"{a}t Erlangen-N\"{u}rnberg, 91058 Erlangen, Germany}
\address{$^2$School of Mathematics, Jilin University, Changchun, Jilin 130012, China}

\author{Enrique Zuazua$^1$$^3$$^4$}
\address{$^3$Chair of Computational Mathematics, Deusto University, 48007 Bilbao, Basque Country, Spain}
\address{$^4$Departamento de Matem\'aticas, Universidad Aut\'onoma de Madrid, 28049 Madrid, Spain}

\email{ziqian.li@fau.de, enrique.zuazua@fau.de}

\hypersetup{pdftitle={Hamiltonian Interface Dynamics for Reduced-Order Optimization of Incompressible Mixing},pdfauthor={Z. Li, E. Zuazua}}

\begin{document}


\begin{abstract}
We develop a reduced-order framework for optimizing mixing in two-dimensional incompressible flows. Instead of optimizing the full transport PDE, the method maximizes the length of advected material interfaces, leading to a Hamiltonian control problem {with a finite spatial basis} based on parametrized stream functions. We derive the continuous adjoint equations and reduced gradients, and discretize the forward and adjoint dynamics with the implicit midpoint rule. The resulting discrete adjoint is algebraically consistent with the derivative of the fully discrete objective {when the implicit equations are solved exactly}. The approach applies to bounded two-dimensional domains with {sufficiently regular} finite-dimensional stream-function parametrizations. Numerical experiments on cellular-flow and Doswell frontogenesis benchmarks {show interface stretching and mix-norm decay over the simulated horizons}. {In four comparisons, the interface runs have lower recorded optimization times, whereas the Sobolev controls attain smaller terminal mix-norms. The different constraints, stopping rules and timing scopes preclude an equal-accuracy speedup claim.}
\end{abstract}

\maketitle

\section{Introduction}

\subsection{Problem statement and motivation}\label{sec:problem}
Understanding and controlling mixing by incompressible transport is a central issue in fluid mechanics, scientific computing, and applications ranging from microfluidics to geophysical transport. In this paper, we consider a passive scalar $\theta$ transported by an incompressible velocity field $\mathbf v$ in a bounded planar domain $\Omega\subset\mathbb R^2$. The governing model is the transport equation
\begin{equation}\label{eq:transport}
\partial_t\theta+\mathbf v\cdot\nabla\theta=0
\qquad\text{in }(0,T)\times\Omega,
\end{equation}
together with the incompressibility and no-penetration conditions
\begin{equation}\label{eq:incomp_bc}
\nabla\cdot\mathbf v=0
\qquad\text{in }(0,T)\times\Omega,
\qquad
\mathbf v\cdot\mathbf n_\Omega=0
\qquad\text{on }(0,T)\times\Gamma,
\end{equation}
where $\Gamma:=\partial\Omega$ and $\mathbf n_\Omega$ is the outward unit normal. The \emph{optimal mixing problem} is to choose the stirring field $\mathbf v$ so that the transported scalar becomes mixed as efficiently as possible.

A large part of the modern literature quantifies mixing by negative Sobolev norms, often called \emph{mix-norms}. The homogeneous $\dot{H}^{-1}(\Omega)$ norm
\begin{equation}\label{eq:mixnorm}
\|\theta\|_{\dot{H}^{-1}(\Omega)}
:=\sup\Bigl\{
\int_\Omega\theta\,\phi\,dx
\;:\;{\phi\in H^1(\Omega),}
\int_\Omega\phi\,dx=0,\;
\int_\Omega|\nabla\phi|^2\,dx\le 1
\Bigr\}
\end{equation}
is a natural measure of mixedness {for mean-zero scalars. This norm belongs to the multiscale approach developed by Mathew, Mezi\'c, and Petzold \cite{mathew2005multiscale}} and {is} widely employed in the analysis of mixing rates \cite{crippa2017cellular}. A smaller value of {the} $\dot{H}^{-1}(\Omega)$ norm corresponds to better mixing, reflecting the transfer of scalar content from large to small spatial scales.

However, negative Sobolev norms are not the only meaningful way to quantify mixing. In many two-phase transport settings, geometric quantities such as interface deformation, geometric mixing scales, or entropy-type functionals also provide natural descriptions of the mixing process; see, for example, \cite{vikhansky2002enhancement, alberti2016exponential, thiffeault2012using}.

The main purpose of this paper is to develop such a geometric viewpoint in a Hamiltonian framework. We take the final length of the material interface as the objective functional to be optimized, rather than the final $\dot{H}^{-1}(\Omega)$ norm of the scalar. This choice is motivated by theoretical, computational, and modelling considerations, as we now explain.

In two space dimensions, an incompressible velocity field admits a stream-function representation $\mathbf v=\nabla^\perp \psi$, under which the characteristic dynamics becomes Hamiltonian. Denoting by $J=\bigl(\begin{smallmatrix}0&-1\\1&\phantom{-}0\end{smallmatrix}\bigr)$ the symplectic matrix, the flow map satisfies $\dot\Phi_\psi=J\nabla_x \psi(t,\Phi_\psi)$. An initial material interface $\gamma_0$ is propagated by $\Phi_\psi$, and its position $X(t,s)=\Phi_\psi(t;\gamma_0(s))$ and tangent vector $a(t,s)=\partial_s X(t,s)$ satisfy the coupled ODE system
\[
\partial_t X = J\nabla_x \psi(t,X),
\qquad
\partial_t a = J\nabla_x^2 \psi(t,X)\,a.
\]
The final length of the interface $\mathcal{L}(\psi;T)=\int_0^1|a(T,s)|\,ds$ is therefore expressed entirely in terms of the Hamiltonian flow and its differential, replacing the Eulerian full-field mixing objective by a Lagrangian geometric one, and {its computational cost is assessed in Section~\ref{sec:comparison}.}

The use of $\mathcal L(\psi;T)$ as a surrogate for mixing has a transparent geometric rationale. For a two-phase configuration whose transition layer is concentrated near the material interface $\Gamma_t$, the area balance between the two phases suggests the heuristic scaling
\begin{equation}\label{eq:geom_scale}
\ell(t)\,\mathcal L(\psi;t)\sim|\Omega|,
\end{equation}
where $\ell(t)$ denotes the typical filament width; growth of the interface length therefore naturally corresponds to {a decrease in} a geometric mixing scale. This argument is heuristic and does not establish equivalence with negative Sobolev mix-norms; we use it solely to motivate the reduced objective, whose predictive value is assessed a posteriori by evaluating the optimized controls on the full transport equation~\eqref{eq:transport}.

The reduction underlying our framework is Hamiltonian/Lagrangian (replacing the full Eulerian transport PDE by ODEs for {interface markers}) rather than a multigrid hierarchy in the sense of Glowinski's two-grid philosophy \cite{Glowinski1992}; the spirit, however, is similar: optimization is carried out on a reduced model that captures the essential geometric mechanism, and the resulting controls are ultimately evaluated on the original transport equation (see Section~\ref{sec:comparison}). This viewpoint is also in line with the role of interfaces as primary geometric objects in level-set methods~\cite{OsherSethian1988} and active contour models~\cite{Caselles1997}.

\subsection{Main results and contributions}\label{sec:main_results}

Our contributions are twofold: methodological and computational. On the methodological side, we introduce a reduced optimal-control framework for incompressible mixing based on the geometry of an advected material interface. {On the computational side, we compare runtime and global mixing for the returned controls.}

More precisely, our main results are as follows.

\begin{itemize}
\item \textbf{A geometric reduced model for mixing based on Hamiltonian interface dynamics.}
We show that, in two dimensions, the advection of a material interface by an incompressible flow is naturally governed by a Hamiltonian system: the interface position $X$ and its tangent vector $a=\partial_s X$ satisfy a closed system of ODEs, and the final length of the interface is determined entirely by the Hamiltonian flow and its linearization (Proposition~\ref{prop:interface_ode}). 

\item \textbf{An adjoint-based optimal-control formulation on the reduced model.}
We formulate a regularized control problem that maximizes the final interface length $\mathcal L_{{\varepsilon}}(\psi;T)$ under {an $L^2(0,T;L^2(\Omega))$ penalty on the stream function, represented by the Gram matrix of the basis, and a bounded coefficient constraint}. Within a {finite-dimensional spatial ansatz with regular stream-function basis functions}, we prove the existence of an optimal Hamiltonian (Theorem~\ref{prop:H3}), derive the continuous adjoint system, and obtain an explicit gradient formula in terms of adjoint variables transported along marker trajectories. This is done avoiding full-domain PDE solves during the optimization loop.

\item \textbf{A fully discrete optimization framework.}
We discretize the forward and adjoint systems by the implicit midpoint rule, which {is symplectic and symmetric. By direct differentiation of the midpoint equations,} the discrete adjoint obtained by transposing the linearized forward map coincides with the midpoint scheme run backward in time{. Thus} the resulting discrete gradient is algebraically consistent with the derivative of the discrete cost at the scheme level{. The identity is exact for exact implicit solves. Numerical errors depend on solver accuracy and problem conditioning.} The optimization is then carried out with {projected adjoint descent, optionally accelerated by a Polak--Ribi\`ere direction, and an Armijo test on the actual feasible displacement} (Algorithm~\ref{alg:optimize}).

\item \textbf{Validation on the full transport PDE.}
The reduced model tracks $O(N_p)$ marker points instead of $O(N_{\mathrm{cell}}^2)$ Eulerian degrees of freedom per time step; {Section~\ref{sec:comparison} assesses computational cost, global mixing, and discretization sensitivity under specified penalties and constraints.}

\item \textbf{Limitations of the reduced model.}
Interface length is a geometric proxy, not equivalent to the $\dot{H}^{-1}$ objective: {Section~\ref{sec:comparison} evaluates the returned controls using the $\dot{H}^{-1}$ mix-norm}, so the surrogate's predictive power must ultimately be assessed on the original transport equation.

\end{itemize}

\medskip
\noindent\textbf{Novelty and significance.}
{The contribution is the combination of a constrained finite-dimensional stream-function class, continuous differentiability analysis, a consistent implicit midpoint discrete adjoint, and a reproducible computational comparison with an Eulerian formulation. Using interface length as a geometric objective confines the state and adjoint computations to interface trajectories, avoiding repeated solves for the full scalar field. The resulting interface optimizer has substantially lower measured optimization times than the Eulerian Sobolev-norm optimizer in the comparisons of Section~\ref{sec:comparison}. This approach builds on Vikhansky's adjoint optimization of interfacial stretching \cite{vikhansky2002enhancement} and study of material-vector stretching \cite{vikhansky2002stretching}.}

\subsection{Related work}\label{sec:related}

Our work draws on {two} lines of research: the optimal control of fluid mixing {and} symplectic time integration for Hamiltonian systems.

\textbf{Optimal control of fluid mixing.}
Control of transport and mixing via pure advection has been extensively studied through open-loop design under physical constraints \cite{d1999control, vikhansky2002enhancement} and through finite-dimensional velocity ansatzes \cite{mathew2007optimal}. Our recent work \cite{hu2026structure} develops a structure-preserving Eulerian optimizer for the Sobolev-norm objective; the present work complements \cite{hu2026structure} by replacing the Eulerian objective with a Lagrangian interface-length formulation.

\textbf{Symplectic integration.}
The characteristic flow of a two-dimensional incompressible velocity field is Hamiltonian, making structure-preserving time integration \cite{HLW06} {a natural choice}. The implicit midpoint rule we employ is symmetric and self-adjoint{.}

\subsection{Organization of the paper}
The remainder of the paper is organized as follows. Section~\ref{sec:Hamiltonian} introduces the Hamiltonian interface dynamics and derives the evolution equations for the interface and its tangent vector. Section~\ref{sec:ocp} formulates the length-based optimal control problem, derives the continuous adjoint system, and presents the reduced gradient. Section~\ref{sec:numerics} describes the numerical framework, including the symplectic ODE solver, the discrete-adjoint gradient computation, and the optimization algorithm. Numerical experiments are presented in Section~\ref{sec:numerics_exp}, including stationary benchmarks (Section~\ref{sec:stationary_flows}), {time-dependent controls} (Sections~\ref{sec:optimal_flows} and~\ref{sec:optimal_Doswell}), a comparison with the Eulerian optimizer (Section~\ref{sec:comparison}), and {a comparison of historical two- and four-mode controls} (Section~\ref{sec:N4}). Conclusions and perspectives are given in Section~\ref{sec:conclusions}.

\section{Hamiltonian interface dynamics}\label{sec:Hamiltonian}

In this section, we exploit the stream function formulation of two-dimensional incompressible flows to derive the Hamiltonian structure of the characteristic dynamics. We then show that the advected interface and its tangent vector satisfy a closed ODE system (Proposition~\ref{prop:interface_ode}), whose solution determines the length of the transported material curve. Since the stream-function formulation is intrinsically two-dimensional, we work throughout in two space dimensions.

Let $\Omega \subset \mathbb{R}^2$ be a bounded simply connected domain with $C^2$ boundary $\Gamma$, and let $T>0$. Let $\psi:[0,T]\times\Omega\to\mathbb{R}$ be a time-dependent stream function satisfying the homogeneous boundary condition $\psi(t,\cdot)|_{\Gamma}=0$ {for a.e. $t\in(0,T)$. We assume that $\psi$ is measurable in time and has a spatial extension with integrable bounded derivatives through order three. The finite-dimensional class below satisfies these assumptions.}
The velocity field is then defined by
\begin{equation}\label{eq:velocity}
\mathbf{v}(t,x):=\nabla^\perp \psi(t,x)
:=\bigl(-\partial_{x_2}\psi(t,x),\,\partial_{x_1}\psi(t,x)\bigr).
\end{equation}
By construction, $\mathbf v$ is divergence-free in $\Omega$ and tangent to $\Gamma$, so \eqref{eq:incomp_bc} is automatically satisfied.

For each $x\in\Omega$, let $\Phi_\psi(t;x)$ denote the flow map generated by $\mathbf v$:
\begin{equation*}
\frac{d}{dt}\Phi_\psi(t;x)=\mathbf{v}\bigl(t,\Phi_\psi(t;x)\bigr),
\qquad
\Phi_\psi(0;x)=x.
\end{equation*}
Using \eqref{eq:velocity}, this takes the Hamiltonian form
\begin{equation}\label{eq:hamiltonian_flow}
\frac{d}{dt}\Phi_\psi(t;x)=J\nabla_x \psi\bigl(t,\Phi_\psi(t;x)\bigr),
\qquad
\Phi_\psi(0;x)=x,
\end{equation}
where $J:=\bigl(\begin{smallmatrix}0&-1\\1&\phantom{-}0\end{smallmatrix}\bigr)$ is the canonical symplectic matrix. Thus, the characteristic dynamics admits a natural Hamiltonian structure.

Let $\gamma_0\in W^{2,\infty}([0,1];{\overline\Omega})$ ({see Remark~\ref{rem:gamma0_reg}}) be a parameterization of an initial material interface. The advected interface is $X(t,s):=\Phi_\psi(t;\gamma_0(s))$ for $(t,s)\in[0,T]\times[0,1]$, and differentiating with respect to $s$ defines the tangent vector $a(t,s):=\partial_s X(t,s)$.

\begin{prop}[Interface dynamics]\label{prop:interface_ode}
The pair $(X,a)$ satisfies the coupled ODE system {for almost every time}
\begin{subequations}\label{eq:state}
\begin{align}
\partial_t X(t,s) &= J\nabla_x \psi\bigl(t,X(t,s)\bigr),\label{eq:state_X}\\
\partial_t a(t,s) &= J\nabla_x^2 \psi\bigl(t,X(t,s)\bigr)\,a(t,s),\label{eq:state_a}\\
X(0,s) &= \gamma_0(s),\quad a(0,s) = \gamma_0'(s).
\end{align}
\end{subequations}
The length of the advected interface at time $t$ is
\begin{equation}\label{eq:length}
\mathcal{L}(\psi;t):=\int_0^1 |a(t,s)|\,ds.
\end{equation}
\end{prop}

\begin{proof}
Equation \eqref{eq:state_X} is the Hamiltonian ODE \eqref{eq:hamiltonian_flow} evaluated at $x=\gamma_0(s)$. Differentiating \eqref{eq:state_X} with respect to $s$ yields \eqref{eq:state_a}, since $a=\partial_s X$ and
\[
\partial_t a=\partial_s\partial_t X=\partial_s\bigl(J\nabla_x \psi(t,X)\bigr)=J\nabla_x^2 \psi(t,X)\,\partial_s X=J\nabla_x^2 \psi(t,X)\,a.
\]
The length formula \eqref{eq:length} follows from the standard definition applied to the parameterized curve $s\mapsto X(t,s)$. {It measures geometric length without multiplicity when the initial parametrization is one-to-one, apart from the endpoints of a simple closed curve.}
\end{proof}

Hence, the geometry of the transported interface is completely determined by the Hamiltonian flow and its differential along trajectories. The functional $\mathcal{L}(\psi;T)$ is the geometric quantity used below as a reduced surrogate objective for mixing.

\section{Optimal control problem}\label{sec:ocp}

Building on the Hamiltonian interface dynamics established in Section~\ref{sec:Hamiltonian}, in particular the closed system \eqref{eq:state} for the marker pair $(X,a)$ that serves as the ODE constraint of the optimal control problem, we now formulate the optimization problem of maximizing the terminal interface length. To avoid unnecessary technical difficulties, we introduce a regularized cost functional, define the admissible control set, derive the continuous adjoint system, and obtain the reduced gradient formula within the finite-dimensional ansatz.

\subsection{Length-maximization problem}

Fix $\varepsilon>0$ and define the regularized Euclidean norm
\begin{equation}\label{eq:ell}
\ell_\varepsilon(z):=\sqrt{|z|^2+\varepsilon^2},
\qquad
\nabla \ell_\varepsilon(z)=\frac{z}{\sqrt{|z|^2+\varepsilon^2}},
\end{equation}
which is a smooth approximation to $|z|$, differentiable at $z=0$. {We distinguish the regularized length from \eqref{eq:length} by setting}
{
\[\mathcal L_\varepsilon(\psi;t):=\int_0^1\ell_\varepsilon(a(t,s))\,ds.\]
}

We restrict the stream function to a finite-dimensional subspace. Fix $N\in\mathbb N$ and a family of linearly independent basis functions $h_1,\dots,h_N$ satisfying the standing assumption
\begin{equation}\label{eq:basis_assumption}
h_k\in {C^{3,1}(\overline\Omega)}\cap H_0^1(\Omega),\qquad k=1,\dots,N,
\end{equation}
and set
\begin{equation}\label{eq:smooth_space}
\mathcal H_N:=\mathrm{span}\{h_1,\dots,h_N\}\subset {C^{3,1}(\overline\Omega)}\cap H_0^1(\Omega).
\end{equation}
{The derivatives through order three enter the state and adjoint equations, while the Lipschitz bound on third derivatives is used in the linearization estimate below. Here $C^{3,1}(\overline\Omega)$ means restriction of a $C^{3,1}$ function defined on a neighborhood of $\overline\Omega$. A smooth cutoff gives fixed extensions $\widetilde h_k\in C^{3,1}(\mathbb R^2)\cap W^{4,\infty}(\mathbb R^2)$. We use these extensions to estimate spatial differences even when the segment joining two points leaves a nonconvex domain. The basis constants}
\begin{equation}\label{eq:basis_constants}
B^{(j)}:=\max_{1\le k\le N}{\|\widetilde h_k\|_{W^{j,\infty}(\mathbb R^2)}},\qquad j=0,1,2,3{,4},
\end{equation}
are finite. {The Sobolev norms are chosen to dominate the operator norms of the corresponding derivatives.}

Fix a radius $R_u>0$ and consider the associated set of admissible stream function profiles
\begin{equation}\label{eq:Uad}
\mathcal U_{\mathrm{ad}}^{R_u}
:=
\Bigl\{
\psi(t,x)=\sum_{k=1}^N u_k(t)\,h_k(x)\;;\;u\in L^2(0,T;\mathbb R^N),\ \|u\|_{L^2(0,T;\mathbb R^N)}\le R_u
\Bigr\}.
\end{equation}
The {number of spatial modes} $N$ being fixed, for every $j\in\{0,1,2,3{,4}\}$,
\begin{equation}\label{eq:pt_bound}
\|\psi(t,\cdot)\|_{W^{j,\infty}(\Omega)}\le\sqrt N\,B^{(j)}\,|u(t)|\qquad\text{for a.e. }t\in(0,T),
\end{equation}
where $|u(t)|$ denotes the Euclidean norm of $u(t)\in\mathbb R^N$. This yields, in particular, bounds for $\psi\in L^2(0,T;W^{3,\infty}(\Omega))$ and $\mathbf v:=\nabla^\perp\psi\in L^2(0,T;W^{2,\infty}(\Omega;\mathbb R^2))$ in terms of the $L^2(0,T;\mathbb R^N)$-norm of $u(t).$

By the linear independence of $\{h_k\}_{k=1}^N$, each $\psi\in\mathcal U_{\mathrm{ad}}^{R_u}$ is determined by a unique coefficient vector $u\in L^2(0,T;\mathbb R^N)$ with $\|u\|_{L^2(0,T;\mathbb R^N)}\le R_u$; we accordingly identify $\psi$ with $u$. The symmetric positive-definite Gram matrix
\begin{equation}\label{eq:gram}
M\in\mathbb R^{N\times N},\qquad M_{k\ell}:=(h_k,h_\ell)_{L^2(\Omega)},\qquad k,\ell=1,\dots,N,
\end{equation}
records the cross-correlations of the basis; a direct expansion yields, for every $u\in\mathbb R^N$ and $\psi=\sum_{k=1}^N u_k h_k$,
\begin{equation}\label{eq:gram_identity}
\|\psi\|_{L^2(\Omega)}^2=u^\top M u.
\end{equation}

For $\lambda>0$, we define the cost functional
\begin{equation}\label{eq:cost}
\mathcal{J}_{\lambda,\varepsilon}(\psi)
:=
-\int_0^1 \ell_\varepsilon\bigl(a(T,s)\bigr)\,ds
+\frac{\lambda}{2}\int_0^T \|\psi(t,\cdot)\|_{L^2(\Omega)}^2\,dt,
\qquad\psi\in\mathcal U_{\mathrm{ad}}^{R_u},
\end{equation}
where $a(T,s)$ {is obtained from the state system \eqref{eq:state} driven by $\psi(t,\cdot)$}.
By \eqref{eq:gram_identity}, the penalty admits the equivalent coefficient representation $$\frac{\lambda}{2}\int_0^T u(t)^\top M u(t)\,dt.$$
The optimal control or design problem we consider is
\begin{equation}\label{eq:ocp}
\min_{\psi\in\mathcal{U}_{\mathrm{ad}}^{R_u}}\mathcal{J}_{\lambda,\varepsilon}(\psi).
\end{equation}
The negative sign in the first term reflects that we \emph{maximize} the final interface length while penalizing the control effort.

The interface $\Gamma_t$ optimized in \eqref{eq:ocp} is not arbitrary: it coincides with the material image of $\Gamma_0$ under the Hamiltonian flow, which in turn coincides {at the continuous level} with the zero level set of the scalar field used in the full PDE validation of Section~\ref{sec:comparison}. We record this identification as a separate proposition because it ties the Lagrangian object that is optimized to the Eulerian object on which the resulting controls are subsequently evaluated.

\begin{prop}[Transport of the reference interface]\label{prop:zero_contour}
Let $\theta_0\in C^1(\overline\Omega)$ be such that $\Gamma_0:=\{x\in\overline\Omega:\theta_0(x)=0\}$ is a regular level set, and let $\theta\in {C([0,T];C^1(\overline\Omega))}$ solve the transport equation \eqref{eq:transport} {in the characteristic sense} for the velocity field $\mathbf v=\nabla^\perp\psi$ associated with some $\psi\in\mathcal U_{\mathrm{ad}}^{R_u}$. Then, for every $t\in[0,T]$,
\begin{equation*}
\Gamma_t:=\{x\in\overline\Omega:\theta(t,x)=0\}=\Phi_\psi(t;\Gamma_0),
\end{equation*}
where $\Phi_\psi(t;\cdot)$ is the Hamiltonian flow map of $\psi$. In particular, {when $\gamma_0$ parametrizes $\Gamma_0$,} the material interface optimized in \eqref{eq:ocp} is the zero contour of the {exact transported scalar whose discretization is used} in the full-PDE validation of Section~\ref{sec:comparison}.
\end{prop}

\begin{proof}
By the method of characteristics, $\theta(t,x)=\theta_0({\Phi_\psi(t;\cdot)^{-1}(x)})$ for $x\in\overline\Omega$, where $\Phi_\psi(t;\cdot)^{-1}$ exists by {uniqueness for the forward and reversed-time flows} and is the inverse of $\Phi_\psi(t;\cdot)$. The condition $\theta(t,x)=0$ is therefore equivalent to ${\Phi_\psi(t;\cdot)^{-1}(x)}\in\Gamma_0$, that is, $x\in\Phi_\psi(t;\Gamma_0)$.
\end{proof}

{This identity does not imply equality between a numerical scalar's zero contour and a separately discretized material curve. Both approximations require resolution checks.}

\begin{remark}[{Interface length and a critical Sobolev heuristic}]\label{rem:critical_mixnorm}
{To relate interface length to the scalar description of mixing, we first record an exact distributional identity for a two-phase configuration.} Let $\Gamma\subset\Omega$ be a smooth material curve and denote by {$\mu_\Gamma$ the restriction of the one-dimensional Hausdorff measure $\mathcal H^1$ to $\Gamma$}. If $E\subset\Omega$ is one of the two phases bounded by $\Gamma={\partial E\cap\Omega}$, then the distributional gradient of the characteristic function $\chi_E$ is the vector-valued Radon measure $D\chi_E=-n_\Gamma\,{\mu_\Gamma}$ {with $n_\Gamma$ pointing out of $E$}, and consequently $|D\chi_E|(\Omega)=\mathcal H^1(\Gamma)=\mathcal L(\Gamma)$; the interface length is exactly the total variation of the singular gradient of a two-phase scalar.

{
A separate Fourier calculation illustrates a possible connection between line length and a critical negative Sobolev quantity. This argument on $\mathbb R^2$ is formal. For the segment $\Gamma=[-L/2,L/2]\times\{0\}$, the Fourier transform of its line measure is
\[
\widehat{\mu_\Gamma}(\xi)=\int e^{-ix\cdot\xi}\,d\mu_\Gamma(x)=2\sin(\xi_1L/2)/\xi_1.
\]
The value at $\xi_1=0$ is $L$, and the transform is independent of the normal frequency $\xi_2$. Thus its decay does not suppress high normal frequencies, where the critical $\dot H^{-1/2}$ weight is proportional to $|\xi_2|^{-1}$. Integrating these normal frequencies up to the mollification cutoff $1/\delta$ suggests logarithmic growth and motivates a renormalized critical Sobolev interpretation of line length. It does not establish an asymptotic formula for a general curve or a homogeneous norm on a bounded domain.

}
{This interpretation does not imply equivalence with the scalar mix-norm \eqref{eq:mixnorm}.}
\end{remark}

We are now in {a} position to state the main well-posedness and existence result for \eqref{eq:ocp}.

\begin{thm}[Existence of a minimizer]\label{prop:H3}
Let $\mathcal H_N$, $B^{(j)}$, $\mathcal U_{\mathrm{ad}}^{R_u}$, and $\mathcal J_{\lambda,\varepsilon}$ be given by \eqref{eq:smooth_space}, \eqref{eq:basis_constants}, \eqref{eq:Uad}, and \eqref{eq:cost}, respectively. Assume that the basis $h_k$ satisfies \eqref{eq:basis_assumption} and that $\gamma_0\in W^{2,\infty}([0,1];\overline\Omega)$. Then the following properties hold.
\begin{enumerate}
\item[\textup{(i)}]\emph{(State regularity and well-posedness.)} For every $u\in L^2(0,T;\mathbb R^N)$ with associated stream function $\psi=\sum_{k=1}^N u_k h_k$,
\begin{equation}\label{eq:v_reg}
\psi\in L^2(0,T;W^{3,\infty}(\overline\Omega)),\qquad
\mathbf v=\nabla^\perp\psi\in L^2(0,T;W^{2,\infty}(\Omega;\mathbb R^2));
\end{equation}
the flow $\Phi_\psi$ maps $\overline\Omega$ into itself, and the state system  \eqref{eq:state} admits a unique solution $(X,a)\in C([0,T]\times[0,1];\overline\Omega\times\mathbb R^2)$. Consequently, the cost functional $\mathcal{J}_{\lambda,\varepsilon}$ is well defined and depends continuously on $u$, i.e., on  $\psi$.
\item[\textup{(ii)}]\emph{(Existence of a minimizer.)} The minimization problem \eqref{eq:ocp} admits at least one minimizer $\psi^*\in\mathcal U_{\mathrm{ad}}^{R_u}$.
\end{enumerate}
\end{thm}

\begin{proof}
{We prove well-posedness and then apply the direct method.}
\noindent\textit{Proof of \textup{(i)}.}

\vskip 5pt
\noindent\textit{Step 1. Regularity of the stream function and velocity field.}

By \eqref{eq:pt_bound} with $j=3$, we readily get \eqref{eq:v_reg}.

\vskip 5pt
\noindent\textit{Step 2. Flow invariance of $\overline\Omega$.}

{For a.e.\ $t$, the zero boundary trace gives $\mathbf v\cdot\mathbf n_\Omega=-\partial_\tau\psi=0$ on $\Gamma$, with $\tau=J\mathbf n_\Omega$. Since $\Gamma$ is $C^2$, its signed distance $d$ and nearest-point projection $\pi$ satisfy $\nabla d(x)=\mathbf n_\Omega(\pi(x))$ in a tubular neighborhood. Tangency and the spatial Lipschitz bound imply
\begin{equation}\label{eq:invariance_id}
|\nabla d(x)\cdot\mathbf v(t,x)|\le\sqrt N B^{(2)}|u(t)|\,|d(x)|\qquad\text{for a.e. }t.
\end{equation}
Gronwall along absolutely continuous trajectories prevents boundary crossing. The reversed-time argument gives a bijective flow on $\overline\Omega$.}

\vskip 5pt
\noindent\textit{Step 3. Existence and uniqueness of $(X,a)$.}

{Using the extensions, define} $F(t,X,a;u):=\bigl(J\nabla_x\psi(t,X),\,J\nabla_x^2\psi(t,X)\,a\bigr)$.
{It is measurable in $t$, with the integrable bound $|F(t,X,a;u)|\le\sqrt N(B^{(1)}+B^{(2)}A)|u(t)|$ on $|a|\le A$, and locally Lipschitz in $(X,a)$, with bound
\begin{equation}\label{eq:F_lip}
L_F(t,A):=2\sqrt N(B^{(2)}+B^{(3)})(1+A)|u(t)|\in L^1(0,T).
\end{equation}
Carath\'eodory's existence theorem \cite[Ch.~I, Thm.~5.1, p.~28]{hale1980ode} gives a local absolutely continuous solution; \eqref{eq:F_lip} and Gronwall give uniqueness. Step~2 bounds $X$, while the tangent equation and Gronwall give}
\begin{equation}\label{eq:Ma_u}
|a(t,s)|\le M_a(u):=\|\gamma_0'\|_{L^\infty([0,1])}\,\exp\!\Bigl(\sqrt N\,B^{(2)}\int_0^T|u(\tau)|\,d\tau\Bigr).
\end{equation}
{This bound and the integrable majorant above give continuation to $[0,T]$. Continuous dependence on the initial data $(\gamma_0(s),\gamma_0'(s))$ gives the asserted joint continuity.}

\vskip 5pt
\noindent\textit{Step 4. Length well-posedness.}

{Since $\ell_\varepsilon$ is $1$-Lipschitz, continuity of $a$ gives continuity of the length in time. For $\|u\|_{L^2},\|v\|_{L^2}\le R$, Gronwall applied to the state difference gives
\[
\|(X_u-X_v,a_u-a_v)\|_{C([0,T]\times[0,1])}\le C_R\|u-v\|_{L^2},
\]
with $C_R$ depending on the fixed data and $R$. The length and quadratic penalty are therefore continuous in $u$.}

\vskip 5pt
\noindent\textit{Proof of \textup{(ii)}.}

\vskip 5pt
\noindent\textit{Step 1. A uniform \emph{a priori} bound on the tangent.}

{For $\|u\|_{L^2}\le R_u$, Cauchy--Schwarz and \eqref{eq:Ma_u} give}
\begin{equation}\label{eq:Ma}
|a(t,s)|\le M_a:=\|\gamma_0'\|_{L^\infty([0,1])}\,\exp\,\bigl(\sqrt N\,B^{(2)}\,\sqrt T\,R_u\bigr).
\end{equation}
{Hence $\mathcal J_{\lambda,\varepsilon}\ge-\sqrt{M_a^2+\varepsilon^2}$ on $\mathcal U_{\mathrm{ad}}^{R_u}$.}

\vskip 5pt
\noindent\textit{Step 2. Weak compactness of a minimizing sequence.}

{A minimizing sequence has coefficients $u_m\rightharpoonup u^*$ along a subsequence, by weak compactness of the coefficient ball. Thus $\|u^*\|_{L^2}\le R_u$ and $\psi^*:=\sum_k u_k^*h_k\in\mathcal U_{\mathrm{ad}}^{R_u}$.}

\vskip 5pt
\noindent\textit{Step 3. Uniform convergence of the state.}

{To pass to the limit in the length term, we first establish compactness of the corresponding states $(X_m,a_m)$. Integrating \eqref{eq:state} and using Cauchy--Schwarz gives
\[
|X_m(t,s)-X_m(t',s)|+|a_m(t,s)-a_m(t',s)|
\le\sqrt N(B^{(1)}+B^{(2)}M_a)R_u\sqrt{|t-t'|}.
\]
For parameter differences, $\partial_sX_m=a_m$, and differentiating the tangent equation gives
\[
\partial_t(\partial_s a_m)=J\nabla_x^2\psi_m\,\partial_s a_m+J\nabla_x^3\psi_m[a_m,a_m],
\]
with spatial derivatives evaluated at $X_m$ and initial value $\gamma_0''$. Gronwall yields}
\begin{equation}\label{eq:Map}
|\partial_s a_m(t,s)|\le M_a':=\bigl(\|\gamma_0''\|_{L^\infty}+\sqrt N\,B^{(3)}\,M_a^2\,\sqrt T\,R_u\bigr)\,\exp\!\bigl(\sqrt N\,B^{(2)}\,\sqrt T\,R_u\bigr),
\end{equation}
{for a.e.\ $s$. These bounds give joint equicontinuity, so Arzel\`a--Ascoli supplies a uniformly convergent subsequence $(X_m,a_m)\to(\widetilde X,\widetilde a)$.
To identify this limit, fix $(t,s)$ and observe that the coefficients $J\nabla h_k(X_m(\cdot,s))$ and $J\nabla^2h_k(X_m(\cdot,s))a_m(\cdot,s)$ converge strongly in $L^2(0,t)$. Pairing them with $u_{k,m}\rightharpoonup u_k^*$ allows passage to the limit in both integral state equations. Uniqueness identifies the limit with the state $(X^*,a^*)$ of $\psi^*$.}

\vskip 5pt
\noindent\textit{Step 4. Lower semicontinuity and passage to the limit.}

{Uniform convergence of $a_m(T,\cdot)$ gives convergence of the length term. The positive quadratic penalty is weakly lower semicontinuous, hence}
\begin{align*}
\mathcal J_{\lambda,\varepsilon}(\psi^*)\le\liminf_m\mathcal J_{\lambda,\varepsilon}(\psi_m)=\inf_{\mathcal U_{\mathrm{ad}}^{R_u}}\mathcal J_{\lambda,\varepsilon},
\end{align*}
so $\psi^*$ is a minimizer of \eqref{eq:ocp}.
\end{proof}

{The following remarks explain the roles of the initial interface, regularization, and control constraints in this existence result.}

\begin{remark}[{Regularity of the initial interface and length regularization}]\label{rem:gamma0_reg}
{
The existence proof needs compactness of the transported interfaces, which explains the regularity assumed on the initial curve. The condition $\gamma_0\in W^{2,\infty}([0,1];\overline\Omega)$ is sufficient for this argument, but we do not claim that it is sharp. The bound on $\gamma_0'$ gives the uniform tangent estimate \eqref{eq:Ma_u}, while the bound on $\gamma_0''$ gives \eqref{eq:Map} and hence equicontinuity in $s$. Finite length alone only requires $\gamma_0'\in L^1$.

The role of length regularization is separate: it makes the terminal integrand smooth at a zero tangent. For a regular initial curve, this smoothing is not needed to differentiate the length, because the differential of the flow and its inverse give
\[
|a(t,s)|\ge |\gamma_0'(s)|\exp\!\left(-\sqrt N B^{(2)}\sqrt T R_u\right).
\]
Thus $\min_s|\gamma_0'(s)|>0$ keeps the tangent separated from zero and permits differentiation of the unregularized length. If $\gamma_0'(s)=0$, then $a(t,s)=0$ for every control, so this degeneracy alone does not force nonsmooth control dependence. We use $\varepsilon>0$ to obtain a smooth terminal integrand for all parametrizations.

To compare the regularized and unregularized optimization problems, we use the bound $0\le\ell_\varepsilon(z)-|z|\le\varepsilon$, which implies
\begin{equation}\label{eq:epsilon_uniform}
0\le\mathcal L_\varepsilon(\psi;T)-\mathcal L(\psi;T)\le\varepsilon.
\end{equation}
The estimate is uniform in $\psi$, and the existence proof also applies to $\varepsilon=0$. To pass from regularized minimizers to an unregularized one, compare each minimizer with a fixed admissible control and apply \eqref{eq:epsilon_uniform}. State convergence and lower semicontinuity then show that every weak cluster point of minimizers with $\varepsilon_m>0$ and $\varepsilon_m\to0$ minimizes the unregularized cost.
}
\end{remark}

\begin{remark}[On the role of control constraints]\label{rem:control_constraints}
{
Although the cost already penalizes control effort, the existence proof also uses the ball constraint in \eqref{eq:Uad} to obtain a weakly convergent minimizing sequence. The available length estimate does not show that the quadratic penalty alone can replace this constraint. Indeed, \eqref{eq:Ma_u} and $\ell_\varepsilon(a)\le |a|+\varepsilon$ give
\begin{equation}\label{eq:length_u_bound}
\mathcal L_\varepsilon(\psi;t)\le\|\gamma_0'\|_{L^\infty}\exp\!\bigl(\sqrt{NT}B^{(2)}\|u\|_{L^2}\bigr)+\varepsilon.
\end{equation}
This exponential upper bound, combined with the positive quadratic penalty, yields a lower bound on the cost that proves neither coercivity nor noncoercivity. Unconstrained existence therefore requires a separate analysis for the chosen basis. Throughout, $N$ is fixed and the constants may depend on the basis.

The numerical method retains this constraint: Algorithm~\ref{alg:optimize} enforces feasibility throughout the line search.
}
\end{remark}

\begin{remark}[Choice of the $L^2$ penalty]\label{rem:penalty_choice}
On $\mathcal H_N$, all norms are equivalent and the pointwise bound \eqref{eq:pt_bound} ensures that an $L^2$-in-time control on $u$ also controls $\psi$ in {the spatial norms used in the analysis}. The $L^2(\Omega)$ penalty on $\psi$ is therefore {a natural choice}: by \eqref{eq:gram_identity} it is equivalent to the coefficient-space penalty $\frac{\lambda}{2}\int_0^T u(t)^\top M u(t)\,dt$.
\end{remark}

{
\paragraph{Scope of the ansatz and geometry}\label{sec:scope}
The same proof indicates what would be needed to extend the control space or change the domain geometry. For the control space, finite dimensionality supplies spatial derivative bounds and finitely many weak-limit pairings. An infinite-dimensional extension would therefore need a replacement for these two ingredients, for example a reflexive spatial control space with sufficient regularity and continuous derivative evaluations along moving trajectories. An $L^2$ spatial bound alone does not supply these properties.

For the domain, flow invariance relies on boundary tangency rather than convexity. This allows nonconvex geometries, although narrow passages may hinder global transport despite local stretching. On multiply connected domains, zero trace on every boundary component is sufficient but restrictive; different constants on different components also enforce the no-penetration condition.

The $C^2$ boundary assumption enters through the signed-distance argument. The square used in the cellular experiments can instead be treated directly: for the cellular basis, the normal velocity vanishes on each edge and the corner velocities vanish. Uniqueness therefore prevents crossing, without using a smooth distance function at a corner; the remaining state estimates and compactness proof are unchanged. These arguments concern the continuous flow; numerical preservation of curved boundaries requires a separate check.
}

\subsection{Adjoint system and reduced gradient}\label{sec:adjoint}

Having established well-posedness and existence in Theorem~\ref{prop:H3}, we now derive the first-order optimality system for \eqref{eq:ocp}. To dualize the state equations \eqref{eq:state}, we introduce a pair of adjoint variables $(p,q)$ and first verify that they are uniquely determined by a backward linear ODE.

\begin{lem}[Adjoint system and well-posedness]\label{lem:adjoint_wp}
Under the hypotheses of Theorem~\ref{prop:H3}, {fix $u\in L^2(0,T;\mathbb R^N)$. The} backward linear system
\begin{subequations}\label{eq:adjoint}
\begin{align}
-\partial_t p(t,s)
&={\bigl(J\nabla_x^2\psi(t,X)\bigr)^\top} p(t,s)+\bigl(J\nabla_x^3\psi(t,X)[\,\cdot\,,a(t,s)]\bigr)^{\!\top}\!q(t,s),\label{eq:adjoint_p}\\
-\partial_t q(t,s)
&={\bigl(J\nabla_x^2\psi(t,X)\bigr)^\top} q(t,s),\label{eq:adjoint_q}\\
p(T,s)&=0,\qquad
q(T,s)=-\nabla\ell_\varepsilon(a(T,s))=-\frac{a(T,s)}{\sqrt{|a(T,s)|^2+\varepsilon^2}},\label{eq:adjoint_tc}
\end{align}
\end{subequations}
admits a unique solution $(p,q)\in C([0,T]\times[0,1];\mathbb R^2\times\mathbb R^2)$.
\end{lem}

\begin{proof}
{We solve the system backward in time using an integrable bound on its coefficient matrices. By \eqref{eq:pt_bound} and \eqref{eq:Ma_u}, uniformly in $s$,
\[
|J\nabla_x^2\psi(t,X)|+|J\nabla_x^3\psi(t,X)[\,\cdot\,,a]|
\le\sqrt N\bigl(B^{(2)}+B^{(3)}M_a(u)\bigr)|u(t)|\in L^1(0,T).
\]
The terminal data are continuous in $s$ and bounded since $|\nabla\ell_\varepsilon|\le1$. After reversing time, Carath\'eodory's theorem \cite[Ch.~I, Thm.~5.1, p.~28]{hale1980ode} and Gronwall give unique solutions $q$, then $p$, on $[0,T]$. Dominated convergence for the coefficient matrices, followed by Gronwall, gives continuous dependence on $s$ and hence joint continuity.}
\end{proof}

The terminal condition \eqref{eq:adjoint_tc} reflects that the cost depends on $X(T,\cdot)$ only through $a(T,\cdot)$ via the regularized length. With $(p,q)$ at hand, we can now identify the Fr\'echet derivative of the reduced cost.

\begin{prop}[Reduced gradient via the adjoint system]\label{prop:adjoint}
Under the hypotheses of Theorem~\ref{prop:H3}{, set} $j_{\lambda,\varepsilon}(u):=\mathcal J_{\lambda,\varepsilon}(\psi)$. {Defined by the same cost formula for all $L^2$ controls,} $j_{\lambda,\varepsilon}$ is Fr\'echet differentiable on $L^2(0,T;\mathbb R^N)$, and for every $\delta u\in L^2(0,T;\mathbb R^N)$ with $\eta:=\sum_{k=1}^N\delta u_k h_k$,
\begingroup\small
\begin{align}\label{eq:DJ}
Dj_{\lambda,\varepsilon}(u)[\delta u]
=\lambda\!\int_0^T\!\!(Mu(t))\cdot\delta u(t)\,dt
+\!\int_0^T\!\!\int_0^1\!\Bigl[p\cdot J\nabla_x\eta(t,X)+q\cdot J\nabla_x^2\eta(t,X)\,a\Bigr]ds\,dt,
\end{align}
\endgroup
or, componentwise, with $k=1,\dots,N$,
\begin{equation}\label{eq:grad_uk}
\frac{\partial j_{\lambda,\varepsilon}}{\partial u_k}(t)
=\lambda\,(Mu(t))_k+\int_0^1\!\Bigl[p(t,s)\cdot J\nabla h_k(X)+q(t,s)\cdot J\nabla_x^2 h_k(X)\,a(t,s)\Bigr]ds.
\end{equation}
\end{prop}

\begin{proof}
{We first differentiate the state map and then use the adjoint equations to eliminate the state variation from the cost derivative.}

\noindent\textit{Step 1. Linearized state.}

Since each $h_k\in {C^{3,1}(\overline\Omega)}$, the right-hand side of \eqref{eq:state} is ${C^{1,1}_{\mathrm{loc}}}$ in $(X,a)$ and affine in $u$. Gr\"onwall's inequality applied to the difference of two state trajectories, with the $L^2$-integrable Lipschitz constants furnished by the pointwise bound \eqref{eq:pt_bound} and {\eqref{eq:Ma_u}}, yields
\[
\|(X_{u+\delta u}-X_u,\,a_{u+\delta u}-a_u)\|_{C([0,T]\times[0,1])}\le C\,\|\delta u\|_{L^2(0,T;\mathbb R^N)}.
\]
{
To establish differentiability, we estimate the remainder after subtracting the linearized state. On a fixed control ball, write $z=(X,a)$ and $\Delta z=z_{u+\delta u}-z_u$. If $\delta z$ solves the linearized equations below, the $C^{1,1}$ spatial dependence and the bound $B^{(4)}$ yield, by Gronwall,
\[
\|\Delta z-\delta z\|_\infty\le C\bigl(\|u\|_{L^1}\|\Delta z\|_\infty^2
+\|\delta u\|_{L^1}\|\Delta z\|_\infty\bigr)
\le C\|\delta u\|_{L^2}^2.
\]
The constant depends on the ball radius, basis bounds, $T$, and $\gamma_0$.
} Hence $u\mapsto(X,a)$ is Fr\'echet differentiable from $L^2(0,T;\mathbb R^N)$ into $C([0,T]\times[0,1];\mathbb R^2\times\mathbb R^2)$, and its derivative $(\delta X,\delta a):=(DX[\delta u],Da[\delta u])$ solves
\begin{subequations}
\begin{align}
\partial_t\delta X &= J\nabla_x^2\psi(t,X)\,\delta X+\sum_{k=1}^N\delta u_k(t)\,J\nabla h_k(X),\label{eq:linearized_X}\\
\partial_t\delta a &= J\nabla_x^2\psi(t,X)\,\delta a+J\nabla_x^3\psi(t,X)[\delta X,a]+\sum_{k=1}^N\delta u_k(t)\,J\nabla_x^2 h_k(X)\,a,\label{eq:linearized_a}\\
\delta X(0,s)&=0,\qquad \delta a(0,s)=0.\label{eq:linearized_ic}
\end{align}
\end{subequations}

\vskip 5pt
\noindent\textit{Step 2. Lagrangian identity.}

Pair \eqref{eq:linearized_X} with $p$ and \eqref{eq:linearized_a} with $q$ in $L^2((0,T)\times(0,1))$ and integrate by parts in $t$; the boundary conditions \eqref{eq:linearized_ic} and $p(T,\cdot)=0$ of \eqref{eq:adjoint_tc} give
\begin{equation}\label{eq:IBP}
\begin{aligned}
\int_0^T\!\!\int_0^1 p\cdot\partial_t\delta X\,ds\,dt &= -\int_0^T\!\!\int_0^1 \partial_t p\cdot\delta X\,ds\,dt,\\
\int_0^T\!\!\int_0^1 q\cdot\partial_t\delta a\,ds\,dt &= \int_0^1 q(T,s)\cdot\delta a(T,s)\,ds-\int_0^T\!\!\int_0^1 \partial_t q\cdot\delta a\,ds\,dt.
\end{aligned}
\end{equation}
Substituting \eqref{eq:linearized_X}--\eqref{eq:linearized_a} on the left and \eqref{eq:adjoint_p}--\eqref{eq:adjoint_q} on the right of \eqref{eq:IBP}, the transpose identities
\[
p\cdot J\nabla_x^2\psi\,\delta X=(J\nabla_x^2\psi)^\top p\cdot\delta X,
\qquad
q\cdot J\nabla_x^3\psi[\delta X,a]=\bigl(J\nabla_x^3\psi[\,\cdot\,,a]\bigr)^\top q\cdot\delta X,
\]
and the analogous identity for $q\cdot J\nabla_x^2\psi\,\delta a$ pair up the bulk contributions and cancel them. What remains is
\begin{align}\label{eq:Lagrangian_id}
\int_0^1 q(T,s)\cdot\delta a(T,s)\,ds
=\int_0^T\!\!\int_0^1\sum_{k=1}^N\delta u_k(t)\Bigl[p\cdot J\nabla h_k(X)+q\cdot J\nabla_x^2 h_k(X)\,a\Bigr]ds\,dt.
\end{align}

\vskip 5pt
\noindent\textit{Step 3. Differentiating the cost.}

By \eqref{eq:cost} and \eqref{eq:gram_identity}, $$\mathcal J_{\lambda,\varepsilon}(\psi)=\frac\lambda2\int_0^T u^\top Mu\,dt-\int_0^1\ell_\varepsilon(a(T,s))\,ds.$$ Differentiating in $\delta u$ and using $q(T,s)=-\nabla\ell_\varepsilon(a(T,s))$,
\[
Dj_{\lambda,\varepsilon}(u)[\delta u]
=\lambda\!\int_0^T\!\!(Mu)\cdot\delta u\,dt+\int_0^1 q(T,s)\cdot\delta a(T,s)\,ds,
\]
which combined with \eqref{eq:Lagrangian_id} yields \eqref{eq:DJ}. Taking $\delta u_\ell(t)=\delta_{k\ell}\varphi(t)$ for arbitrary $\varphi\in L^2(0,T)$ specializes \eqref{eq:DJ} to \eqref{eq:grad_uk}. {This completes the proof of Proposition~\ref{prop:adjoint}.}
\end{proof}

Before turning to the numerical scheme, we record a consistency property linking the continuous gradient \eqref{eq:grad_uk} to its discrete counterpart introduced in Section~\ref{sec:numerics}.

\begin{remark}[Discrete gradient consistency]\label{rem:disc_cont_consistency}
The continuous penalty $$\frac\lambda 2\int_0^T u(t)^\top M u(t)\,dt,$$ equivalently $\frac\lambda 2\int_0^T\|\psi(t,\cdot)\|_{L^2(\Omega)}^2\,dt$ by \eqref{eq:gram_identity}, {is exactly $\frac{\lambda\Delta t}{2}\sum_nu_n^\top Mu_n$ for piecewise-constant controls. A diagonal coefficient penalty $\frac{\Delta t}{2}\sum_{n,k}\gamma_ku_{k,n}^2$ represents it only when $M$ is diagonal and $\gamma_k=\lambda M_{kk}$. In particular, $\gamma_k=\lambda$ corresponds to an orthonormal basis. Otherwise, the off-diagonal Gram-matrix entries must also be retained.} Correspondingly, {in the diagonal case,} the $\lambda\bigl(Mu(t)\bigr)_k$ contribution in \eqref{eq:grad_uk} discretizes to $\Delta t\,\gamma_k u_{k,n}$, matching \eqref{eq:grad_impl}. Algebraic exactness of this discrete gradient is established in Remark~\ref{rem:discrete_exact}.
\end{remark}

For use in the numerical implementation of Section~\ref{sec:numerics}, we record that on $\mathcal U_{\mathrm{ad}}^{R_u}$ the velocity field and its gradient read $\mathbf v_N(t,x)=\sum_{k=1}^N u_k(t)\,J\nabla h_k(x)$ and $\nabla_x\mathbf v_N(t,x)=\sum_{k=1}^N u_k(t)\,J\nabla_x^2 h_k(x)$.
The coupled system \eqref{eq:state}--\eqref{eq:adjoint}--\eqref{eq:grad_uk}{ must be supplemented by the ball constraint. Using the $L^2$ gradient of $j_{\lambda,\varepsilon}$ in \eqref{eq:grad_uk}, every local minimizer satisfies}
{
\begin{equation}\label{eq:optimality_vi}
\langle\nabla j_{\lambda,\varepsilon}(u^*),v-u^*\rangle_{L^2}\ge0
\qquad\text{for all }\|v\|_{L^2}\le R_u.
\end{equation}
This necessary condition follows by differentiating along admissible segments.
}

\section{Numerical framework for the optimal control problem}\label{sec:numerics}

We now describe the full numerical pipeline for solving the optimal control problem \eqref{eq:ocp}: the control parameterization, the symplectic time integrator, the discrete adjoint and gradient computation, and the optimization algorithm.

\subsection{Finite-dimensional ansatz and control parameterization}\label{sec:ansatz}

Following the framework of Section~\ref{sec:ocp}, the stream function is expanded in a finite-dimensional basis $\{h_k\}_{k=1}^N\subset C^{{3,1}}(\overline\Omega)\cap H_0^1(\Omega)$ as $\psi_N(t,x)=\sum_{k=1}^N u_k(t)\,h_k(x)$, where $u=(u_1,\dots,u_N)\allowbreak\in L^2(0,T;\mathbb R^N)$ {is the control vector}. The associated velocity fields $\mathbf v_k:=\nabla^\perp h_k=J\nabla h_k$ are divergence-free by construction. {For spatial evaluations in the midpoint equations, including those outside a nonconvex domain, we use the extensions specified after \eqref{eq:basis_assumption}.} The specific choice of basis is problem-dependent; in Section~\ref{sec:optimal_flows}, we employ the cellular-flow basis, and in Section~\ref{sec:optimal_Doswell}, we use a Doswell-type multi-vortex velocity field.

\subsection{Symplectic ODE solver for the interface}\label{sec:symplectic}

The initial material interface $\Sigma(0)$ is discretized by $N_p$ marker points $X_0^1,\dots,X_0^{N_p}$, distributed uniformly with respect to arc length along $\Sigma(0)$. {Write $X_0^j=\gamma_0(s_j)$, $0=s_1<\cdots<s_{N_p}=1$, and $\Delta s_j=s_{j+1}-s_j$. These material parameter values remain fixed during optimization; redistribution would change the forward map and its derivative.}

Since the flow map \eqref{eq:hamiltonian_flow} is Hamiltonian, we employ a symplectic time integrator \cite{HLW06} to preserve the volume-preserving structure {on a differentiable solution branch. This does not imply exact preservation of a general Hamiltonian; a time-dependent Hamiltonian need not be conserved even by the exact flow.} We use the \emph{implicit midpoint rule}, the simplest symplectic Runge--Kutta method {\cite[Theorem~VI.3.5, p.~190]{HLW06}}: for $\dot y=f(t,y)$,
\begin{equation}\label{eq:midpoint_def}
y_{n+1}=y_n+\Delta t\,f\!\left(\frac{t_n+t_{n+1}}{2},\,\frac{y_n+y_{n+1}}{2}\right).
\end{equation}
This method is second-order accurate {for smooth solutions}, A-stable, symmetric (self-adjoint), and symplectic. 

Applied to our setting, the marker update reads
\begin{equation}\label{eq:midpoint_marker}
X_{n+1}^j
=
X_n^j
+\Delta t\sum_{k=1}^N u_{k,n}\,
\mathbf v_k\!\left(\frac{X_n^j+X_{n+1}^j}{2}\right),
\end{equation}
where {$u_{k,n}$ denotes the constant coefficient on the $n$-th time interval,} and $\mathbf v_k:=J\nabla h_k$. Throughout, the controls $u_k$ are taken to be piecewise constant on each interval $[t_n,t_{n+1})$, so that the velocity field is time-independent within each step and the implicit midpoint rule~\eqref{eq:midpoint_def} applies directly. {To differentiate this implicit marker update, we work on an invertible solution branch. Marker containment is checked separately, since symplecticity does not imply preservation of a curved boundary.}

The terminal polyline length is computed using the same $\varepsilon$-regularization as in the continuous cost \eqref{eq:cost}:
\begin{equation}\label{eq:polyline_length}
\mathcal L_{h,\varepsilon}(X_{M})
:=
\sum_{j=1}^{N_p-1}{\Delta s_j}\,\ell_\varepsilon\!\left({\frac{X_{M}^{j+1}-X_{M}^j}{\Delta s_j}}\right),
\qquad
\ell_\varepsilon(z):=\sqrt{|z|^2+\varepsilon^2},
\end{equation}
where $M$ is the number of time steps {and $\Delta t=T/M$} and $X_{M}^j$ denotes the position of the $j$-th marker at the final time $t_{M}=T$. {The regularization on each chord is scaled by its parameter interval. Consequently, $0\le\mathcal L_{h,\varepsilon}-\mathcal L_{h,0}\le\varepsilon$, independently of $N_p$. The regularized continuous and discrete functionals use the same fixed parameterization; the regularization can depend on that choice.} The smoothing ensures differentiability at zero chord length.

{We next check consistency with the regularized continuous objective. For exact terminal samples $X_{M}^j=X(T,s_j)$, as $\max_j\Delta s_j\to0$}, the polyline length $\mathcal L_{h,\varepsilon}(X_{M})$ converges to the regularized continuous length $\int_0^1\ell_\varepsilon(\partial_sX(T,s))\,ds$ whenever the interface $s\mapsto X(T,s)$ is $W^{1,1}$-regular, and {its directional derivative along a smooth variation approximates the corresponding continuous first variation for fixed $\varepsilon>0$. Indeed, the chord quotients are cell averages of $X_s$, which converge in $L^1$, and $\ell_\varepsilon$ is $1$-Lipschitz. The first variation pairs $\nabla\ell_\varepsilon$ of these averages with the cell averages of the variation's derivative; boundedness and continuity of $\nabla\ell_\varepsilon$ give convergence}. For interfaces in $W^{2,\infty}([0,1];\overline\Omega)$ with uniformly spaced markers $s_j=(j-1)/(N_p-1)$, a standard piecewise-linear quadrature estimate yields the rate
\begin{equation}\label{eq:length_consistency}
\Bigl|\,\mathcal L_{h,\varepsilon}({X(T,s_j)})-\int_0^1\ell_\varepsilon\!\bigl(\partial_sX(T,s)\bigr)\,ds\,\Bigr|
\le
{\frac{1}{N_p-1}}\,\|\partial_s^2X(T,\cdot)\|_{L^\infty([0,1])},
\end{equation}
{by comparison with the cell averages. This is a spatial estimate for exact samples; it does not bound the time-discretization error. At a fixed resolution, optimization uses the fully discrete midpoint/polyline objective defined next.}

\subsection{Discrete adjoint and gradient computation}\label{sec:discrete_adjoint}

We minimize the fully discrete objective
\begin{equation}\label{eq:Jh_impl}
J_h(u)
:=
-\mathcal L_{h,\varepsilon}(X_{M})
+\frac{\Delta t}{2}\sum_{n=0}^{M-1}{u_n^\top Q u_n},
\end{equation}
where {$Q$ is symmetric positive definite and $u=\{u_{k,n}\}$ denotes the discrete controls. The stream-function penalty in \eqref{eq:cost} gives $Q=\lambda M$; componentwise weights give $Q=\operatorname{diag}(\gamma_1,\ldots,\gamma_N)$. Section~\ref{sec:comparison} specifies the common numerical penalty.}

{The implemented adjoint differentiates the marker/polyline objective \eqref{eq:Jh_impl} directly. It differs from the continuous $(X,a,p,q)$ system: tangent information is encoded in chord differences, so only marker-position adjoints are needed.} The terminal condition is the {negative} gradient of the regularized polyline length \eqref{eq:polyline_length}: letting $p_n^j\in\mathbb R^2$ denote the discrete adjoint variable at time step $n$ and marker $j$, we have $p_{M}^j=-\partial_{X_{M}^j}\mathcal L_{h,\varepsilon}(X_{M})$.
Defining the regularized unit tangent vectors
\[
\tau_{M}^j:=\frac{X_{M}^{j+1}-X_{M}^j}{\sqrt{|X_{M}^{j+1}-X_{M}^j|^2+{(\varepsilon\Delta s_j)^2}}},
\qquad j=1,\dots,N_p-1,
\]
the terminal adjoint values {for an open curve} are
\begin{equation}\label{eq:terminal_discrete}
p_{M}^1=\tau_{M}^1,
\qquad
p_{M}^{N_p}=-\tau_{M}^{N_p-1},
\qquad
p_{M}^j=\tau_{M}^j-\tau_{M}^{j-1}
\quad (2\le j\le N_p-1).
\end{equation}

The backward adjoint step is also midpoint-based. Writing
\begin{equation}\label{eq:Gmid_impl}
G_n^j
:=
\sum_{k=1}^N u_{k,n}\,\nabla \mathbf v_k\!\left(\frac{X_n^j+X_{n+1}^j}{2}\right),
\end{equation}
{we use the full transpose $(D\mathbf v_k)^\top=(J D^2h_k)^\top=D^2h_kJ^\top$.} {The} discrete adjoint update is
\begin{equation}\label{eq:adjoint_midpoint_impl}
p_n^j
=
p_{n+1}^j
+\Delta t\,(G_n^j)^\top\,\frac{p_n^j+p_{n+1}^j}{2}.
\end{equation}
{This is a linear $2\times2$ solve at each marker, provided the matrix $I-\Delta t(G_n^j)^\top/2$ is invertible.}

The gradient of the discrete objective with respect to the controls is then assembled as
\begin{equation}\label{eq:grad_impl}
\frac{\partial J_h}{\partial u_{k,n}}
=
\Delta t\,{(Q u_n)_k}
+\Delta t\sum_{j=1}^{N_p}
\frac{p_n^j+p_{n+1}^j}{2}\cdot
\mathbf v_k\!\left(\frac{X_n^j+X_{n+1}^j}{2}\right),
\qquad k=1,\dots,N.
\end{equation}
\begin{remark}[Discrete gradient exactness]\label{rem:discrete_exact}
{
To verify that \eqref{eq:adjoint_midpoint_impl} and \eqref{eq:grad_impl} give the derivative of the discrete objective, we differentiate the implicit forward step \eqref{eq:midpoint_marker}. Suppressing the marker index, set
\[
A_n=I-\frac{\Delta t}{2}G_n,\qquad B_n=I+\frac{\Delta t}{2}G_n.
\]
On a solution branch where $A_n$ is invertible, the step Jacobian $\Phi_n=\partial X_{n+1}/\partial X_n$ satisfies $A_n\Phi_n=B_n$. The two factors commute, so the chain rule gives
\[
p_n=\Phi_n^\top p_{n+1}=A_n^{-\top}B_n^\top p_{n+1},\qquad
A_n^{-\top}p_{n+1}=\frac{p_n+p_{n+1}}2.
\]
The first identity is exactly \eqref{eq:adjoint_midpoint_impl}. Differentiation with respect to the control gives
\[
A_n\partial_{u_{k,n}}X_{n+1}
=\Delta t\,\mathbf v_k\!\left(\frac{X_n+X_{n+1}}2\right).
\]
Pairing this identity with $A_n^{-\top}p_{n+1}$, summing over markers, and adding the penalty derivative yields \eqref{eq:grad_impl}. This algebra assumes exact implicit solves; numerical accuracy also depends on their conditioning and residuals.
}
\end{remark}

{The reported forward solves use FP64 Newton iteration with at most $30$ residual checks per time step. The global midpoint residual must be at most $2\times10^{-16}(1+\max_j\|X_n^j\|_\infty)$. Nonfinite residuals, failed determinant checks or exhausted iterations terminate the evaluation; the adjoint uses direct $2\times2$ solves. Section~\ref{sec:validation} reports directional-derivative checks of the assembled gradient.}

\subsection{Optimization algorithm}\label{sec:algorithm}

The control update is performed by a {projected-gradient method, with optional} Polak--Ribi\`ere (PR) conjugate-\allowbreak gradient (CG) {acceleration and} a \emph{projected} Armijo backtracking line search, combined with the closed-form projection onto the $L^2(0,T;\mathbb R^N)$-ball of radius $R_u$ that enforces the admissibility constraint $\|u\|_{L^2(0,T;\mathbb R^N)}\le R_u$ built into $\mathcal U_{\mathrm{ad}}^{R_u}$ in \eqref{eq:Uad}:
\begin{equation}\label{eq:projection}
\Pi_{B_{R_u}}(u):=\min\!\Bigl(1,\frac{R_u}{\|u\|_{L^2(0,T;\mathbb R^N)}}\Bigr)\,u\quad {(u\ne0),\qquad\Pi_{B_{R_u}}(0)=0}.
\end{equation}
{To express the projected update in the $L^2$ metric of the admissible ball, we use the uniform-grid inner product and gradient
\[
\langle v,w\rangle_h=\Delta t\sum_n v_n\cdot w_n,\qquad
\|v\|_h^2=\langle v,v\rangle_h,\qquad
(\nabla J_h(u))_{k,n}=\frac1{\Delta t}\frac{\partial J_h}{\partial u_{k,n}}.
\]
Thus $\nabla J_h$ is the discrete $L^2$ gradient, whereas \eqref{eq:grad_impl} gives the coordinate derivative.} The Armijo decrease is checked at the projected trial point {using the actual feasible displacement} so that every accepted iterate is admissible by construction; {the optional CG variant} reduces to the unconstrained PR-CG method whenever the ball is inactive along the line search.

{To test stationarity under the ball constraint, we use the discrete first-order condition and its equivalent projected residual,
\[
\langle \nabla J_h(u),v-u\rangle_h\ge0\quad(v\in B_{R_u}),\qquad
r_h(u):=\frac{\|u-\Pi_{B_{R_u}}(u-\nabla J_h(u))\|_h}{\max(1,\|u\|_h)}=0.
\]
Algorithm~\ref{alg:optimize} uses this residual for stopping; a small objective change alone does not certify stationarity.} The complete procedure is stated in Algorithm~\ref{alg:optimize}.

\begin{algorithm}[t]
\caption{Adjoint-based optimization of interface length}\label{alg:optimize}
\begin{algorithmic}[1]
    \State \textbf{Input:} Initial interface $\gamma_0$, basis functions $\{h_k\}_{k=1}^N$, parameters $T,\lambda>0$, {smoothing parameter $\varepsilon>0$,} regularization {matrix $Q$ (componentwise weights $\{\gamma_k\}_{k=1}^N$ when diagonal)}, ball radius $R_u>0$, number of markers $N_p$, time steps $M$, Armijo constant ${0<}c_1{<1}$, stopping tolerance $\varepsilon_{\mathrm{tol}}>0$.
    \State \textbf{Output:} {A feasible control and its stopping status (stationarity tolerance, budget limit, or line-search failure).}
    \State \textbf{Initialize:} Set initial controls {$u^{(0)}:=\Pi_{B_{R_u}}(u_{\rm init})$}, iteration counter $\ell:=0$.
    \State Compute $X_{\mathrm{traj}}^{(0)}$ by forward solve \eqref{eq:midpoint_marker} and $J_h^{(0)}:=J_h(u^{(0)})$.
    \State Compute discrete adjoint $\{p_n^{j,(0)}\}$ by backward solve \eqref{eq:adjoint_midpoint_impl}.
    \State Compute gradient ${\nabla J_h(u^{(0)})}$ via \eqref{eq:grad_impl} {divided by $\Delta t$}; set search direction $d^{(0)}:=-{\nabla J_h(u^{(0)})}$.
    \State {Return if $r_h(u^{(0)})\le\varepsilon_{\mathrm{tol}}$.}
    \Repeat
        \State \textbf{(1) Projected line search.} Find step size $\eta>0$ {by backtracking at $z=\Pi_{B_{R_u}}(u^{(\ell)}+\eta d^{(\ell)})$,} satisfying the Armijo condition evaluated at the projected iterate:
        \[
            J_h\!\bigl(\Pi_{B_{R_u}}(u^{(\ell)}+\eta\,d^{(\ell)})\bigr)\le J_h(u^{(\ell)})+c_1\,{\langle {\nabla J_h(u^{(\ell)})},z-u^{(\ell)}\rangle_h},
        \]
        where $\Pi_{B_{R_u}}$ is given by \eqref{eq:projection}. {Require $z\ne u^{(\ell)}$ and $\langle {\nabla J_h(u^{(\ell)})},z-u^{(\ell)}\rangle_h<0$. If an optional CG trial fails this descent check, or its backtracking limit is reached, restart with $d^{(\ell)}=-{\nabla J_h(u^{(\ell)})}$ and apply the same test.}
        \State {If a computational limit is reached, or the projected-gradient backtracking fails, return the current control with the corresponding unconverged status.}
        \State \textbf{(2) Update controls.} Set $u^{(\ell+1)}:=\Pi_{B_{R_u}}(u^{(\ell)}+\eta\,d^{(\ell)})$.
        \State \textbf{(3) Forward solve.} Compute $X_{\mathrm{traj}}^{(\ell+1)}$ and $J_h^{(\ell+1)}$.
        \State \textbf{(4) Adjoint solve.} Compute $\{p_n^{j,(\ell+1)}\}$ by backward solve.
        \State \textbf{(5) Gradient and CG update.}
        \begin{enumerate}[(i)]
            \item Compute ${\nabla J_h(u^{(\ell+1)})}$ via \eqref{eq:grad_impl} {divided by $\Delta t$}.
            \item {Set $\beta^{(\ell)}=0$ for projected descent; for optional CG, compute the} Polak--Ribi\`ere parameter:
            $\beta^{(\ell)}:=\max\!\Big(0,\;\frac{\langle {\nabla J_h(u^{(\ell+1)})},{\nabla J_h(u^{(\ell+1)})}-{\nabla J_h(u^{(\ell)})}\rangle_{{h}}}{\|{\nabla J_h(u^{(\ell)})}\|_{{h}}^2}\Big)$.
            \item Update search direction: $d^{(\ell+1)}:=-{\nabla J_h(u^{(\ell+1)})}+\beta^{(\ell)}d^{(\ell)}$.
            \item If $\langle {\nabla J_h(u^{(\ell+1)})},d^{(\ell+1)}\rangle_{{h}}\ge 0$, reset $d^{(\ell+1)}:=-{\nabla J_h(u^{(\ell+1)})}$.
        \end{enumerate}
        \State Set $\ell\leftarrow\ell+1$.
    \Until{${r_h(u^{(\ell)})\le}\varepsilon_{\mathrm{tol}}$}
\end{algorithmic}
\end{algorithm}

{In the reported runs, the initial trial step is bounded in the discrete $L^2$ norm and in each component. After the first accepted update, its $L^2$ bound also uses twice the norm of the previous accepted displacement, retained on restart. Backtracking halves the step at most $60$ times; the case-specific bounds are recorded in the accompanying source manifest.}

\section{Numerical experiments}\label{sec:numerics_exp}

In this section, we demonstrate the effectiveness of the optimal control approach and assess the proposed framework. We first analyze the mixing properties of two stationary velocity fields (Section~\ref{sec:stationary_flows}), then apply the optimization algorithm to {compute time-dependent stirring controls} for the cellular-flow (Section~\ref{sec:optimal_flows}) and Doswell frontogenesis (Section~\ref{sec:optimal_Doswell}) settings. A systematic comparison with the Eulerian Sobolev-norm optimizer of \cite{hu2026structure} is presented in Section~\ref{sec:comparison}. {We also compare historical two- and four-mode controls} (Section~\ref{sec:N4}). Throughout this section{, unless otherwise stated,} the smoothing parameter $\varepsilon$ in \eqref{eq:ell} is fixed at $10^{-8}$, {with the parameter-mesh scaling in \eqref{eq:polyline_length}}. The code is implemented in Python using PyTorch; the repository is {\url{https://github.com/DCN-FAU-AvH/Mixing-Hamiltonian}}. {The accompanying \href{https://github.com/DCN-FAU-AvH/Mixing-Hamiltonian/blob/013bcd85d8cf353bc7db4fdfd41a6c0dcd1a33f2/reproducibility/source_data_identity.json}{source manifest} identifies the archived optimization, transport evaluation and diagnostic implementations separately.}

Throughout this section, the initial scalar distribution for the Eulerian transport solve is
\begin{equation}\label{eq:theta_init_1}
\theta^0(x_1,x_2) := \tanh\!\left(\frac{x_2-0.5}{0.01}\right),
\qquad (x_1,x_2)\in\Omega,
\end{equation}
and the initial material interface is the zero contour $\Sigma(0)=\{(x_1,x_2)\in\Omega:x_2=0.5\}$. {The interface optimizations impose $R_u=10$. The control cost and the Sobolev comparison are specified in Section~\ref{sec:comparison}. All optimization runs use an NVIDIA A100 GPU.}

\subsection{Stationary flows}\label{sec:stationary_flows}

\subsubsection{Cellular flow}
Cellular flows~\cite{Fannjiang2006Quenching} are standard benchmark velocity fields in the study of fluid mixing. We consider the steady divergence-free velocity field on $\Omega=(0,1)^2$:
\begin{equation}\label{eq:cellular}
\mathbf{v}(x_1,x_2)=
\begin{bmatrix}
-\pi\sin(\pi x_1)\cos(\pi x_2)\\[0.3em]
\ \ \pi\cos(\pi x_1)\sin(\pi x_2)
\end{bmatrix},
\end{equation}
corresponding to the stream function $h(x_1,x_2)=\sin(\pi x_1)\sin(\pi x_2)$.

We simulate over $t\in[0,5]$ with Eulerian mesh ${\Delta x=1/500,\ \Delta t=1/1000}$ and $N_p=100{,}000$ Lagrangian markers. Figures~\ref{fig:interface_cellular}--\ref{fig:theta_cellular} show the interface and scalar field evolution. As quantified in Figure~\ref{fig:cellular_mixnorm-length}, the interface length grows approximately linearly, consistent with polynomial decay of $\|\theta_h\|_{\dot{H}^{-1}(\Omega)}$. {For related results on algebraic mixing in autonomous Hamiltonian flows and exponential mixing by time-dependent flows, see \cite{zelati2024mixing}.} In Sections~\ref{sec:optimal_flows}--\ref{sec:optimal_Doswell} {we investigate optimized \emph{time-dependent} controls as a route to faster decay}; the rates reported there are descriptive log-linear fits over the full simulated horizon, not claims of monotone exponential decay or asymptotic rates for arbitrary incompressible flows.

\begin{remark}[Local versus global sensitivity]\label{rem:local_global}
{In the stationary benchmarks,} the interface length begins to grow immediately, whereas the $\dot{H}^{-1}$ mix-norm remains nearly constant during an initial transient before decaying. {A possible interpretation is that early stretching changes the interface geometry while leaving the large-scale scalar distribution nearly unchanged. Negative Sobolev norms give greater weight to large scales \cite{mathew2005multiscale, thiffeault2012using}, so their evolution need not track interface growth.}
\end{remark}

\begin{figure}[!htbp]
  \centering
  \begin{subfigure}[b]{0.16\textwidth}
    \includegraphics[width=\textwidth]{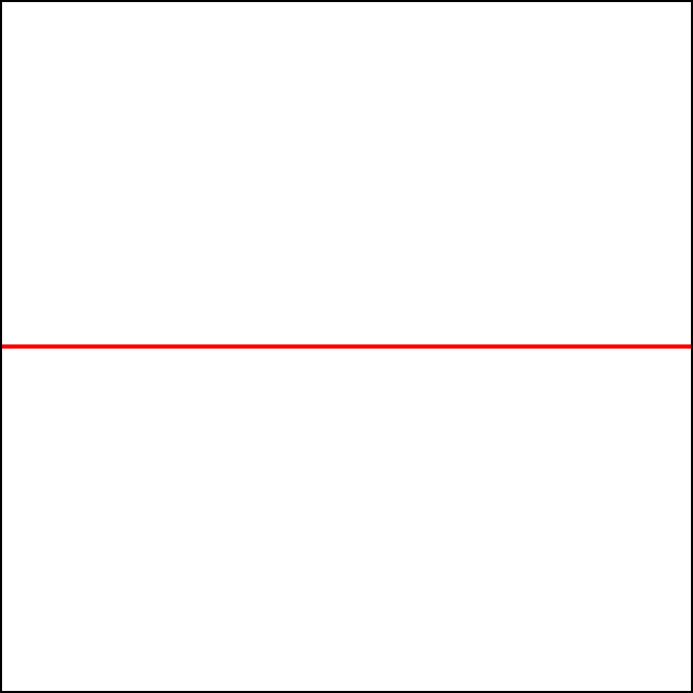}
    \caption{$t=0$}
  \end{subfigure}
  \begin{subfigure}[b]{0.16\textwidth}
    \includegraphics[width=\textwidth]{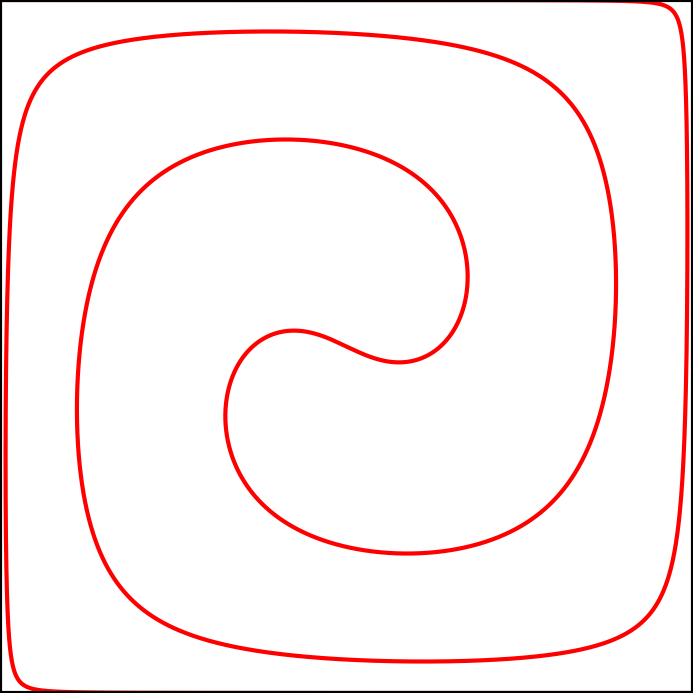}
    \caption{$t=1$}
  \end{subfigure}
  \begin{subfigure}[b]{0.16\textwidth}
    \includegraphics[width=\textwidth]{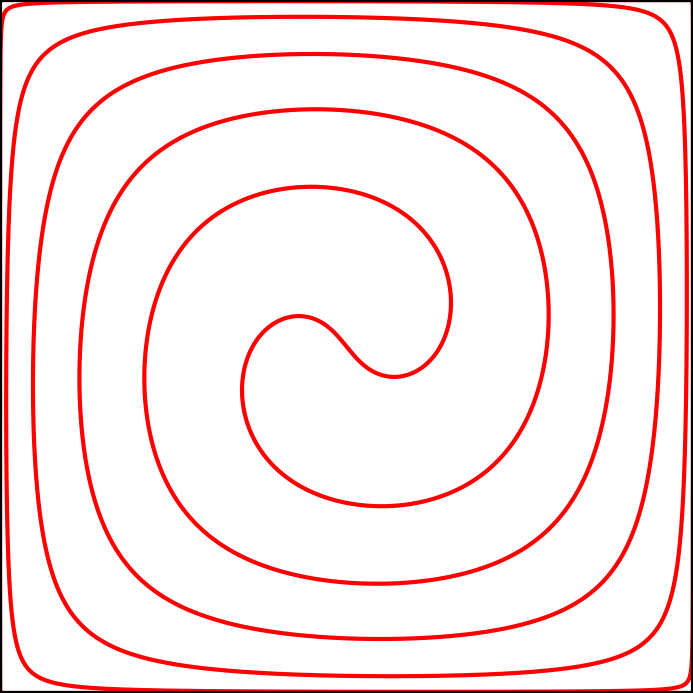}
    \caption{$t=2$}
  \end{subfigure}
  \begin{subfigure}[b]{0.16\textwidth}
    \includegraphics[width=\textwidth]{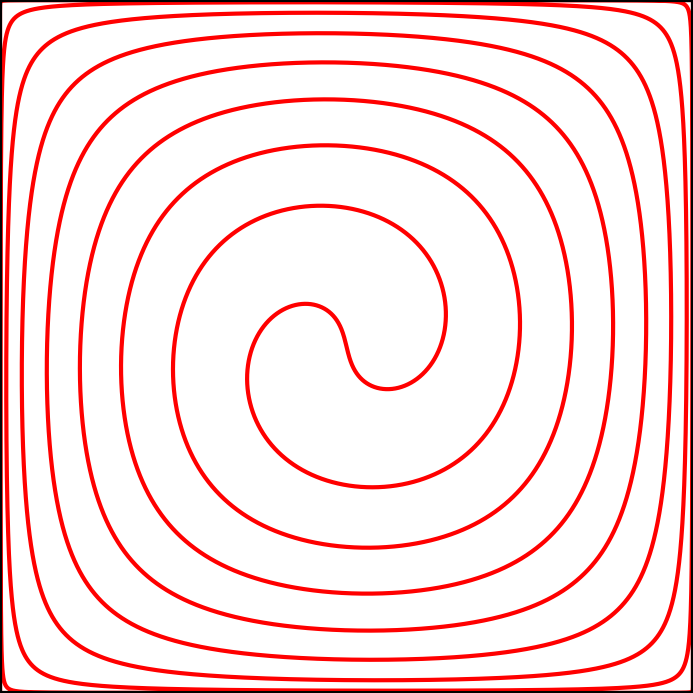}
    \caption{$t=3$}
  \end{subfigure}
  \begin{subfigure}[b]{0.16\textwidth}
    \includegraphics[width=\textwidth]{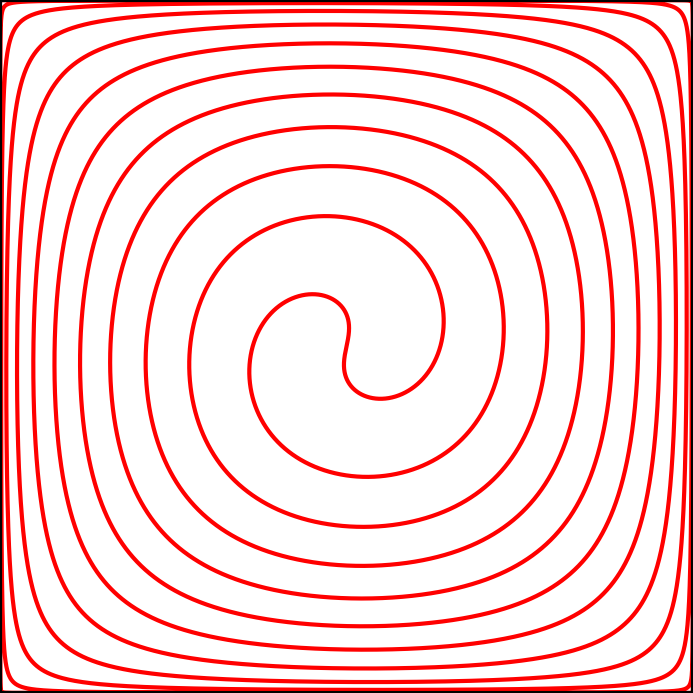}
    \caption{$t=4$}
  \end{subfigure}
  \begin{subfigure}[b]{0.16\textwidth}
    \includegraphics[width=\textwidth]{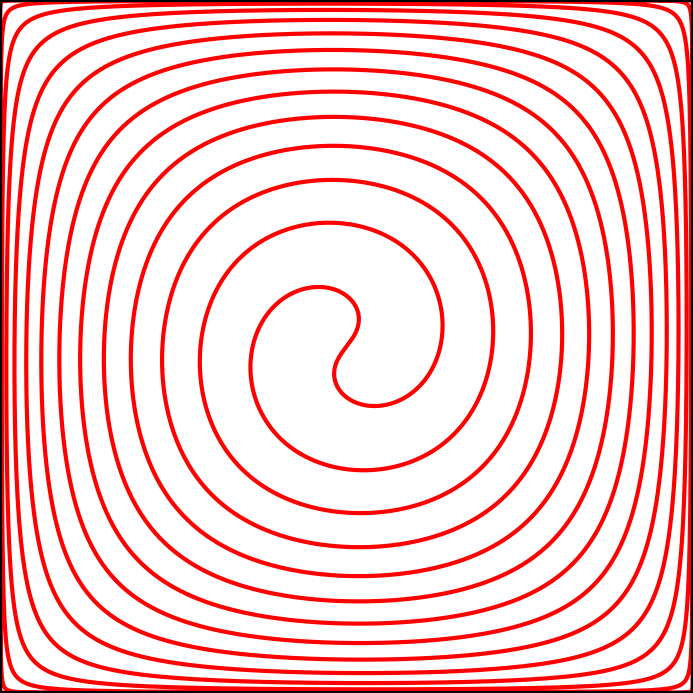}
    \caption{$t=5$}
  \end{subfigure}
  \caption{Evolution of the interface under the steady cellular flow \eqref{eq:cellular} for $t\in[0,5]$.}
  \label{fig:interface_cellular}
\end{figure}

\begin{figure}[!htbp]
  \centering
  \begin{subfigure}[b]{0.16\textwidth}
    \includegraphics[width=\textwidth]{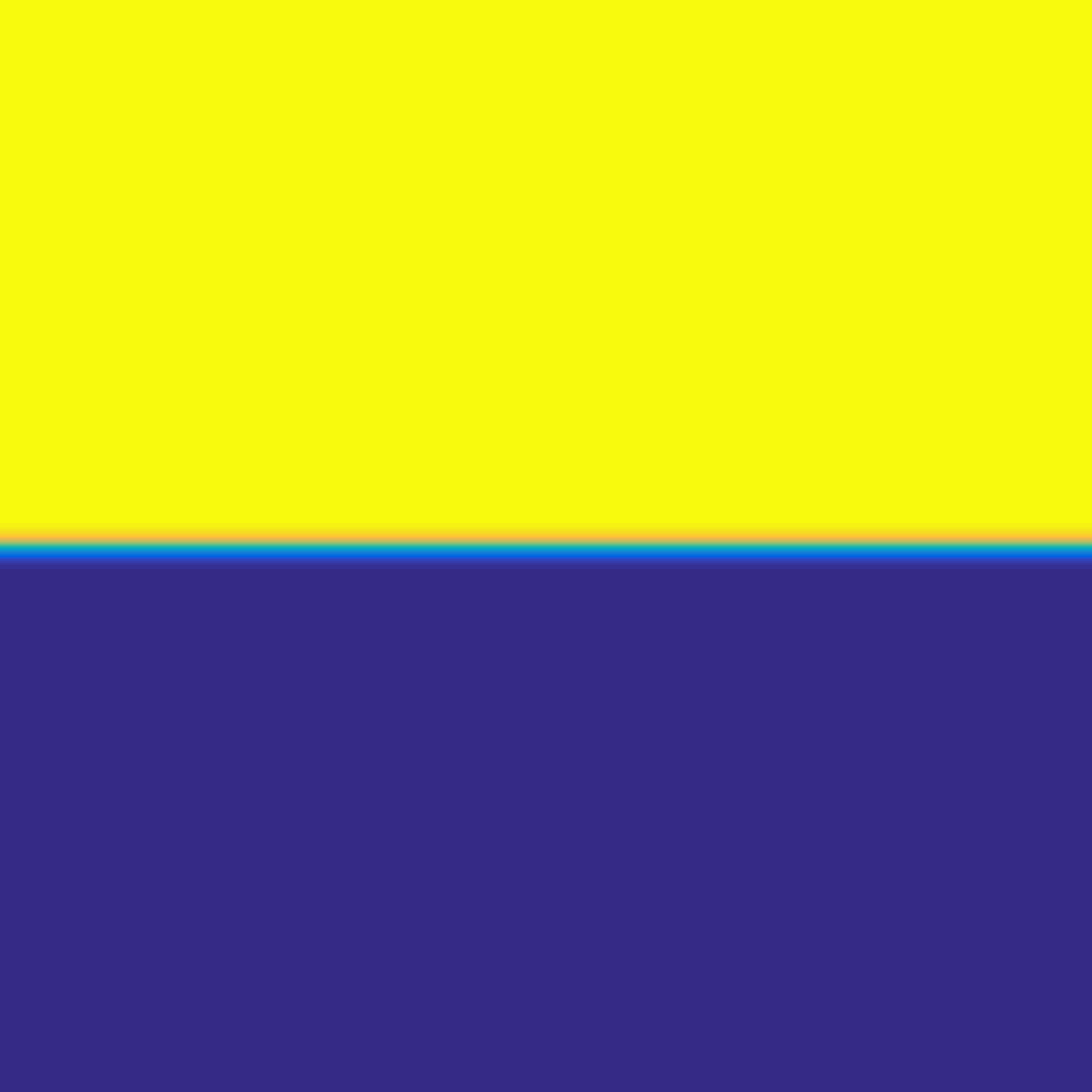}
    \caption{$t=0$}
  \end{subfigure}
  \begin{subfigure}[b]{0.16\textwidth}
    \includegraphics[width=\textwidth]{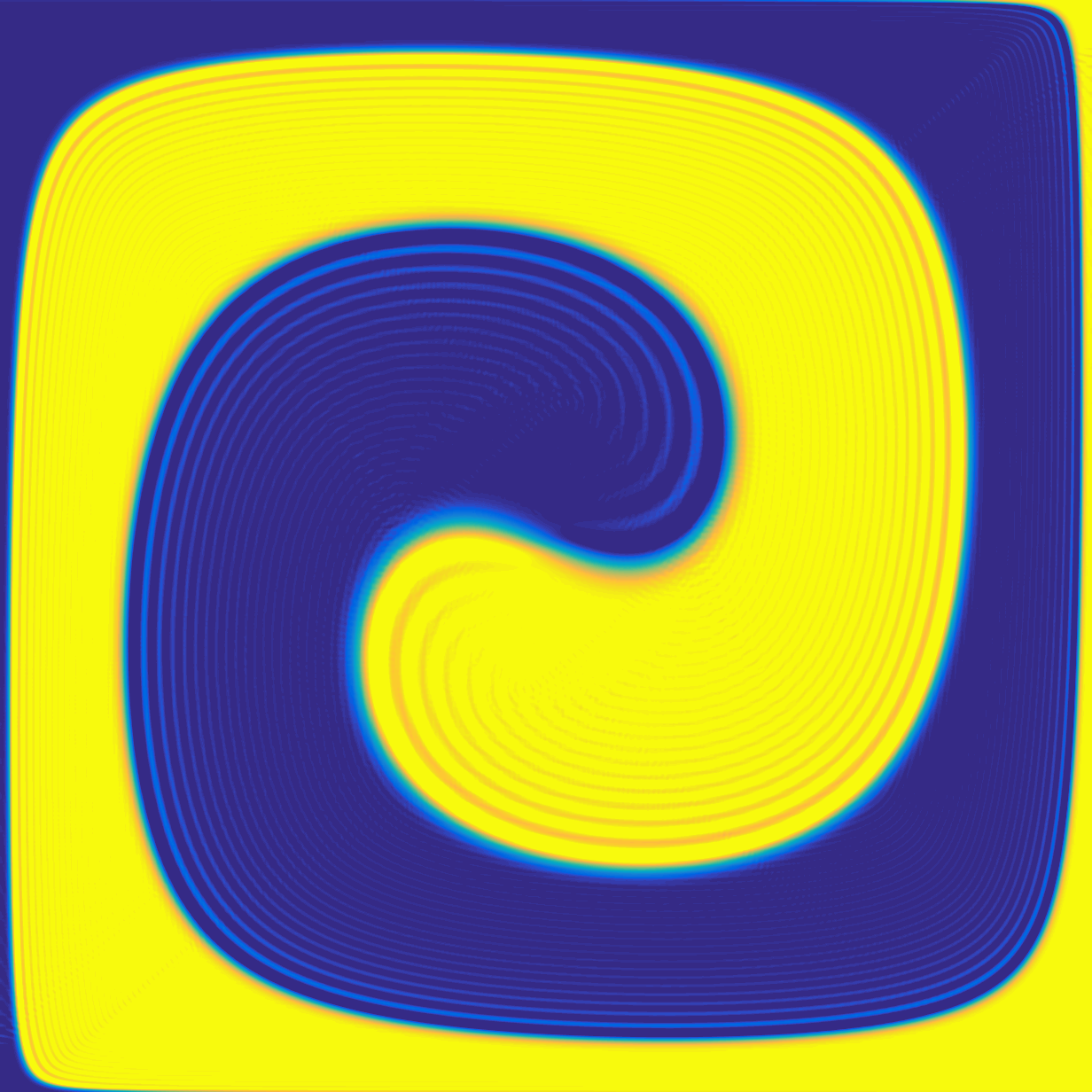}
    \caption{$t=1$}
  \end{subfigure}
  \begin{subfigure}[b]{0.16\textwidth}
    \includegraphics[width=\textwidth]{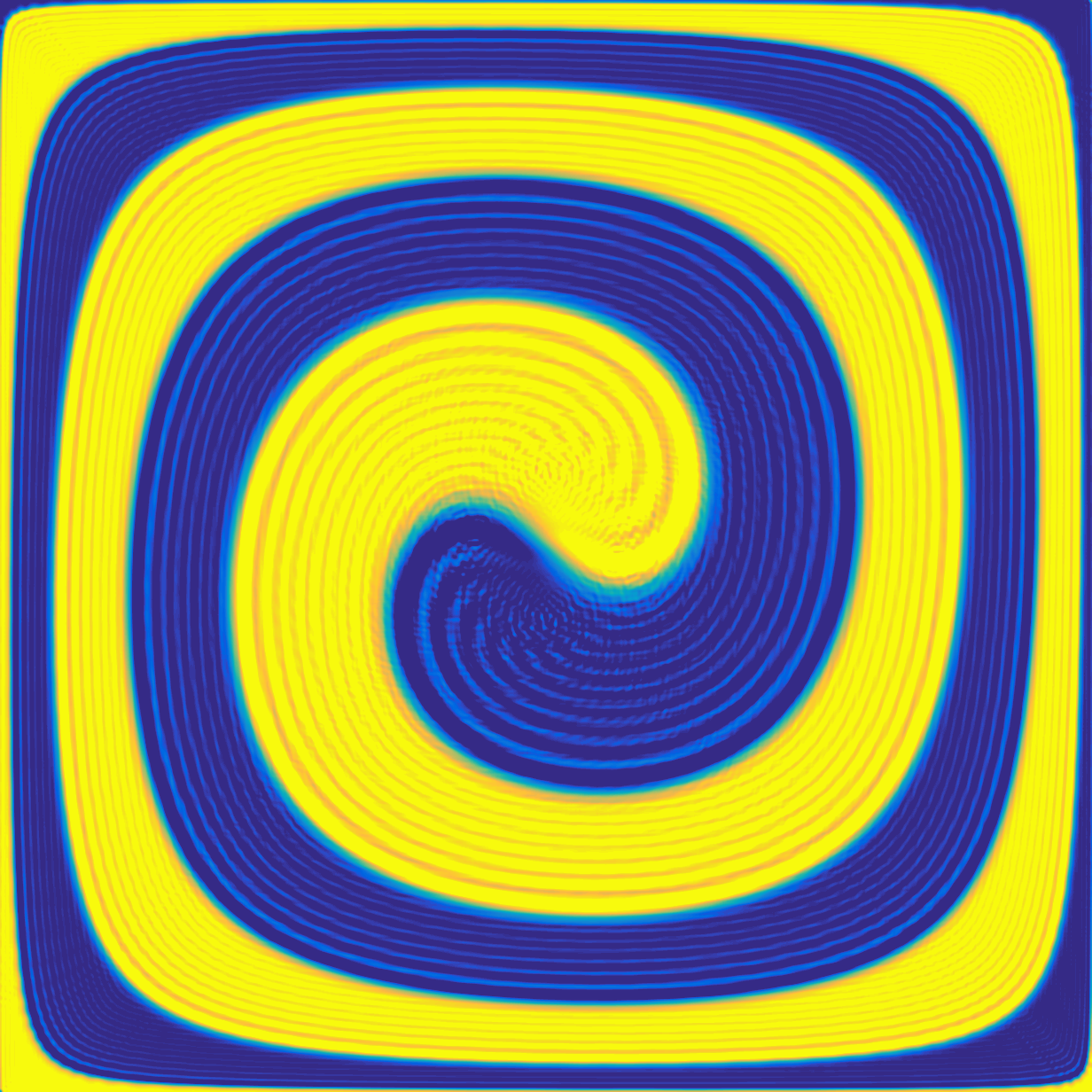}
    \caption{$t=2$}
  \end{subfigure}
  \begin{subfigure}[b]{0.16\textwidth}
    \includegraphics[width=\textwidth]{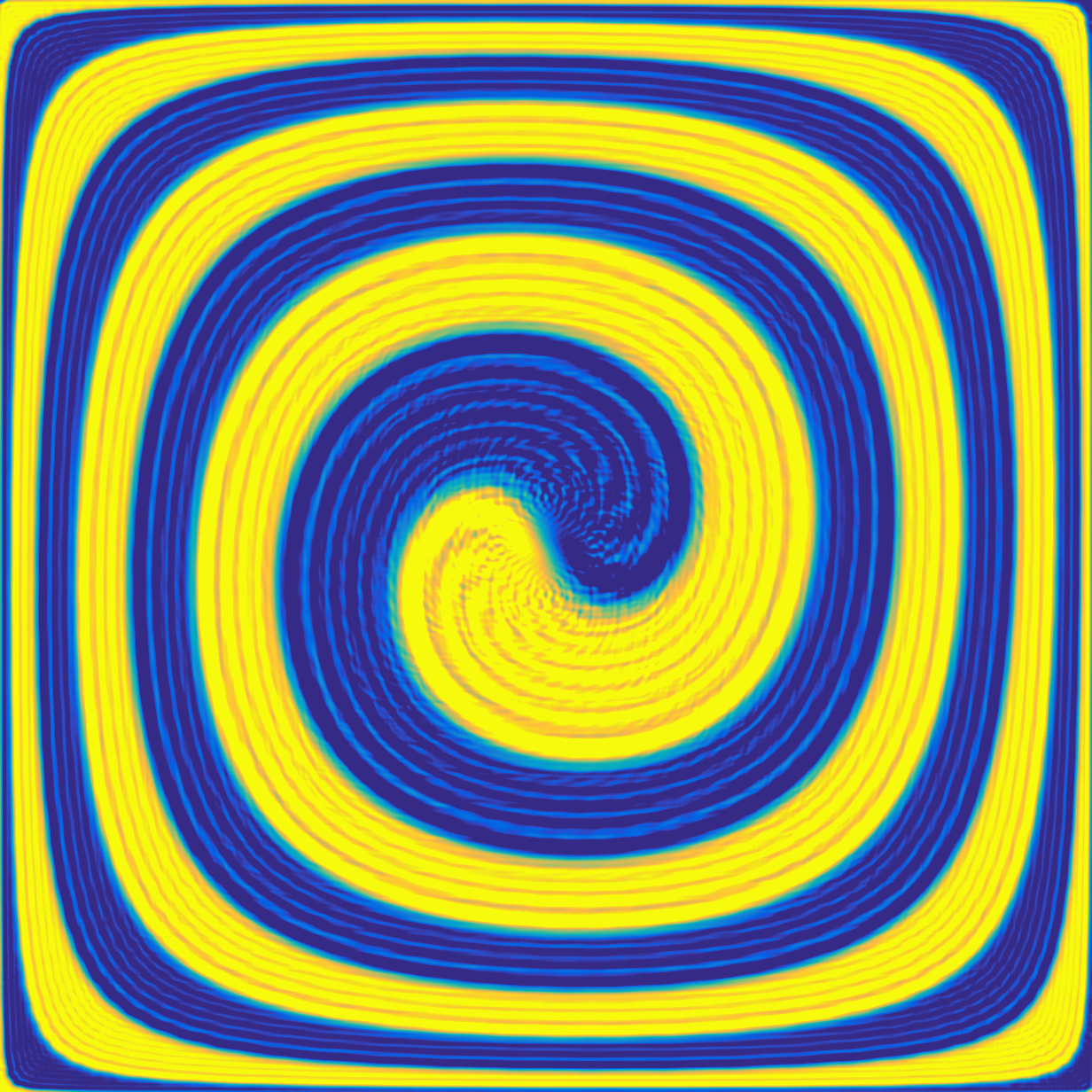}
    \caption{$t=3$}
  \end{subfigure}
  \begin{subfigure}[b]{0.16\textwidth}
    \includegraphics[width=\textwidth]{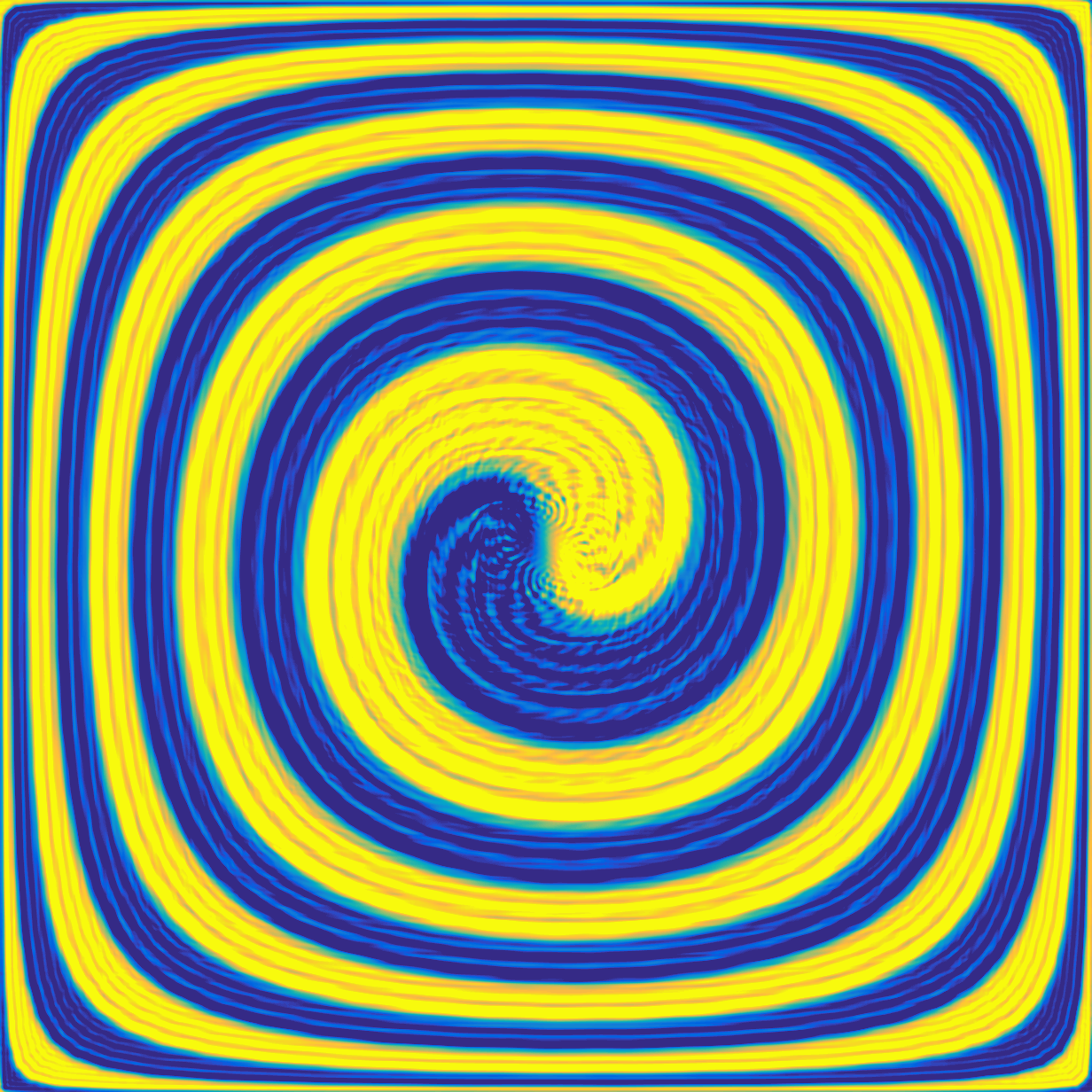}
    \caption{$t=4$}
  \end{subfigure}
  \begin{subfigure}[b]{0.16\textwidth}
    \includegraphics[width=\textwidth]{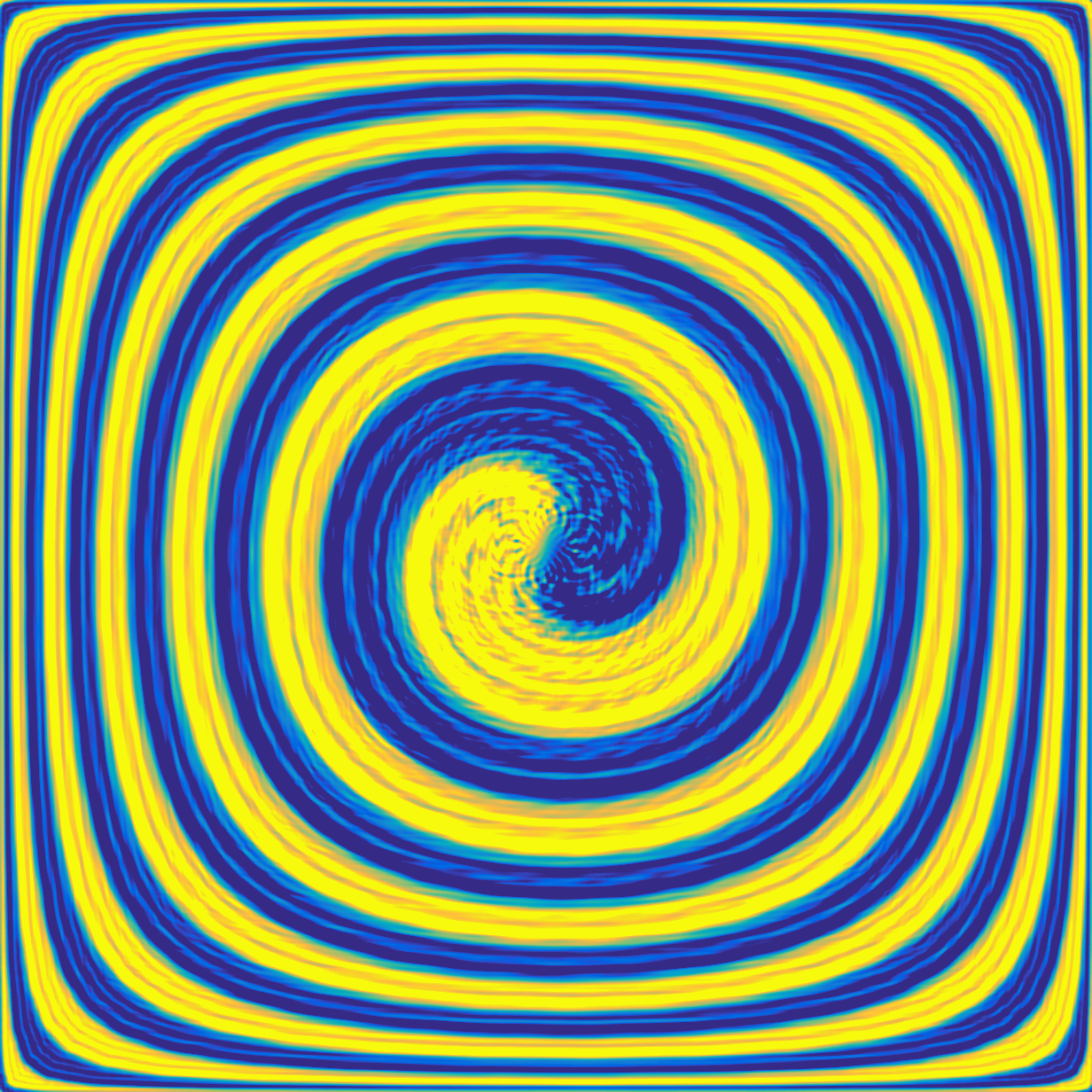}
    \caption{$t=5$}
  \end{subfigure}
  \caption{Evolution of $\theta_h$ under the steady cellular flow for $t\in[0,5]$.}
  \label{fig:theta_cellular}
\end{figure}

\begin{figure}[!htbp]
  \centering
  \begin{subfigure}[b]{0.475\textwidth}
    \centering
    \includegraphics[height=3cm]{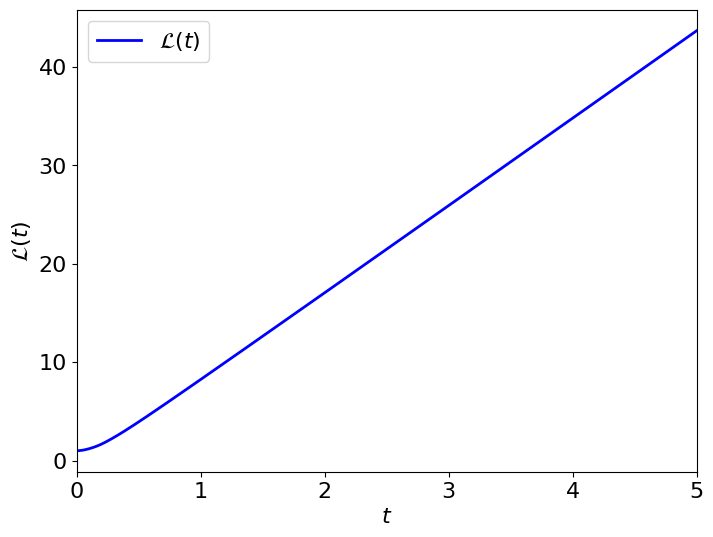}
    \caption{Interface length $\mathcal{L}(t)$ (linear growth).}
    \label{fig:cellular_length}
  \end{subfigure}
   \begin{subfigure}[b]{0.49\textwidth}
    \centering
    \includegraphics[height=3cm]{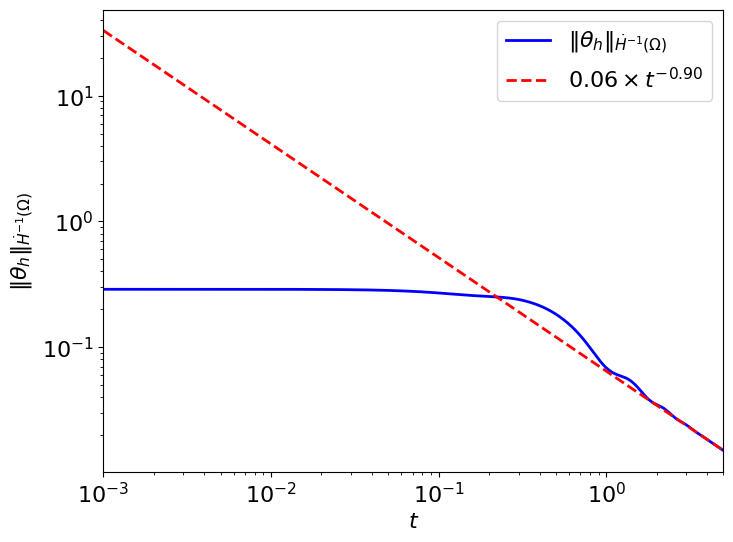}
    \caption{$\dot{H}^{-1}$ mix-norm (polynomial decay).}
    \label{fig:cellular_mixnorm}
  \end{subfigure}
    \caption{Interface growth and mix-norm decay under the steady cellular flow.}
  \label{fig:cellular_mixnorm-length}
\end{figure}

\subsubsection{Doswell frontogenesis} \label{sec:Doswell-stable}
To verify our observations in a setting with circular geometry, we consider the Doswell frontogenesis model \cite[p.~1243, eqs.~(9)--(10)]{doswell1984kinematic}. The domain is the disc $\Omega=\{\mathbf{x}\in\mathbb{R}^2:|\mathbf{x}-\mathbf{x}_c|<R\}$ with center $\mathbf{x}_c=(0.5,0.5)$ and radius $R=0.5$, so that the initial interface $\Sigma(0)=\{x_2=0.5\}$ passes through the disc center. The velocity field is
\begin{equation}
\label{eq:v_Doswell}
\mathbf{v}(x_1,x_2)=
 \begin{bmatrix}
      -(x_2-0.5)\,g(r)          \\[0.3em]
        \phantom{-}(x_1-0.5)\,g(r)
     \end{bmatrix},
\end{equation}
where $r:=\sqrt{(x_1-0.5)^2+(x_2-0.5)^2}$, $\overline{v}=2.59807$, and the angular velocity profile is $g(r):=\frac{\overline{v}}{r}\,\operatorname{sech}^2(r)\,\tanh(r)$, {with the continuous extension $g(0):=\overline v$}.
The velocity field \eqref{eq:v_Doswell} is divergence-free and tangent to $\partial\Omega$. The transport equation is discretized on a polar grid ($N_r=N_\varphi={500}$) over $t\in[0,10]$ with $\Delta t=1/1000$ and $N_p=100{,}000$ markers. As shown in Figures~\ref{fig:interface_Doswell}--\ref{fig:Doswell_mixnorm-length}, the same qualitative behavior is observed: linear interface growth and mix-norm decay.

\begin{figure}[!htbp]
  \centering
  \begin{subfigure}[b]{0.16\textwidth}
    \includegraphics[width=\textwidth]{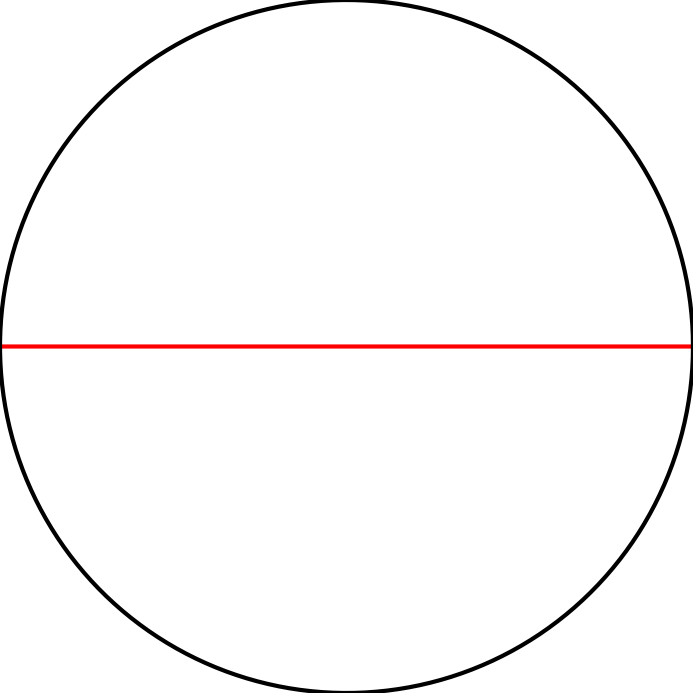}
    \caption{$t=0$}
  \end{subfigure}
  \begin{subfigure}[b]{0.16\textwidth}
    \includegraphics[width=\textwidth]{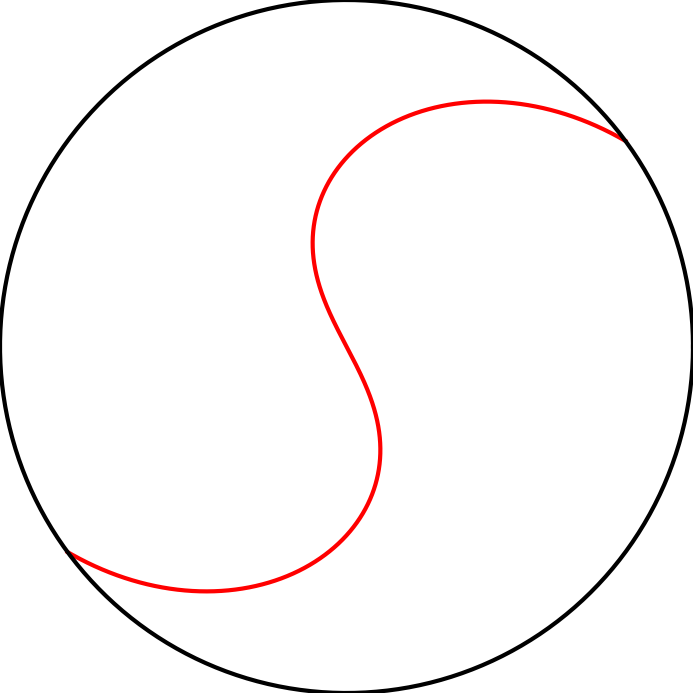}
    \caption{$t=2$}
  \end{subfigure}
  \begin{subfigure}[b]{0.16\textwidth}
    \includegraphics[width=\textwidth]{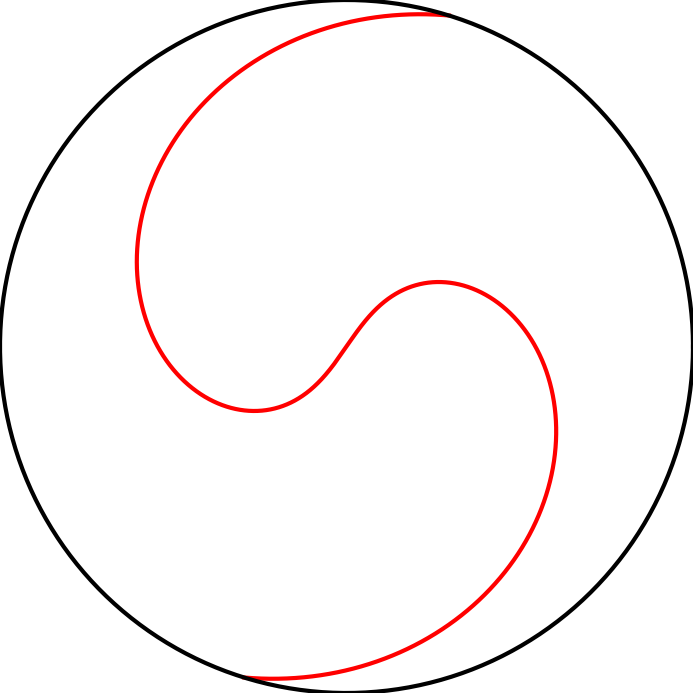}
    \caption{$t=4$}
  \end{subfigure}
  \begin{subfigure}[b]{0.16\textwidth}
    \includegraphics[width=\textwidth]{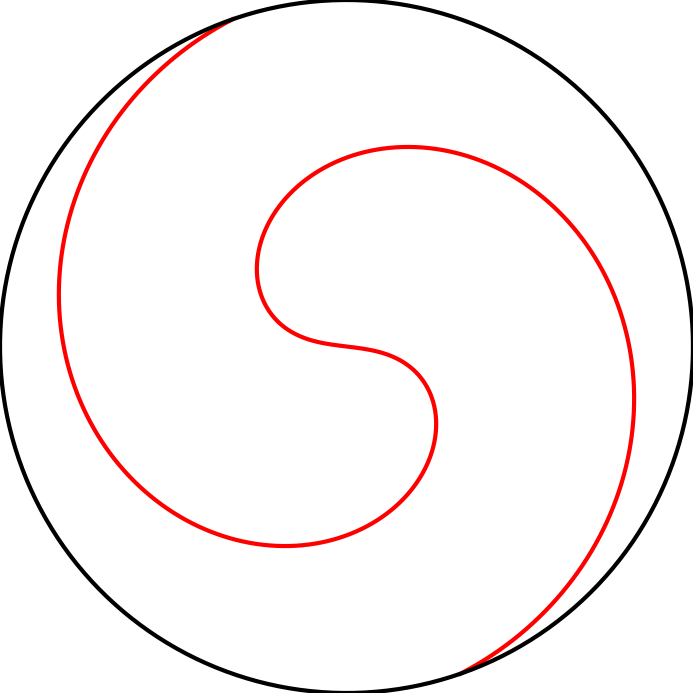}
    \caption{$t=6$}
  \end{subfigure}
  \begin{subfigure}[b]{0.16\textwidth}
    \includegraphics[width=\textwidth]{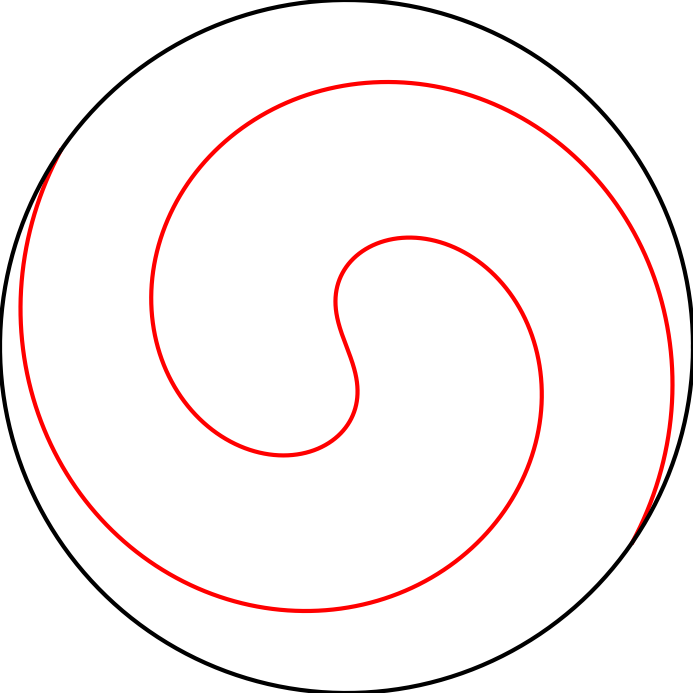}
    \caption{$t=8$}
  \end{subfigure}
  \begin{subfigure}[b]{0.16\textwidth}
    \includegraphics[width=\textwidth]{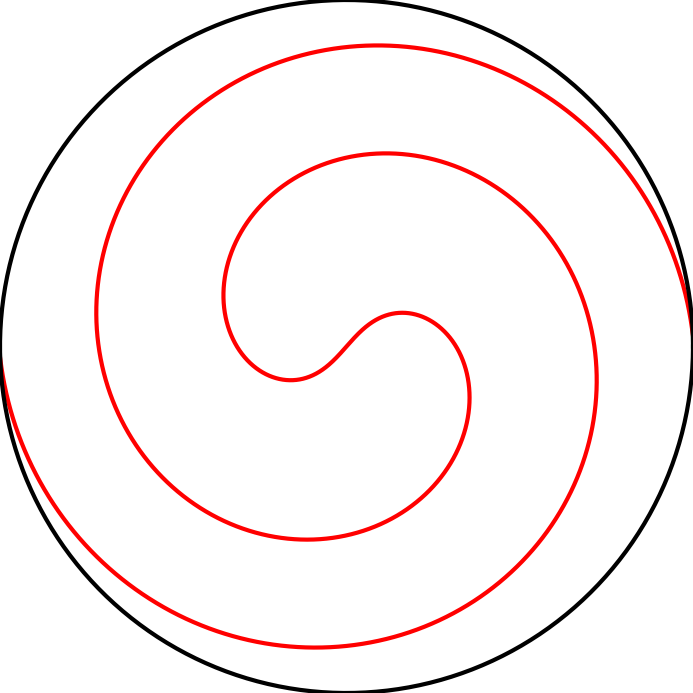}
    \caption{$t=10$}
  \end{subfigure}
  \caption{Evolution of the interface under Doswell frontogenesis for $t\in[0,10]$.}
  \label{fig:interface_Doswell}
\end{figure}

\begin{figure}[!htbp]
  \centering
  \begin{subfigure}[b]{0.16\textwidth}
    \includegraphics[width=\textwidth]{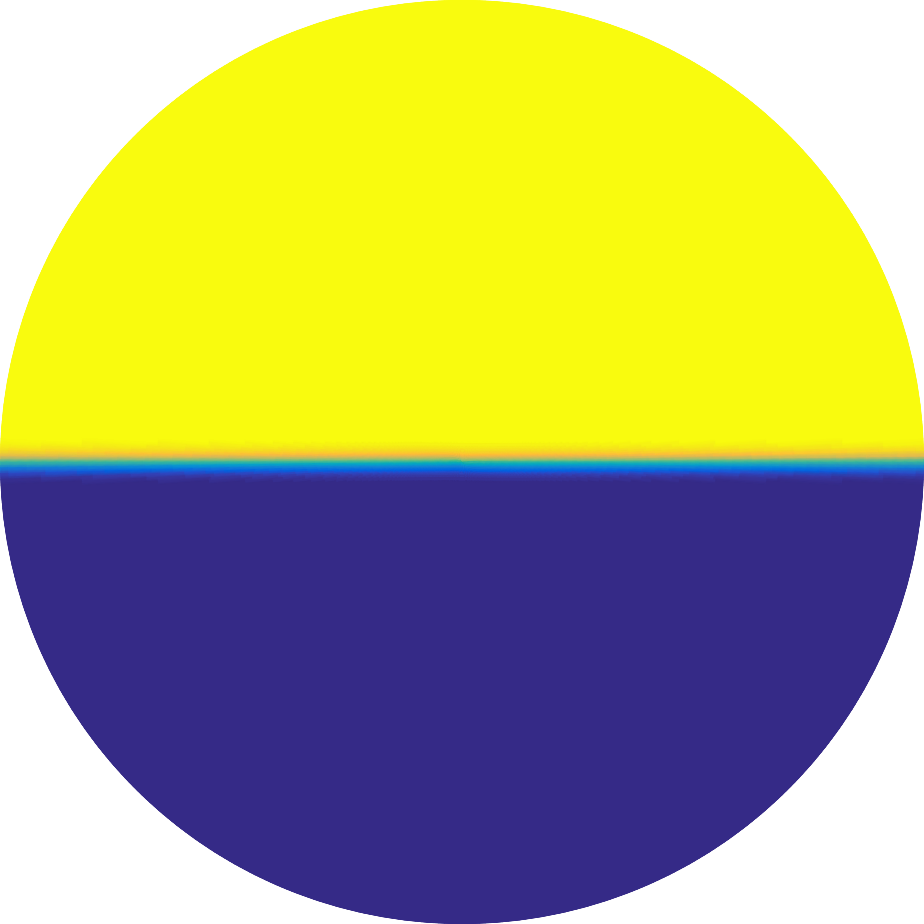}
    \caption{$t=0$}
  \end{subfigure}
  \begin{subfigure}[b]{0.16\textwidth}
    \includegraphics[width=\textwidth]{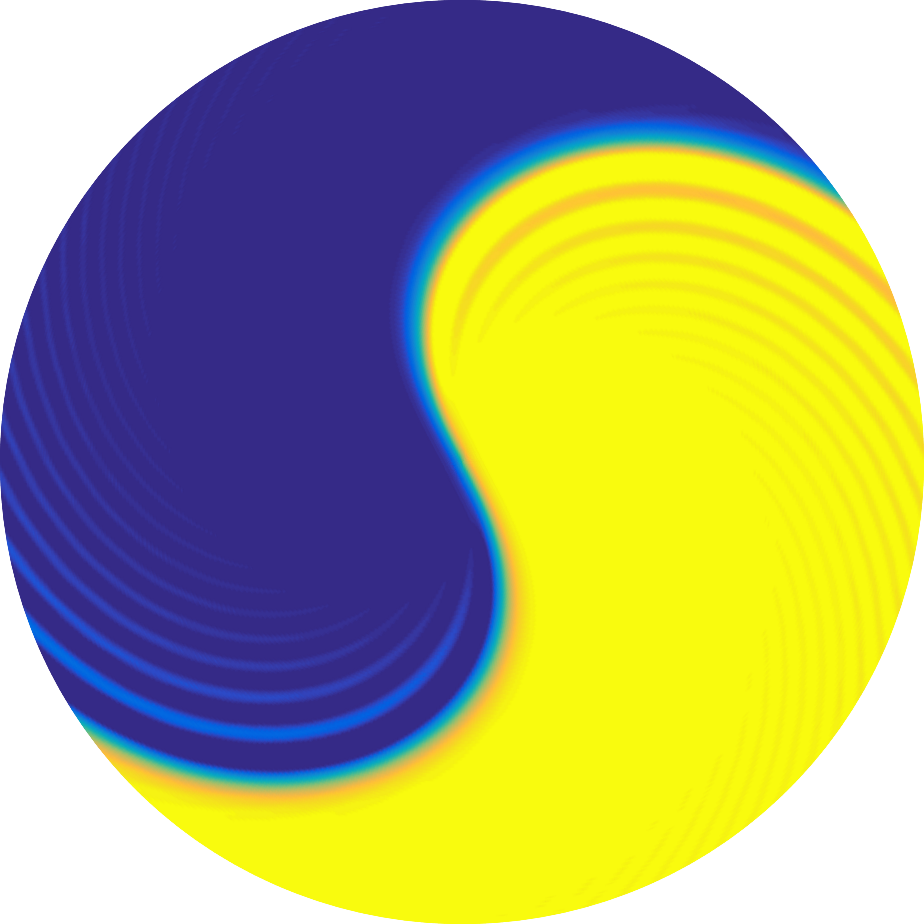}
    \caption{$t=2$}
  \end{subfigure}
  \begin{subfigure}[b]{0.16\textwidth}
    \includegraphics[width=\textwidth]{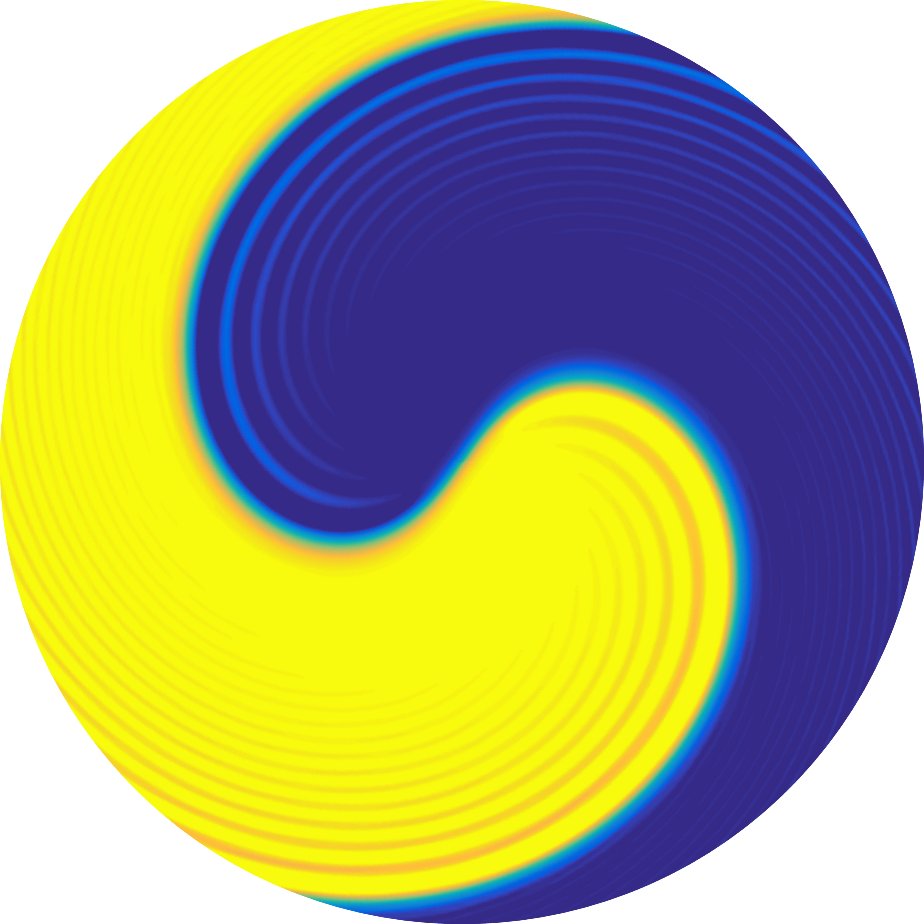}
    \caption{$t=4$}
  \end{subfigure}
  \begin{subfigure}[b]{0.16\textwidth}
    \includegraphics[width=\textwidth]{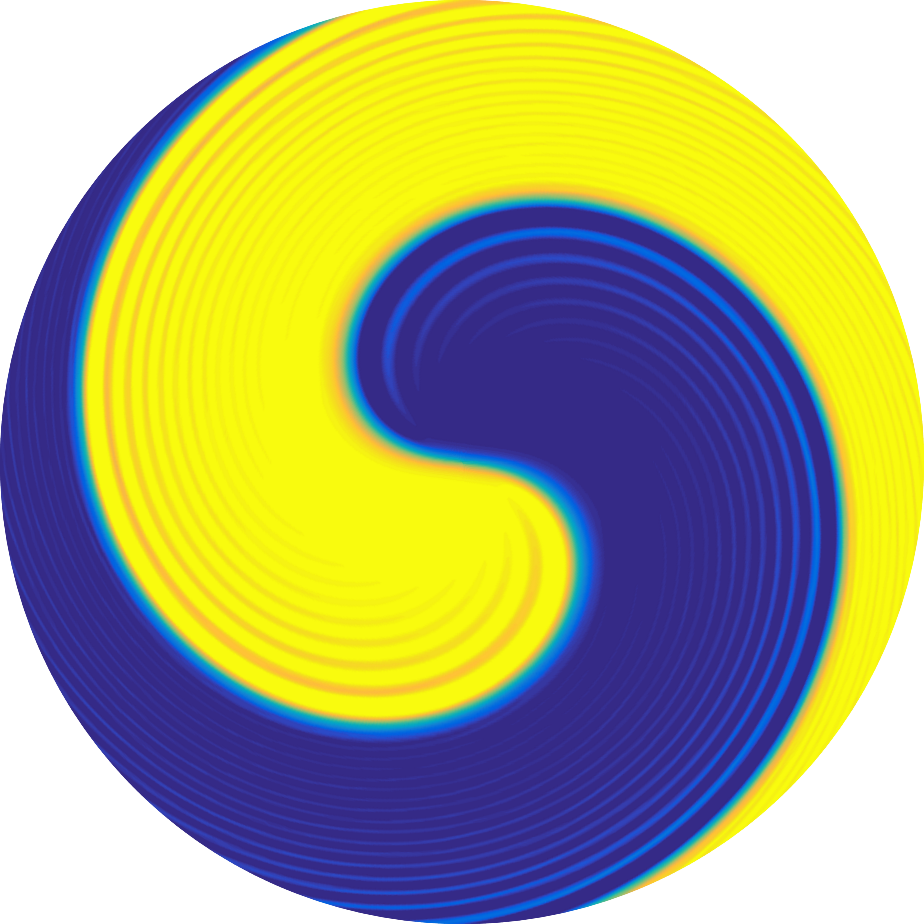}
    \caption{$t=6$}
  \end{subfigure}
  \begin{subfigure}[b]{0.16\textwidth}
    \includegraphics[width=\textwidth]{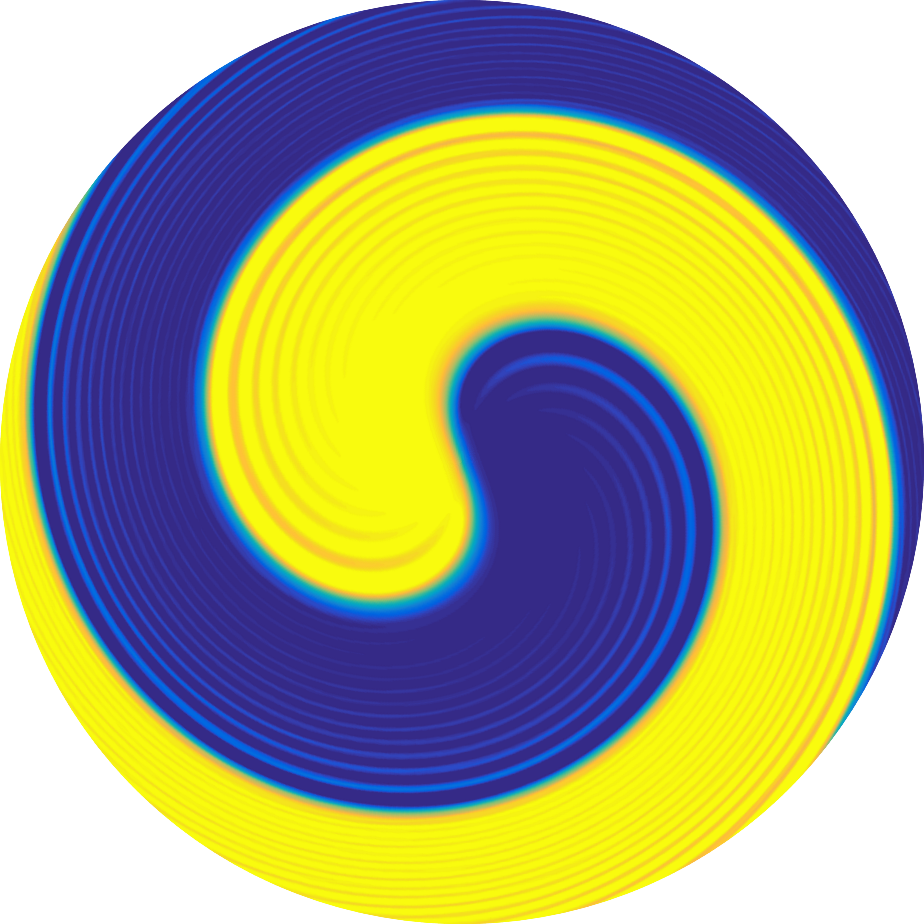}
    \caption{$t=8$}
  \end{subfigure}
  \begin{subfigure}[b]{0.16\textwidth}
    \includegraphics[width=\textwidth]{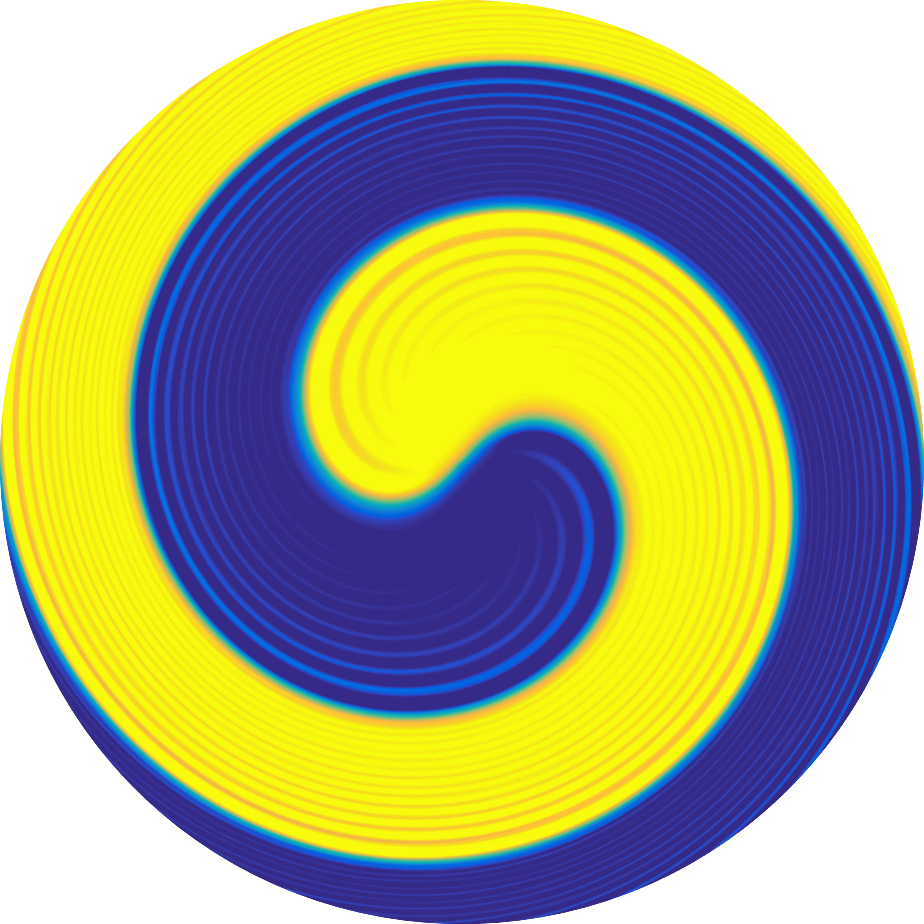}
    \caption{$t=10$}
  \end{subfigure}
  \caption{Evolution of $\theta_h$ under Doswell frontogenesis for $t\in[0,10]$.}
  \label{fig:theta_Doswell}
\end{figure}

\begin{figure}[!htbp]
  \centering
  \begin{subfigure}[b]{0.475\textwidth}
    \centering
    \includegraphics[height=3cm]{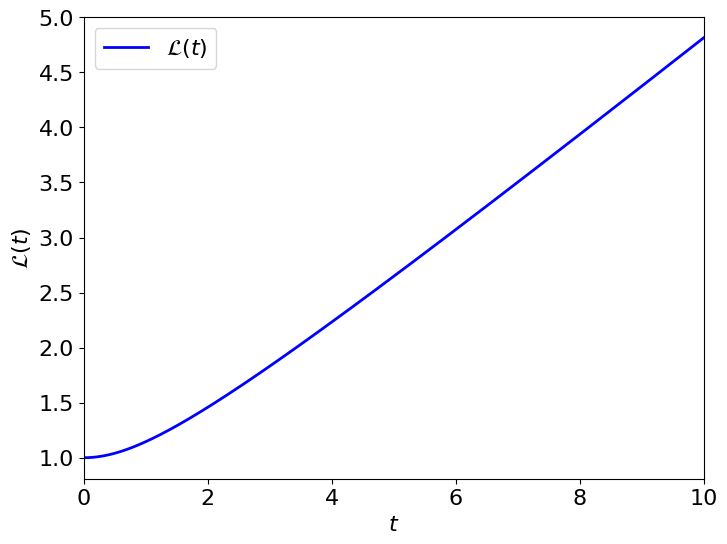}
    \caption{Interface length $\mathcal{L}(t)$ (linear growth).}
    \label{fig:Doswell_length}
  \end{subfigure}
  \begin{subfigure}[b]{0.49\textwidth}
    \centering
    \includegraphics[height=3cm]{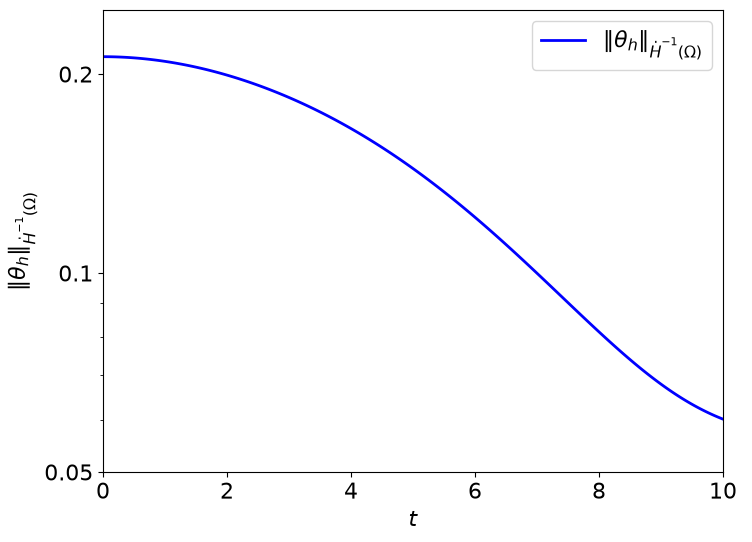}
    \caption{$\dot{H}^{-1}$ mix-norm (observed decay).}
    \label{fig:Doswell_mixnorm}
  \end{subfigure}
    \caption{Interface growth and mix-norm decay under the Doswell velocity field.}
  \label{fig:Doswell_mixnorm-length}
\end{figure}

\subsection{Optimized cellular-flow mixers}\label{sec:optimal_flows}
Having established the limitations of stationary mixing, we now apply adjoint-based optimization to design time-dependent stirring strategies. We restrict the stream function to $N=2$ modes with basis functions $h_k(x_1,x_2)=\sin(k\pi x_1)\sin(k\pi x_2)$, $k\in\{1,2\}$. The optimization is performed on $\Omega=(0,1)^2$ over the time horizon $t\in[0,1]$ with $M=1000$ time steps and $N_p=100{,}000$ markers. To evaluate the $\dot{H}^{-1}$ mix-norm, we use a finite-volume Eulerian transport solve with $\Delta x={1/500}$; the same resolution is used in all subsequent experiments. {The interface pictures use $200{,}000$ markers with arc-length redistribution, and the scalar pictures use backward characteristics on the Eulerian cell centers. These visualizations are separate from the finite-volume mix-norm calculation; the plotted quantity $\mathcal L_{\mathrm{disp}}(t)$ is the ordinary polyline length of the display markers.}

\subsubsection{Constant initial guess}
We first initialize the controls as
\begin{equation}
\label{eq:control_cellular_1}
    u_1(t)=1,\quad u_2(t)=1, \quad t\in[0,1].
\end{equation}

Figure~\ref{fig:interface_Cellular-1} shows the geometric stretching of the interface under the optimized flow. Figure~\ref{fig:Cellular_1_mixnorm-length} shows {fitted display-length growth at rate $6.40$ ($R^2=0.976$) and mix-norm decay at rate $2.42$ ($R^2=0.866$) over the full simulated horizon}. Figure~\ref{fig:theta_Cellular-1} shows the corresponding scalar field, and the optimized controls $u_1(t)$ and $u_2(t)$ are displayed in Figure~\ref{fig:Cellular_1_u1u2}.

\begin{figure}[!htbp]
  \centering
  \begin{subfigure}[b]{0.16\textwidth}
    \includegraphics[width=\textwidth]{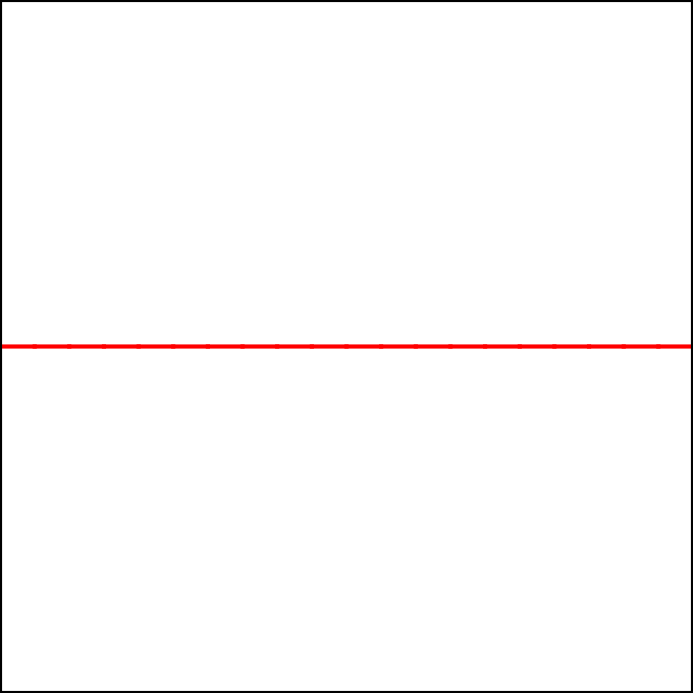}
    \caption{$t=0$}
  \end{subfigure}
  \begin{subfigure}[b]{0.16\textwidth}
    \includegraphics[width=\textwidth]{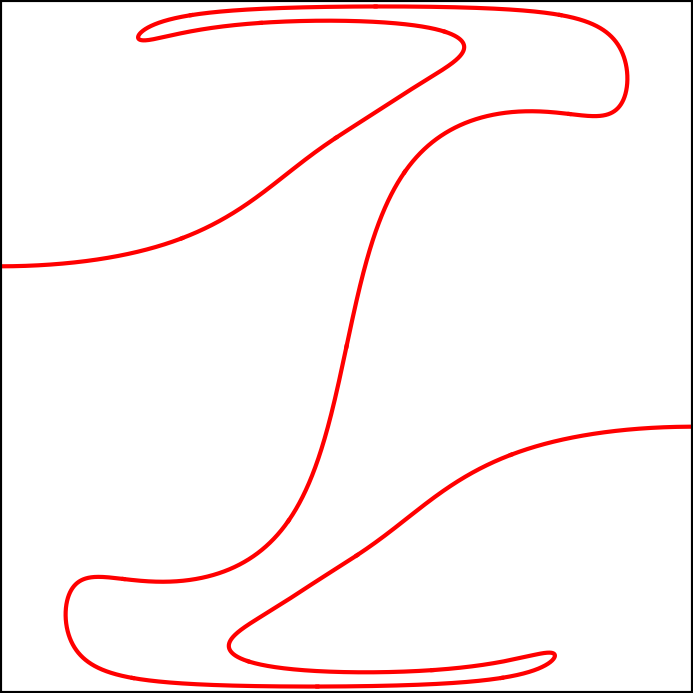}
    \caption{$t=0.2$}
  \end{subfigure}
  \begin{subfigure}[b]{0.16\textwidth}
    \includegraphics[width=\textwidth]{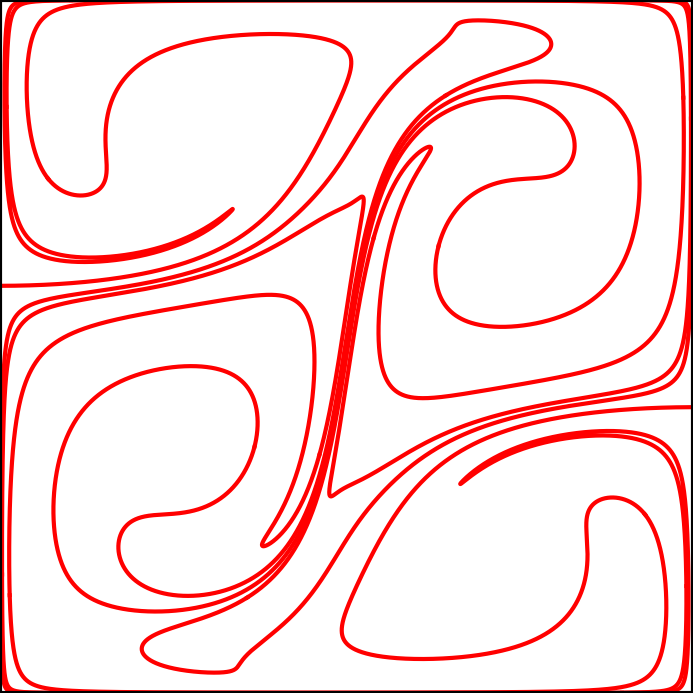}
    \caption{$t=0.4$}
  \end{subfigure}
  \begin{subfigure}[b]{0.16\textwidth}
    \includegraphics[width=\textwidth]{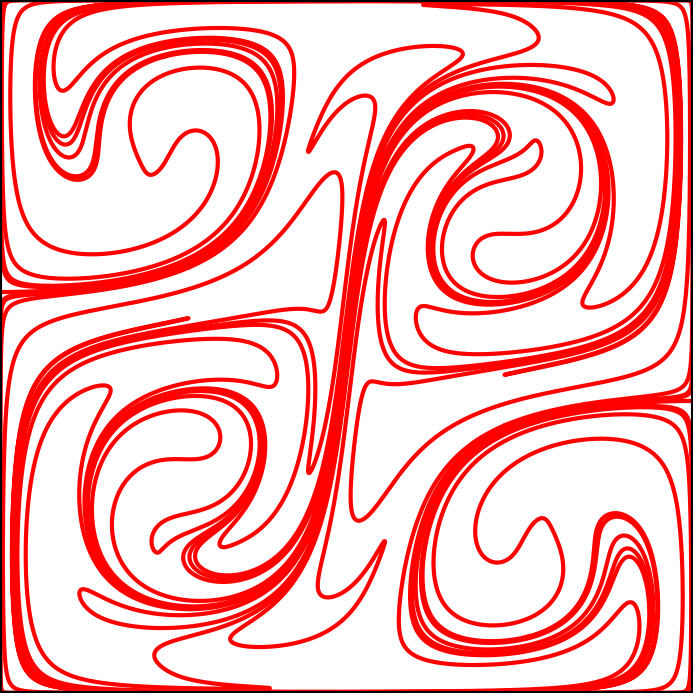}
    \caption{$t=0.6$}
  \end{subfigure}
  \begin{subfigure}[b]{0.16\textwidth}
    \includegraphics[width=\textwidth]{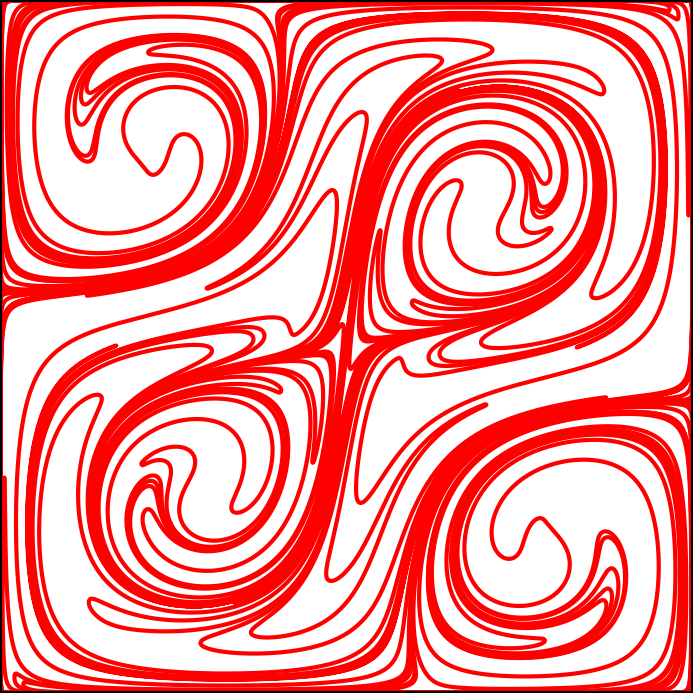}
    \caption{$t=0.8$}
  \end{subfigure}
  \begin{subfigure}[b]{0.16\textwidth}
    \includegraphics[width=\textwidth]{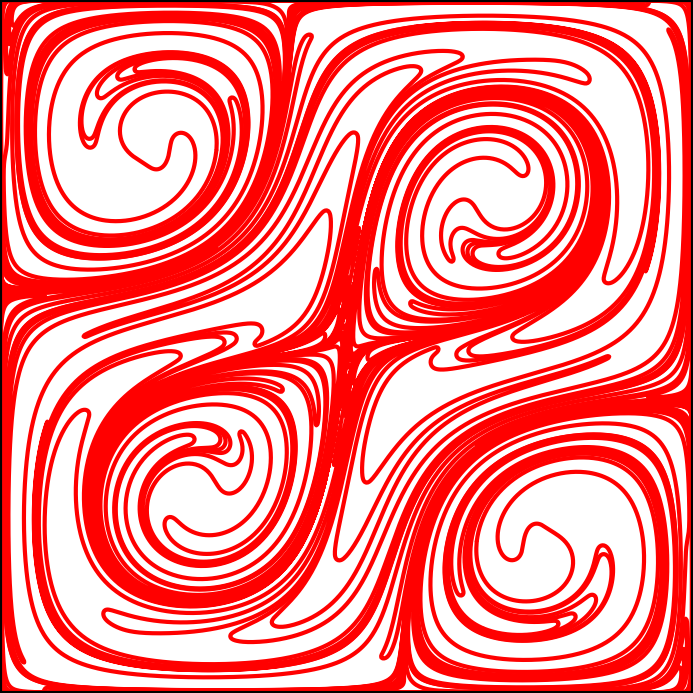}
    \caption{$t=1.0$}
  \end{subfigure}
  \caption{Interface evolution under the optimized cellular flow (initial guess \eqref{eq:control_cellular_1}).}
  \label{fig:interface_Cellular-1}
\end{figure}

\begin{figure}[!htbp]
  \centering
  \begin{subfigure}[b]{0.16\textwidth}
    \includegraphics[width=\textwidth]{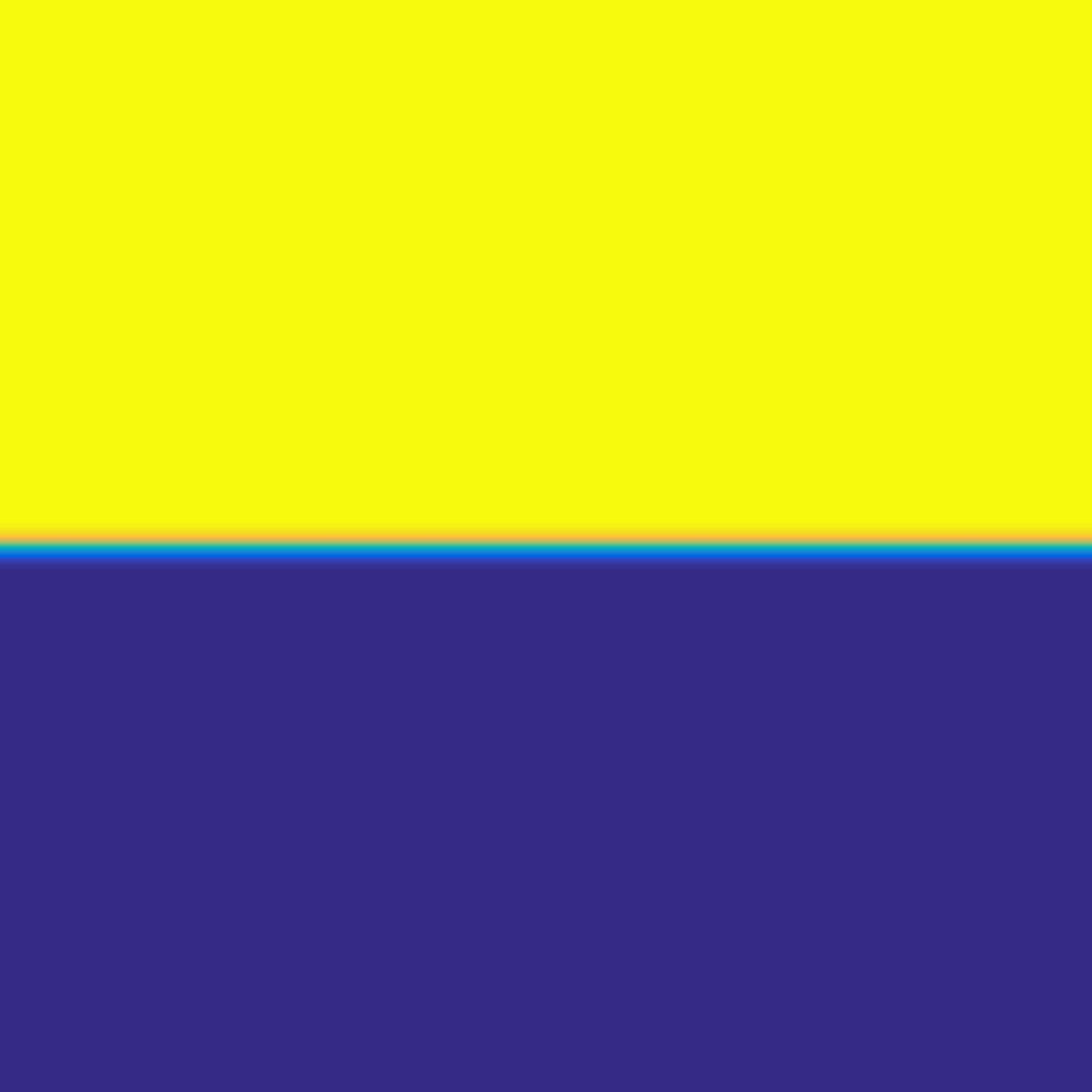}
    \caption{$t=0$}
  \end{subfigure}
  \begin{subfigure}[b]{0.16\textwidth}
    \includegraphics[width=\textwidth]{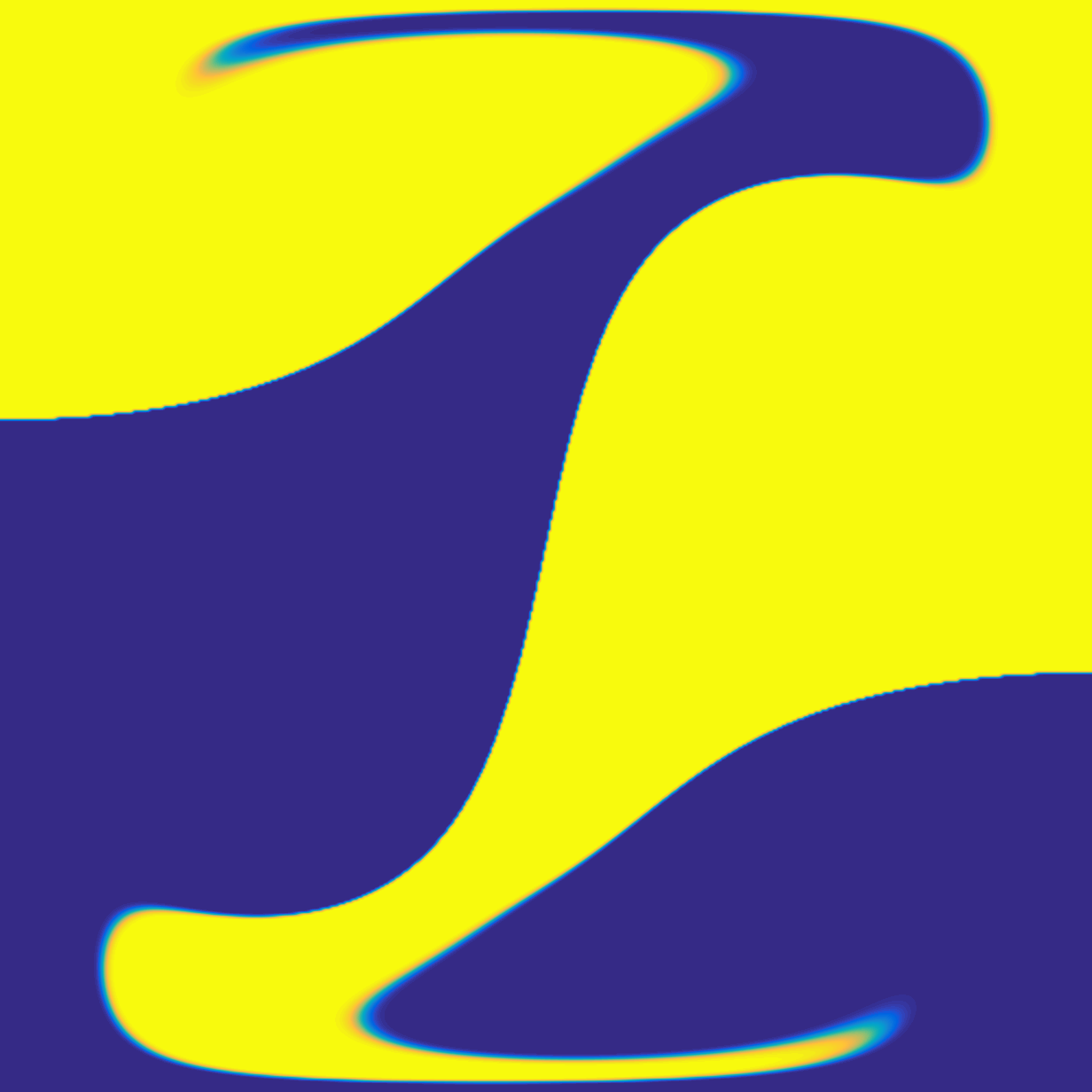}
    \caption{$t=0.2$}
  \end{subfigure}
  \begin{subfigure}[b]{0.16\textwidth}
    \includegraphics[width=\textwidth]{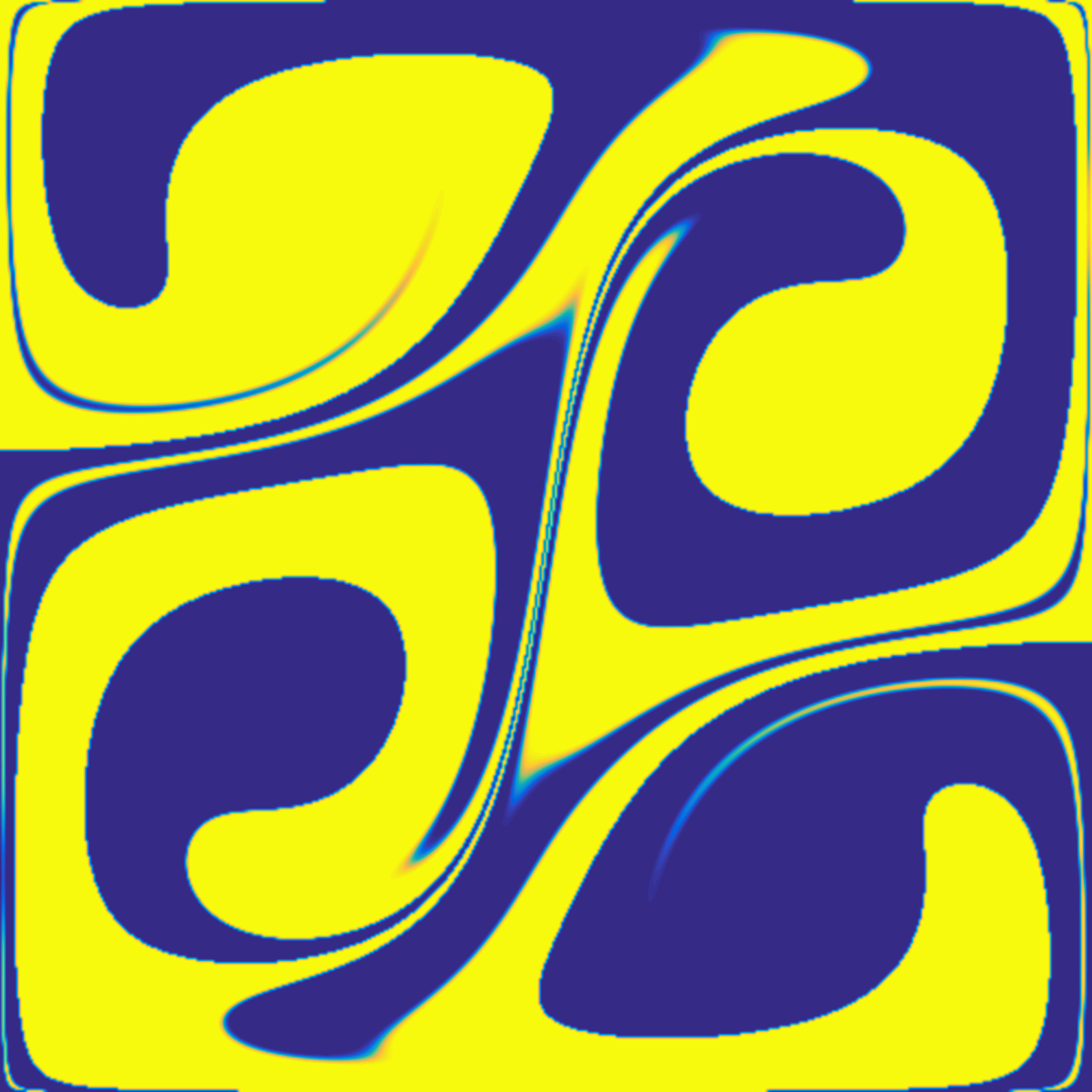}
    \caption{$t=0.4$}
  \end{subfigure}
  \begin{subfigure}[b]{0.16\textwidth}
    \includegraphics[width=\textwidth]{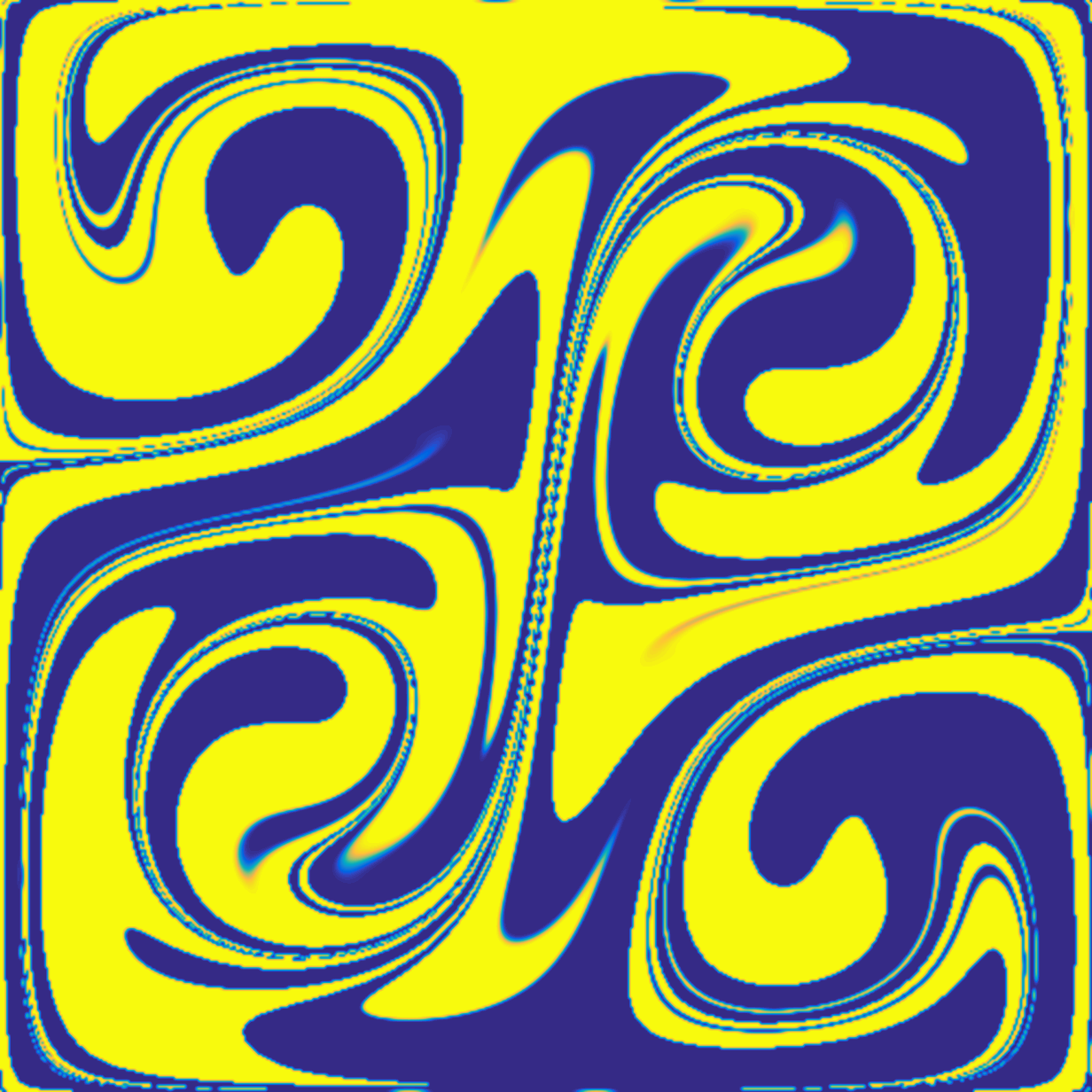}
    \caption{$t=0.6$}
  \end{subfigure}
  \begin{subfigure}[b]{0.16\textwidth}
    \includegraphics[width=\textwidth]{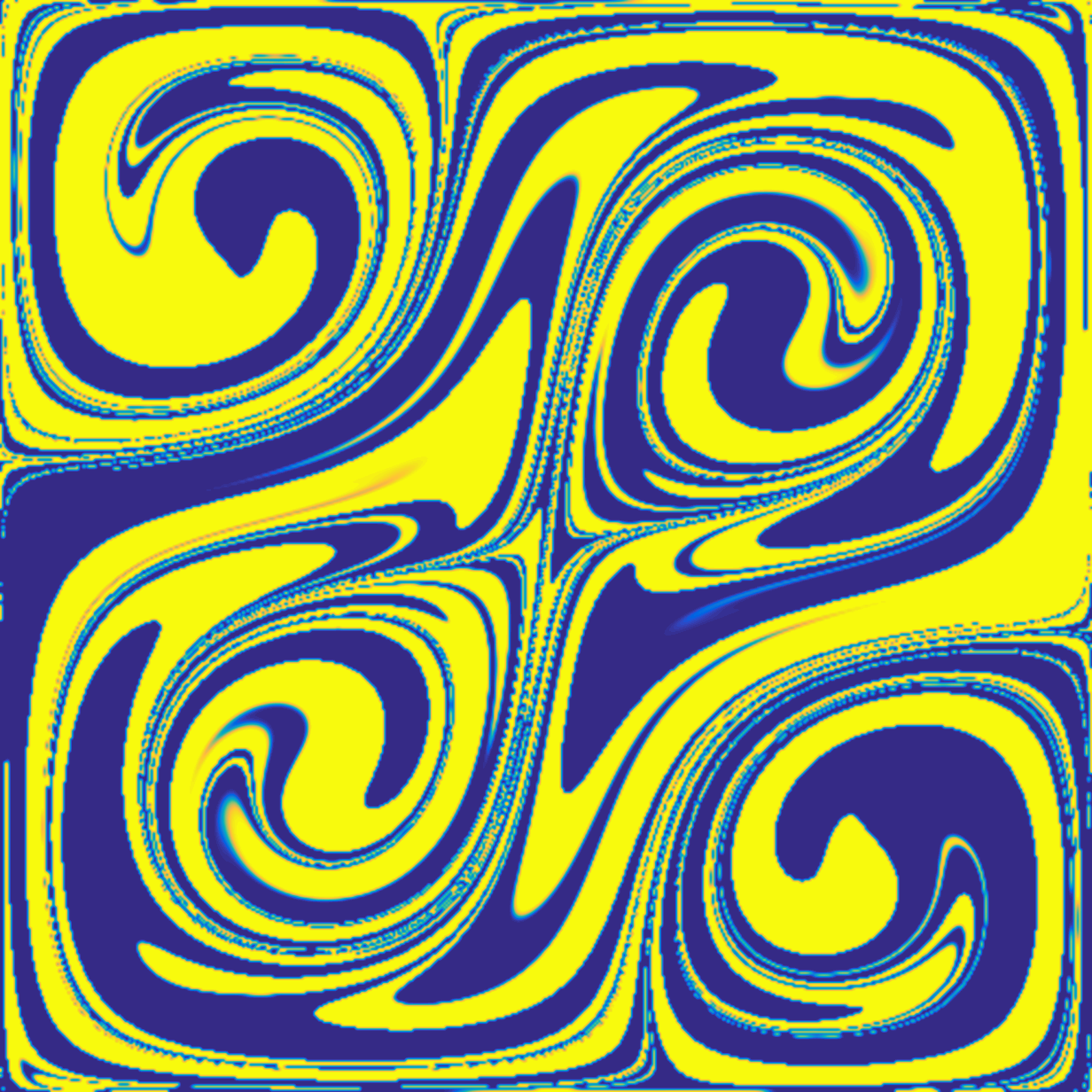}
    \caption{$t=0.8$}
  \end{subfigure}
  \begin{subfigure}[b]{0.16\textwidth}
    \includegraphics[width=\textwidth]{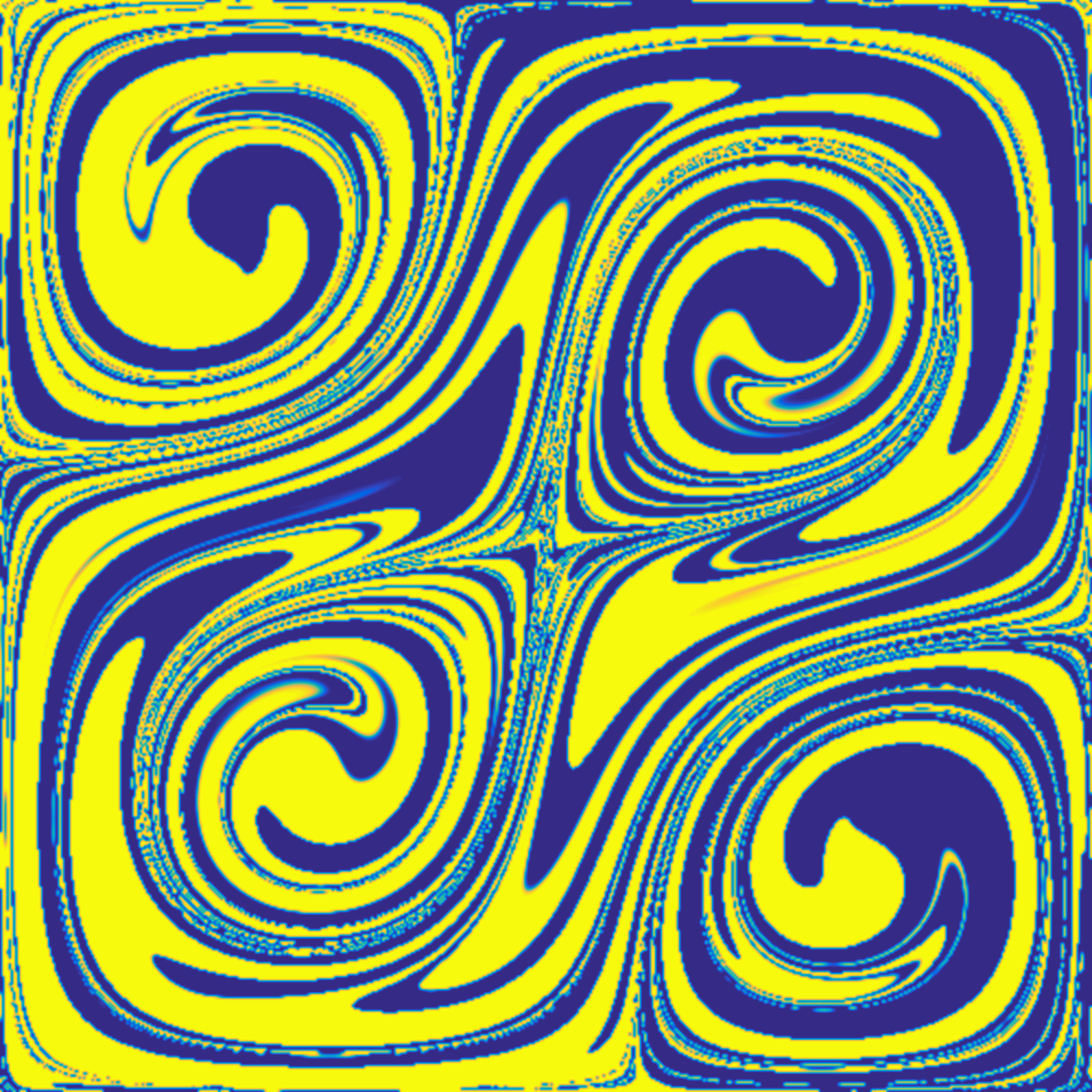}
    \caption{$t=1.0$}
  \end{subfigure}
  \caption{Scalar field {reconstructed by backward characteristics} under the optimized cellular flow (initial guess \eqref{eq:control_cellular_1}).}
  \label{fig:theta_Cellular-1}
\end{figure}

\begin{figure}[!htbp]
  \centering
  \begin{subfigure}[b]{0.32\textwidth}
    \includegraphics[width=\textwidth]{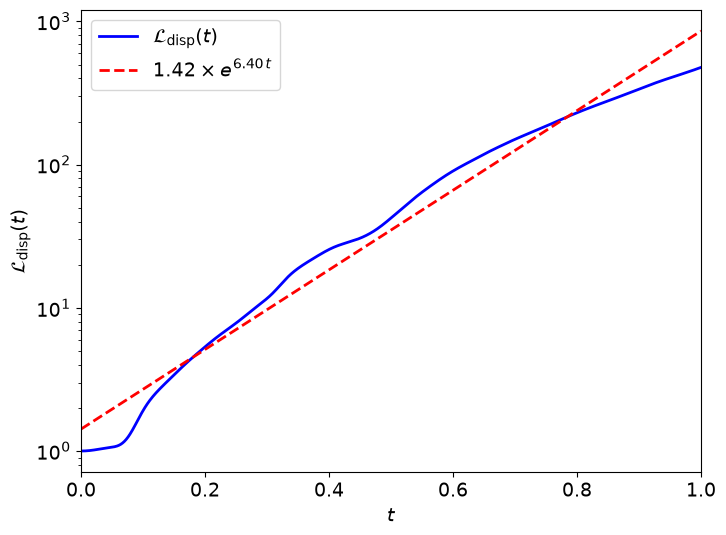}
    \caption{Interface length $\mathcal{L}_{{\mathrm{disp}}}(t)$.}
    \label{fig:Cellular_1_length}
  \end{subfigure}
  \begin{subfigure}[b]{0.32\textwidth}
    \includegraphics[width=\textwidth]{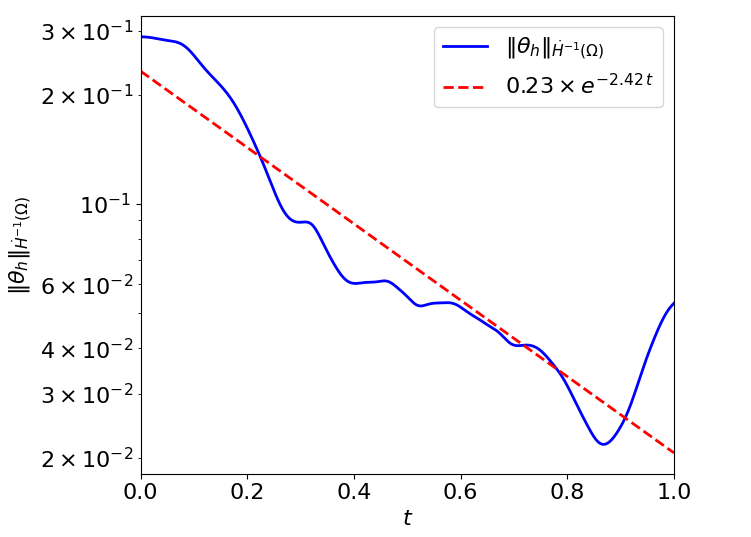}
    \caption{$\dot{H}^{-1}$ mix-norm.}
    \label{fig:Cellular_1_mixnorm}
  \end{subfigure}
  \begin{subfigure}[b]{0.32\textwidth}
    \includegraphics[width=\textwidth]{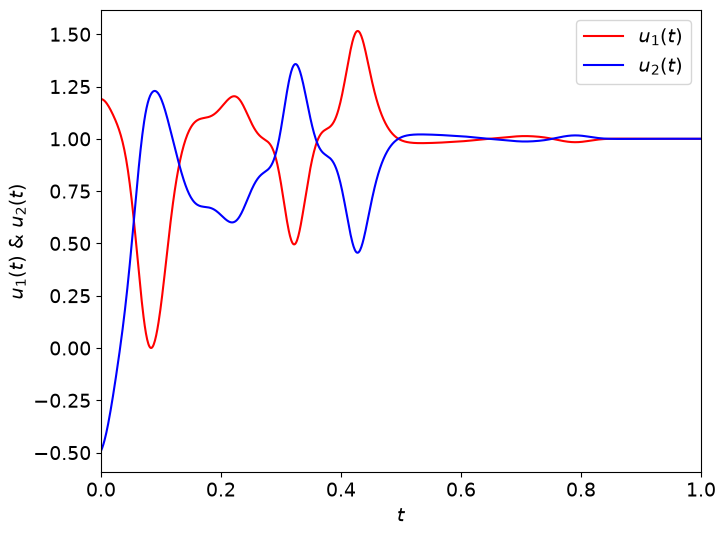}
    \caption{Optimized $u_1(t),u_2(t)$.}
    \label{fig:Cellular_1_u1u2}
  \end{subfigure}
    \caption{Quantitative results under the optimized cellular flow (initial guess \eqref{eq:control_cellular_1}).}
  \label{fig:Cellular_1_mixnorm-length}
\end{figure}

\subsubsection{Oscillatory initial guess}
To {examine sensitivity to} the initial guess, we perform a second optimization with the oscillatory initialization
\begin{equation}
\label{eq:control_cellular_2}
    u_1(t)=\cos(\pi t/2),\quad u_2(t)=\sin(\pi t/2), \quad t\in[0,1].
\end{equation}

{The fitted display-length growth rate is $4.52$ ($R^2=0.993$) and the mix-norm decay rate is $1.84$ ($R^2=0.910$)} (Figure~\ref{fig:Cellular_2_mixnorm-length}). {Both initial guesses yield stretching and mix-norm decay, but two runs do not establish robustness to initialization.}

\begin{figure}[!htbp]
  \centering
  \begin{subfigure}[b]{0.32\textwidth}
    \includegraphics[width=\textwidth]{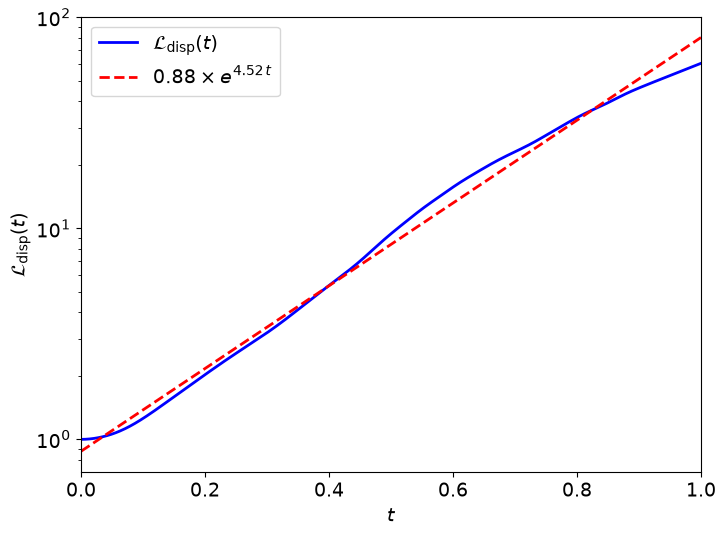}
    \caption{Interface length $\mathcal{L}_{{\mathrm{disp}}}(t)$.}
    \label{fig:Cellular_2_length}
  \end{subfigure}
  \begin{subfigure}[b]{0.34\textwidth}
    \includegraphics[width=\textwidth]{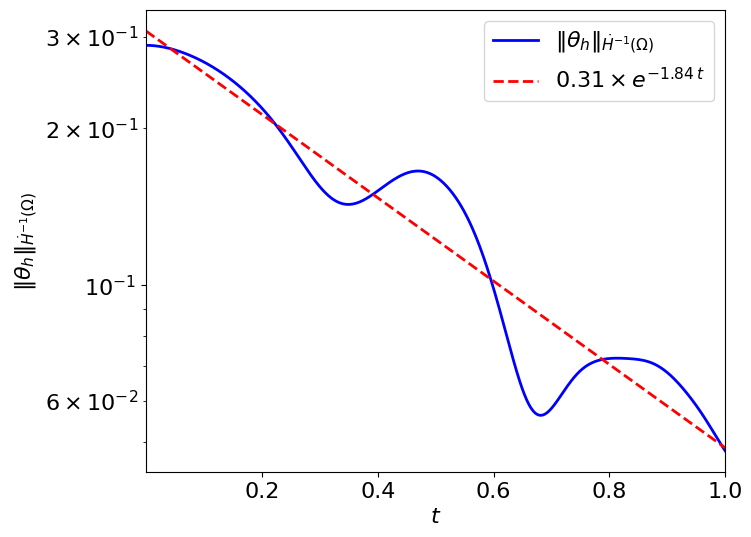}
    \caption{$\dot{H}^{-1}$ mix-norm.}
    \label{fig:Cellular_2_mixnorm}
  \end{subfigure}
  \begin{subfigure}[b]{0.32\textwidth}
    \includegraphics[width=\textwidth]{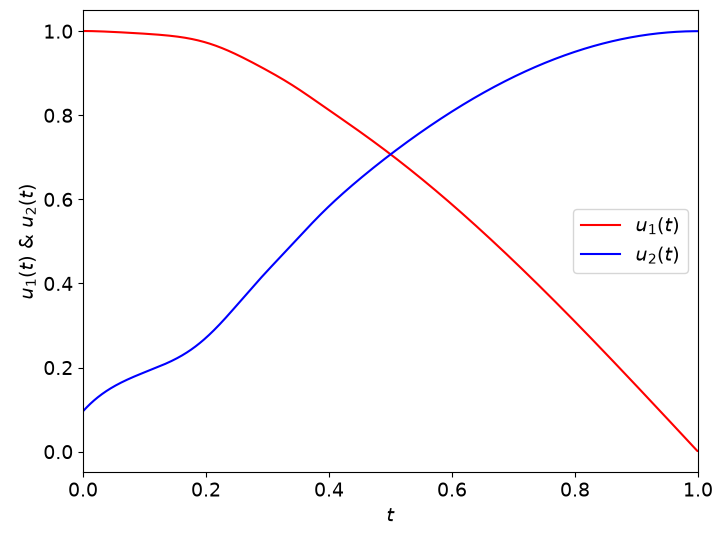}
    \caption{Optimized $u_1(t),u_2(t)$.}
    \label{fig:Cellular_2_u1u2}
  \end{subfigure}
    \caption{Quantitative results under the optimized cellular flow (initial guess \eqref{eq:control_cellular_2}).}
  \label{fig:Cellular_2_mixnorm-length}
\end{figure}

\subsection{Optimized Doswell-type mixers}\label{sec:optimal_Doswell}
We next consider a more complex flow configuration on a circular domain to demonstrate the flexibility of our framework beyond the Cartesian setting. The domain is the same disc $\Omega = \{ \mathbf{x} \in \mathbb{R}^2 : |\mathbf{x} - \mathbf{x}_c| < R \}$ with center $\mathbf{x}_c = (0.5, 0.5)$ and radius $R=0.5$ as in the stationary benchmark of Section~\ref{sec:Doswell-stable}. The optimization is performed over $t\in[0,5]$ with $M=5000$ time steps and $N_p=100{,}000$ markers. For the Eulerian transport solve used in mix-norm evaluation, we use a polar grid with {$N_r=500$ radial cells and $N_\varphi=500$ angular cells}.

The velocity field is parameterized by $N=2$ basis flows. The first, $\mathbf{b}_1$, is the standard Doswell frontogenesis field \eqref{eq:v_Doswell} centered at $\mathbf{x}_c$. The second, $\mathbf{b}_2$, is a five-cell (multi-vortex) Doswell-type field formed by superposing five Doswell vortices supported on five embedded discs within $\Omega$: four congruent discs arranged around the periphery and one centered at $\mathbf{x}_c$. To enforce no-penetration at each disc boundary and to obtain {$C^{2,1}$ spatial regularity} when the field is extended by zero outside the disc, each vortex is multiplied by the radial cutoff
\begin{equation}\label{eq:doswell_cutoff}
C(r)=\bigl(1-(r/R_c)^2\bigr)^3 \quad (0\le r<R_c),
\qquad C(r)=0 \quad (r\ge R_c),
\end{equation}
where $R_c$ denotes the disc radius and $r$ is the distance to the corresponding vortex center. This choice satisfies $C(R_c)=0$ and $C'(R_c)=0${, as well as $C''(R_c)=0$}. The resulting field $\mathbf{b}_2$ is incompressible in $\Omega$ and satisfies $\mathbf{b}_2\cdot\mathbf{n}=0$ on $\partial\Omega$. {The cutoff is $C^{2,1}$ across the support boundary, and its stream function is $C^{3,1}$. This regularity ensures that the basis satisfies the hypothesis of Theorem~\ref{prop:H3}. To make the stream-function representation explicit, let $c_i$ and $R_i$ be the center and radius of the $i$th embedded disc. With $r_i=|x-c_i|$ and $C_i(r)=\max\{0,1-(r/R_i)^2\}^3$, the individual stream functions are}
\begin{equation}\label{eq:doswell_multi_stream}
{H_i(x)=\begin{cases}-\displaystyle\int_{r_i}^{R_i}\rho g(\rho)C_i(\rho)\,d\rho,&r_i<R_i,\\0,&r_i\ge R_i,\end{cases}\qquad h_2=\sum_{i=1}^5H_i.}
\end{equation}
{Here $g$ is the angular velocity profile in Section~\ref{sec:Doswell-stable}, with $g(0)=\overline v$. For the first basis field, the zero-trace stream function is}
\begin{equation}\label{eq:doswell_stream}
{h_1(x)=\frac{\overline v}{2}\bigl(\tanh^2 r-\tanh^2R\bigr).}
\end{equation}

\subsubsection{Constant initial guess}
We initialize the controls as in \eqref{eq:control_cellular_1}, {now on $[0,5]$}. Figures~\ref{fig:interface_Doswell-1} and \ref{fig:theta_Doswell-1} display the optimized interface and scalar field evolution, respectively. The {display length has a fitted growth rate of $0.69$ ($R^2=0.960$)} (Figure~\ref{fig:Doswell_1_length}), {and the mix-norm has a fitted decay rate of $0.43$ ($R^2=0.943$)} (Figure~\ref{fig:Doswell_1_mixnorm}). The optimized controls are shown in Figure~\ref{fig:Doswell_1_u1u2}.

\begin{figure}[!htbp]
  \centering
  \begin{subfigure}[b]{0.16\textwidth}
    \includegraphics[width=\textwidth]{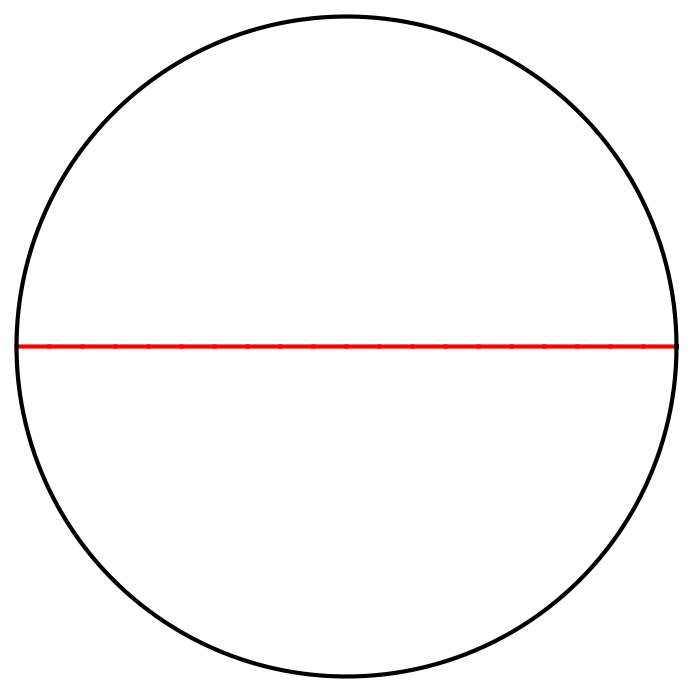}
    \caption{$t=0$}
  \end{subfigure}
  \begin{subfigure}[b]{0.16\textwidth}
    \includegraphics[width=\textwidth]{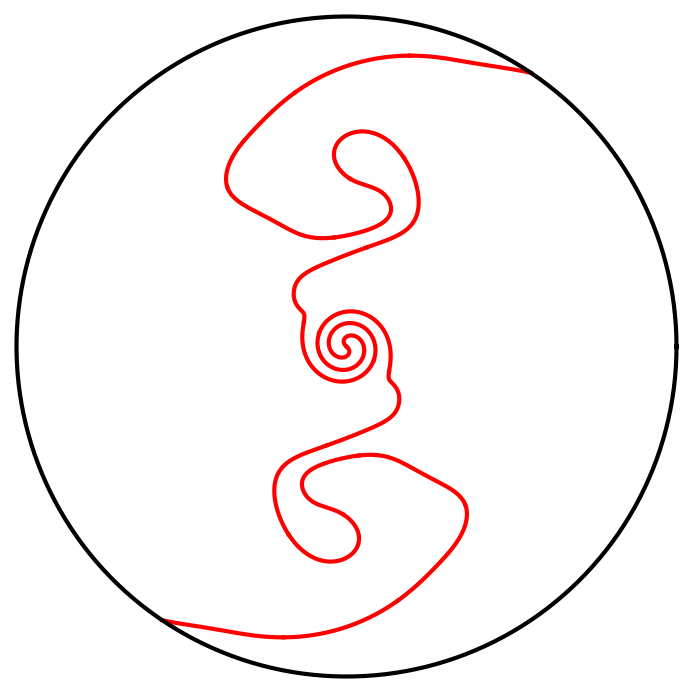}
    \caption{$t=1$}
  \end{subfigure}
  \begin{subfigure}[b]{0.16\textwidth}
    \includegraphics[width=\textwidth]{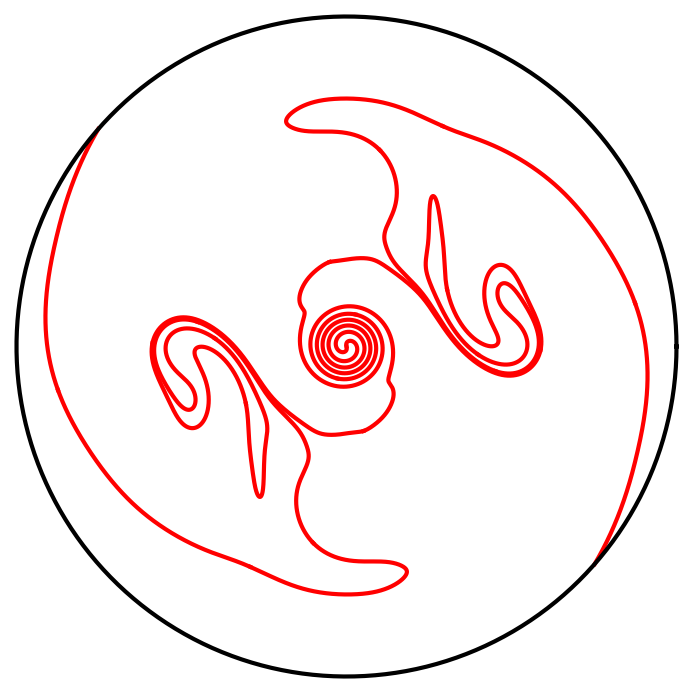}
    \caption{$t=2$}
  \end{subfigure}
  \begin{subfigure}[b]{0.16\textwidth}
    \includegraphics[width=\textwidth]{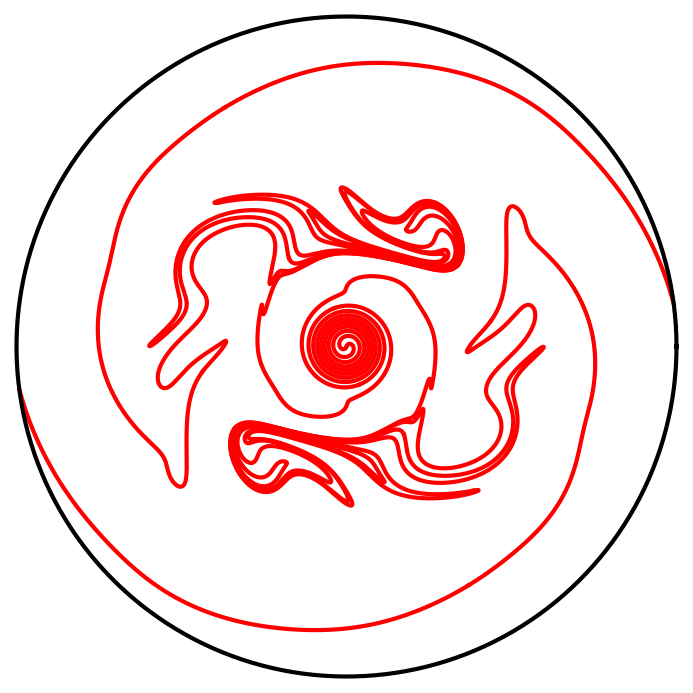}
    \caption{$t=3$}
  \end{subfigure}
  \begin{subfigure}[b]{0.16\textwidth}
    \includegraphics[width=\textwidth]{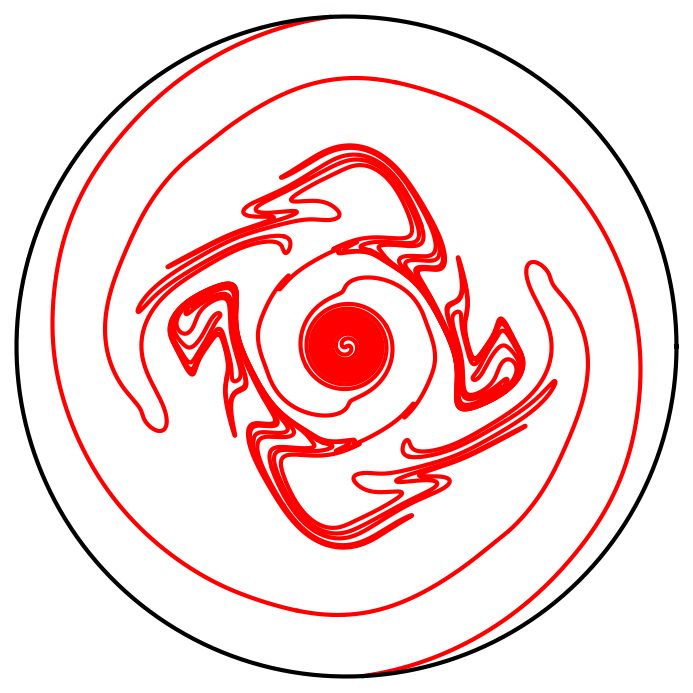}
    \caption{$t=4$}
  \end{subfigure}
  \begin{subfigure}[b]{0.16\textwidth}
    \includegraphics[width=\textwidth]{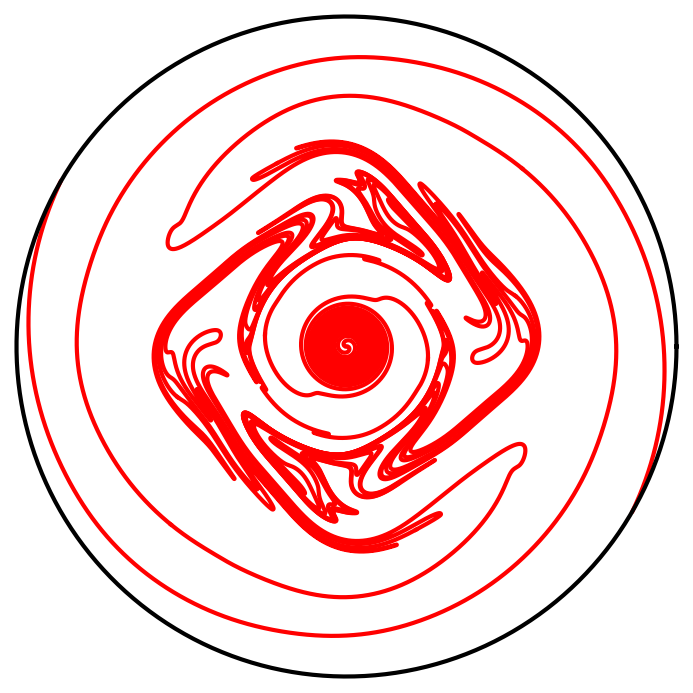}
    \caption{$t=5$}
  \end{subfigure}
  \caption{Interface evolution under the optimized Doswell flow (initial guess \eqref{eq:control_cellular_1}).}
  \label{fig:interface_Doswell-1}
\end{figure}

\begin{figure}[!htbp]
  \centering
  \begin{subfigure}[b]{0.16\textwidth}
    \includegraphics[width=\textwidth]{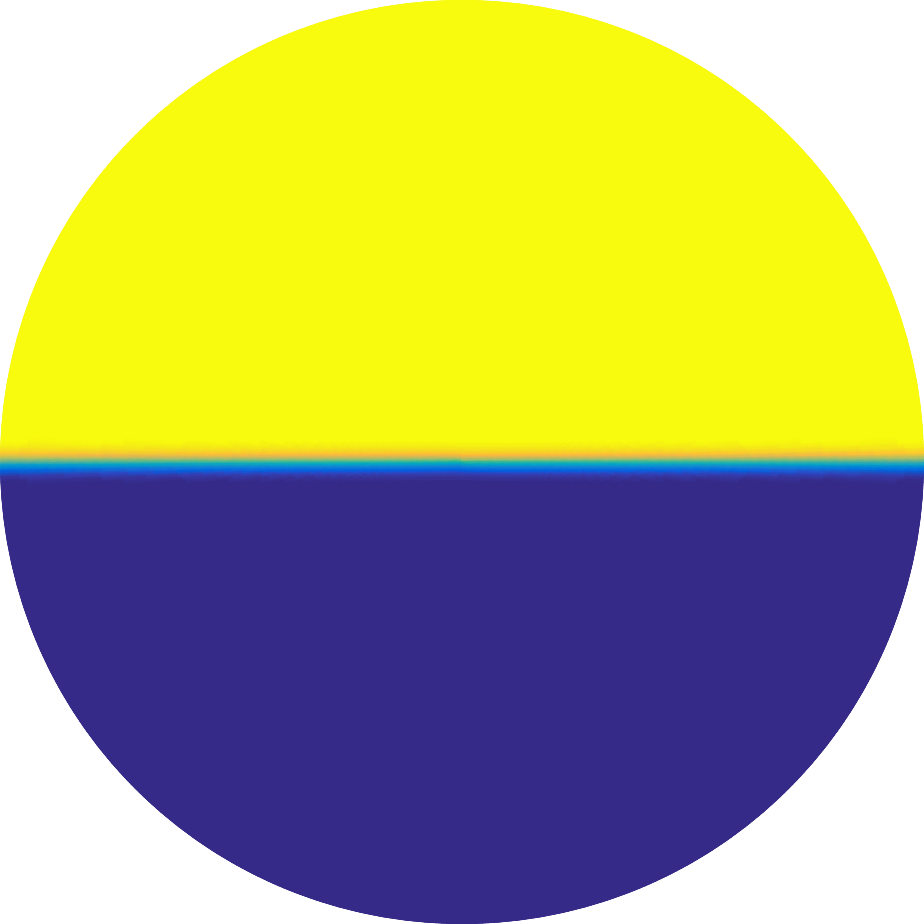}
    \caption{$t=0$}
  \end{subfigure}
  \begin{subfigure}[b]{0.16\textwidth}
    \includegraphics[width=\textwidth]{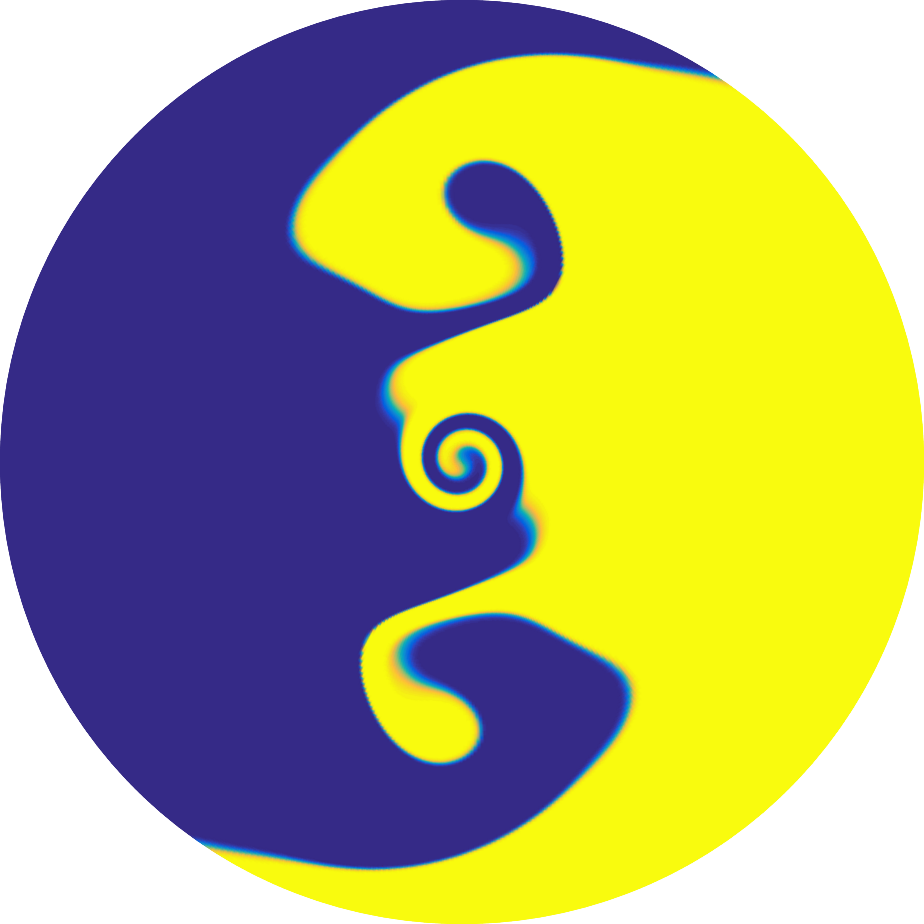}
    \caption{$t=1$}
  \end{subfigure}
  \begin{subfigure}[b]{0.16\textwidth}
    \includegraphics[width=\textwidth]{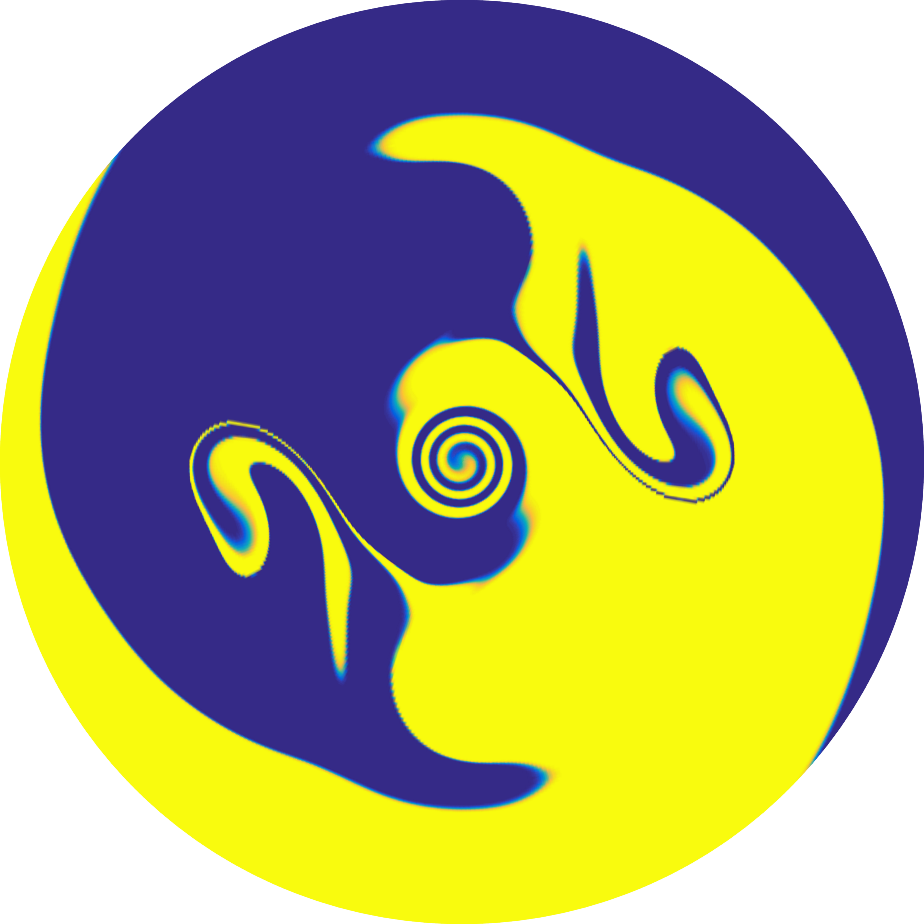}
    \caption{$t=2$}
  \end{subfigure}
  \begin{subfigure}[b]{0.16\textwidth}
    \includegraphics[width=\textwidth]{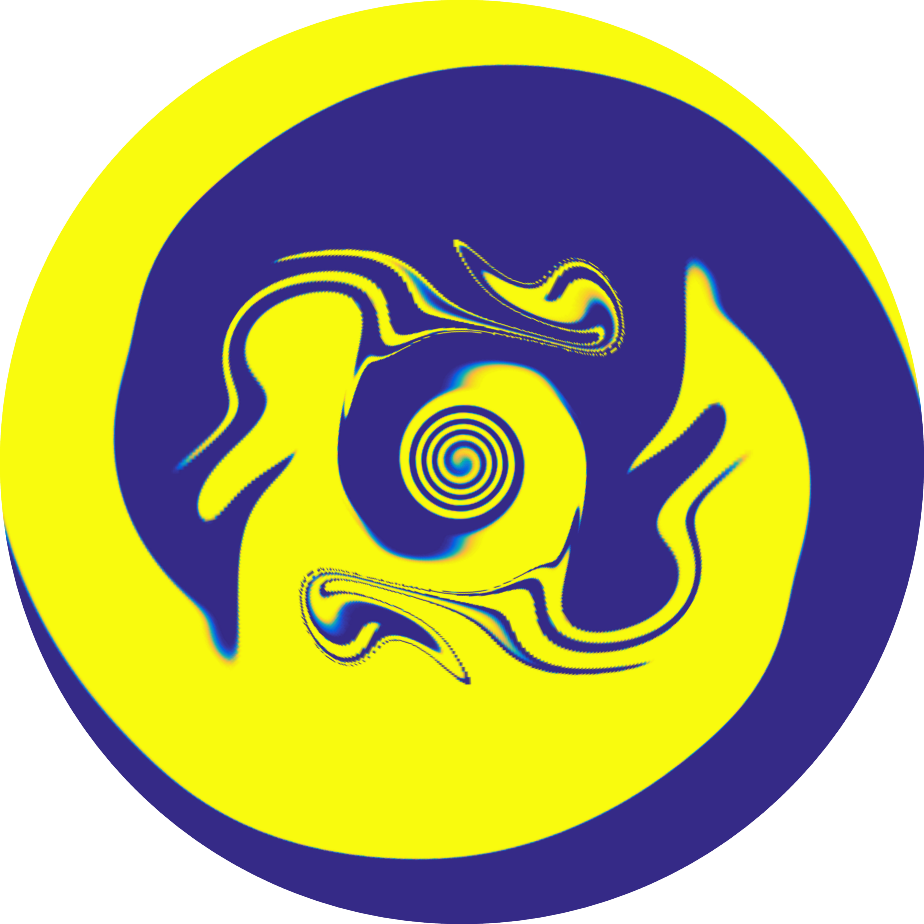}
    \caption{$t=3$}
  \end{subfigure}
  \begin{subfigure}[b]{0.16\textwidth}
    \includegraphics[width=\textwidth]{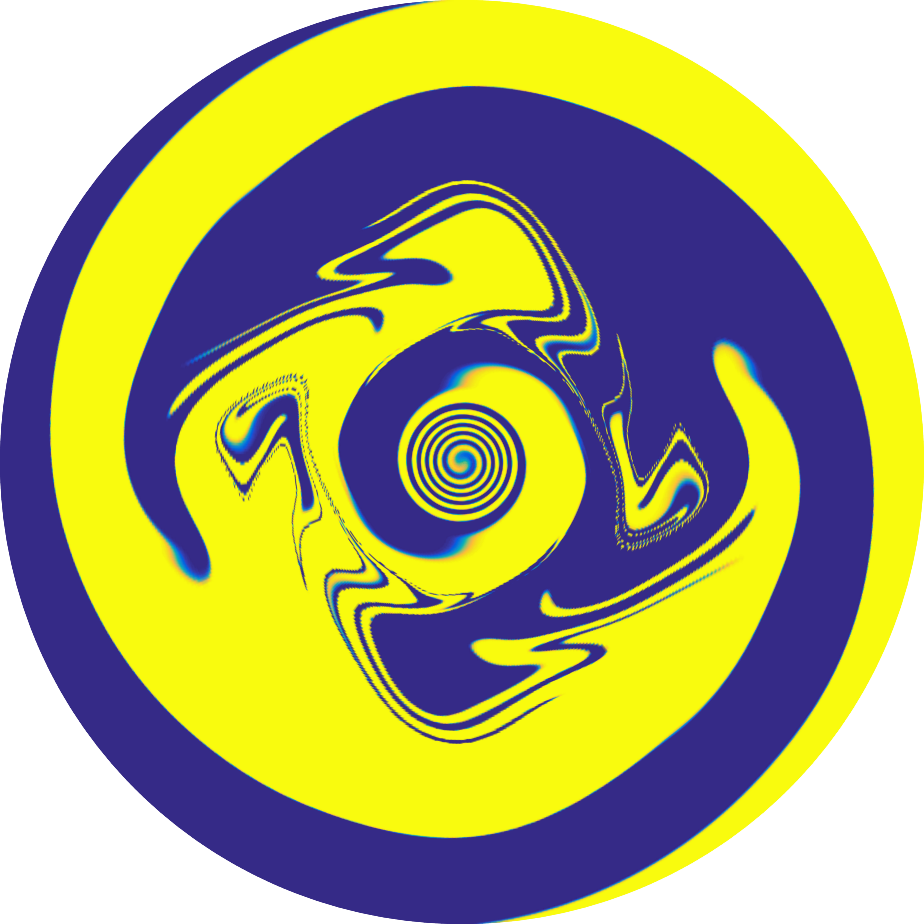}
    \caption{$t=4$}
  \end{subfigure}
  \begin{subfigure}[b]{0.16\textwidth}
    \includegraphics[width=\textwidth]{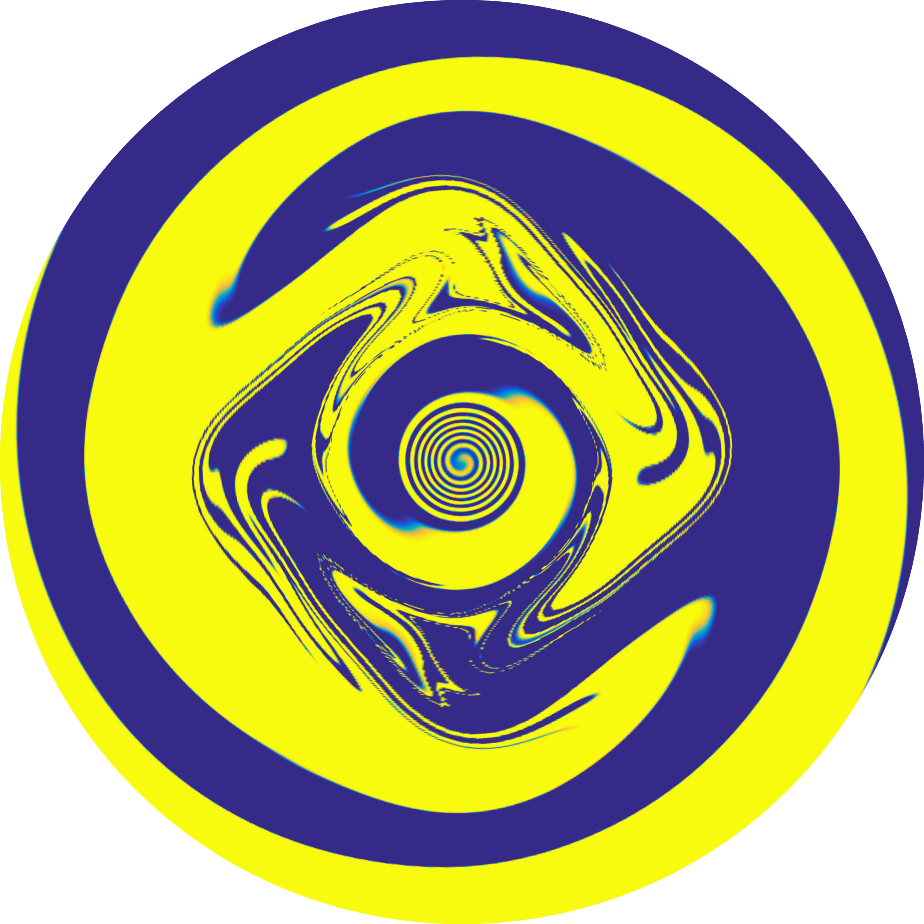}
    \caption{$t=5$}
  \end{subfigure}
  \caption{Scalar field {reconstructed by backward characteristics} under the optimized Doswell flow (initial guess \eqref{eq:control_cellular_1}).}
  \label{fig:theta_Doswell-1}
\end{figure}

\begin{figure}[!htbp]
  \centering
  \begin{subfigure}[b]{0.32\textwidth}
    \includegraphics[width=\textwidth]{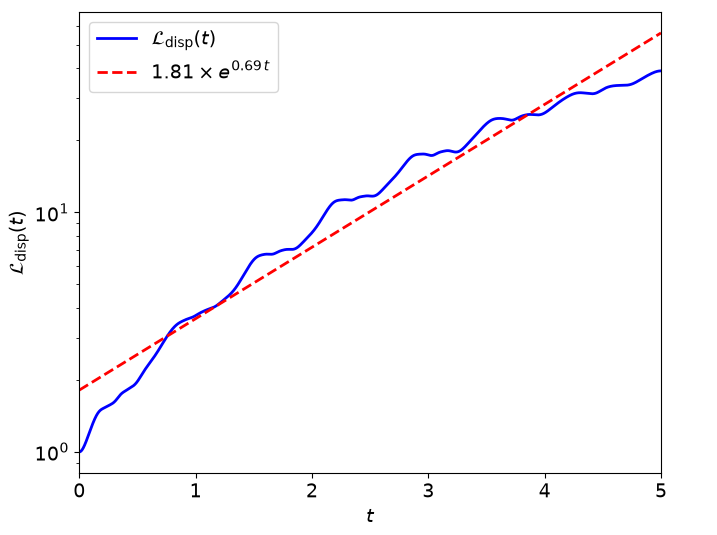}
    \caption{Interface length $\mathcal{L}_{{\mathrm{disp}}}(t)$.}
    \label{fig:Doswell_1_length}
  \end{subfigure}
  \begin{subfigure}[b]{0.34\textwidth}
    \includegraphics[width=\textwidth]{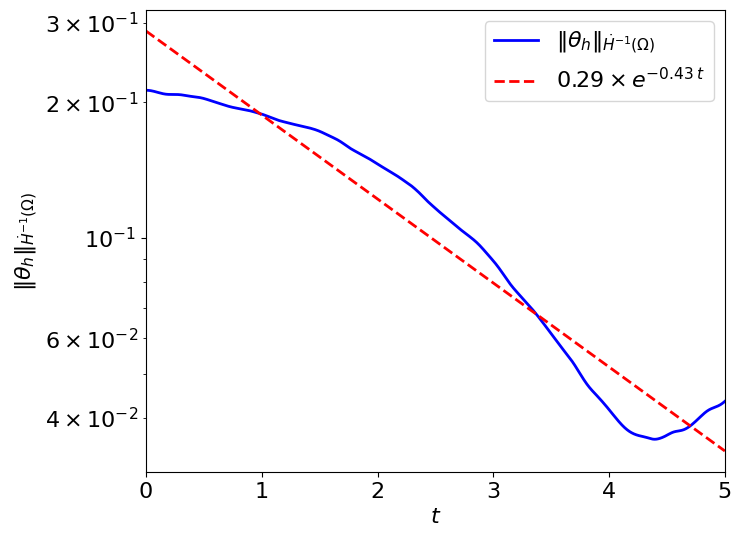}
    \caption{$\dot{H}^{-1}$ mix-norm.}
    \label{fig:Doswell_1_mixnorm}
  \end{subfigure}
  \begin{subfigure}[b]{0.32\textwidth}
    \includegraphics[width=\textwidth]{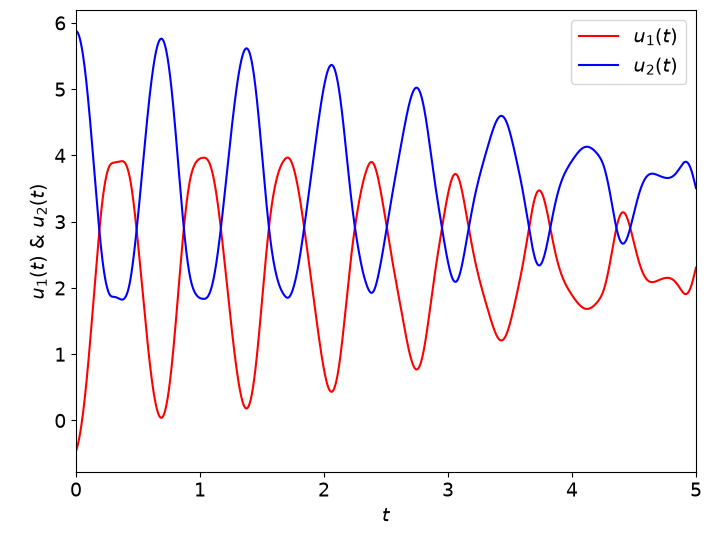}
    \caption{Optimized $u_1(t),u_2(t)$.}
    \label{fig:Doswell_1_u1u2}
  \end{subfigure}
    \caption{Quantitative results under optimized Doswell flow (initial guess \eqref{eq:control_cellular_1}).}
  \label{fig:Doswell_1_mixnorm-length}
\end{figure}

\subsubsection{Oscillatory initial guess}
The second initial control is
\begin{equation}\label{eq:control_doswell_2}
u_1(t)=\cos(\pi t/2),\qquad u_2(t)=\sin(\pi t/2),\qquad t\in[0,5].
\end{equation}
{The fitted display-length growth rate is $0.47$ ($R^2=0.949$) and the mix-norm decay rate is $0.33$ ($R^2=0.808$)} (Figure~\ref{fig:Doswell_2_mixnorm-length}).

\begin{figure}[!htbp]
  \centering
  \begin{subfigure}[b]{0.32\textwidth}
    \includegraphics[width=\textwidth]{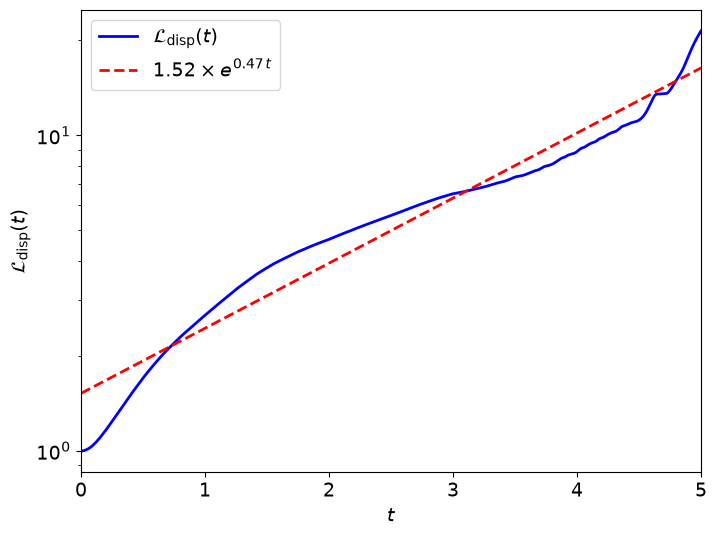}
    \caption{Interface length $\mathcal{L}_{{\mathrm{disp}}}(t)$.}
    \label{fig:Doswell_2_length}
  \end{subfigure}
  \begin{subfigure}[b]{0.34\textwidth}
    \includegraphics[width=\textwidth]{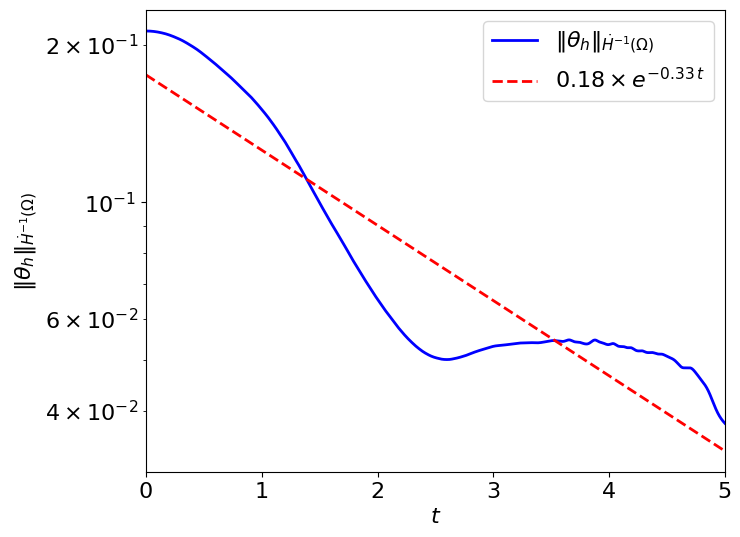}
    \caption{$\dot{H}^{-1}$ mix-norm.}
    \label{fig:Doswell_2_mixnorm}
  \end{subfigure}
  \begin{subfigure}[b]{0.32\textwidth}
    \includegraphics[width=\textwidth]{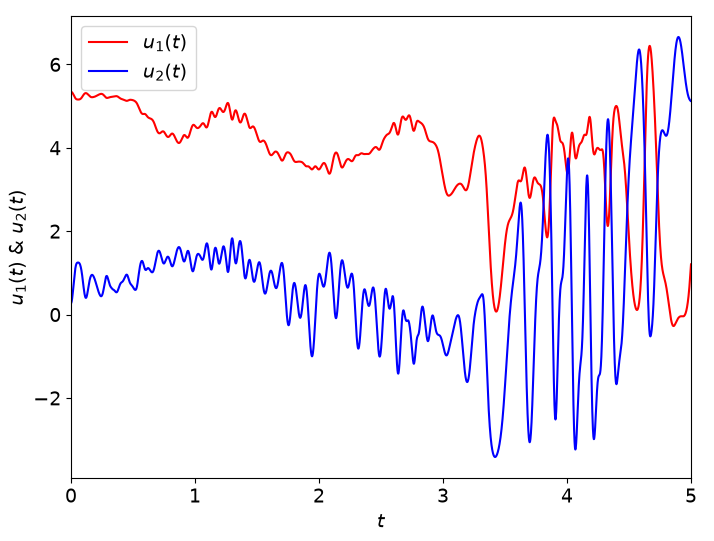}
    \caption{Optimized $u_1(t),u_2(t)$.}
    \label{fig:Doswell_2_u1u2}
  \end{subfigure}
    \caption{Quantitative results under optimized Doswell flow (initial guess \eqref{eq:control_doswell_2}).}
  \label{fig:Doswell_2_mixnorm-length}
\end{figure}

{Even within a single run, display-length growth need not accompany mix-norm decay: near the final time in Figure~\ref{fig:Cellular_1_mixnorm-length}, both quantities increase.} Consistent with Remark~\ref{rem:critical_mixnorm}, this observation reinforces that optimized controls must ultimately be assessed on the full transport equation~\cite{zheng2023numerical, hu2026structure}.

\subsection{Comparison with the Sobolev-norm optimizer}\label{sec:comparison}

As {discussed} in\newline Section~\ref{sec:problem}, the Hamiltonian interface model serves as a reduced-order surrogate for the original mixing problem constrained by the transport equation. We compare it with the Eulerian Sobolev-norm optimizer of \cite{hu2026structure}, {adapted here to the stream-function penalty in \eqref{eq:gram_identity}, with $\lambda=4\times10^{-5}$, in both objectives. This penalizes $\psi$, rather than the velocity. For the cellular basis, $Q=\lambda M=10^{-5}I$; for Doswell, the off-diagonal Gram entries are retained. The Sobolev implementation uses the discrete objective
\[
\frac12\|\theta_h(T)\|_{\dot H^{-1}}^2
+\frac{\Delta t}{2}\sum_{n=0}^{M-1}u_n^\top Q u_n.
\]
Here the terminal norm is the discrete norm used by the original Sobolev optimizer. The interface method imposes $\|u\|_{L^2}\le10$; the Sobolev implementation has no such projection. A common penalty parameter does not enforce equal realized control effort.}

The Sobolev-norm optimizer requires forward and adjoint transport solves on a $500\times500$ grid, together with elliptic solves for the $\dot H^{-1}$ norm. By contrast, the interface-length optimizer tracks $N_p=100{,}000$ marker points and computes its adjoint along the same trajectories. {Within each pair, the basis, initial coefficient arrays, horizon, time step and initial scalar agree. Both frozen controls are reevaluated on the same $500^2$ finite-volume grid with $\Delta t=10^{-3}$. The cellular evaluations retain the original control interpretation: intervalwise coefficients for the interface method and adjacent-velocity averages for the Sobolev implementation. All norms in this comparison use these reevaluations, with residual checks on each linear solve. Table~\ref{tab:optimizer_comparison} reports the full-horizon decay rates, obtained by least-squares fitting $\log\|\theta_h(t)\|_{\dot H^{-1}}$ against $t$ on $[0,T]$.}

{Because the implementations use different PR updates and stopping rules, we report each run's termination status to clarify the scope of the timing comparison. The four $N=2$ interface runs use Polak--Ribi\`ere directions and the Armijo test on the actual projected displacement in Algorithm~\ref{alg:optimize}. All four stop at the budget of $100$ accepted updates. In the row order of Table~\ref{tab:optimizer_comparison}, their normalized projected residuals $r_h$ are $7.087$, $10.046$, $1.223$, and $0.193$, all above the tolerance $10^{-6}$. The Sobolev constant-start cellular run stops on objective change, its oscillatory run on objective change after an unchanged-control line-search fallback, and both Doswell runs on line-search failure. Reevaluating these frozen controls does not validate the historical Sobolev gradients or optimization steps. These are the computed controls being compared, not certified stationary optima.}

{The interface times include the initial objective and gradient evaluation and checkpoint writes, but exclude model setup and the separate transport and display evaluations. The Sobolev cellular times cover the original optimization loop; the Doswell times also include the initial forward solve and history-norm evaluations. These timing scopes are retained in Table~\ref{tab:optimizer_comparison}.}

\begin{table}[!htbp]
\centering
\begin{tabular}{|c|c|c|c|c|}
\hline
\multirow{2}{*}{Flow setting}
& \multicolumn{2}{c|}{Sobolev-norm optimizer}
& \multicolumn{2}{c|}{Interface-length optimizer} \\
\cline{2-5}
& Time (s) & Decay rate & Time (s) & Decay rate \\ \hline
Cellular with \eqref{eq:control_cellular_1} & {7294.30} & {2.55} & {144.66} & {2.42} \\ \hline
Cellular with \eqref{eq:control_cellular_2} & {14202.63} & {1.81} & {144.78} & {1.84} \\ \hline
Doswell with \eqref{eq:control_cellular_1} & {23875.32} & {0.43} & {945.44} & {0.43} \\ \hline
Doswell with \eqref{eq:control_doswell_2} & {34318.63} & {0.48} & {541.97} & {0.33} \\ \hline
\end{tabular}
\caption{Comparison of the Sobolev-norm optimizer \cite{hu2026structure} and the interface-length optimizer: measured optimization time and full-horizon $\dot{H}^{-1}$ mix-norm decay rate. The timing scopes are specified in the text.}
\label{tab:optimizer_comparison}
\end{table}

{We also report the terminal norms to distinguish the final mixing level from the decay rate fitted over the full trajectory. In the row order of Table~\ref{tab:optimizer_comparison}, the terminal $\dot H^{-1}$ norms are $0.024891$, $0.044328$, $0.036586$, and $0.024605$ for the Sobolev method, and $0.053247$, $0.048124$, $0.043481$, and $0.037869$ for the interface method.}

{The interface optimizer has lower measured optimization time in all four settings. The ratios are approximately $50.4$--$98.1$ for cellular flows and $25.3$--$63.3$ for Doswell flows. Since timing scopes, stopping rules and constraints differ, these ratios do not measure equal-accuracy speedups. The Sobolev method gives smaller terminal norms in all four cases. Thus lower optimization cost does not uniformly imply faster global mixing.}

\paragraph{Numerical verification}\label{sec:validation}
{We supplement the comparison with checks of the regularized length, the control gradient, and discretization sensitivity. For a straight segment with $\varepsilon=0.1$, the regularized length and its derivative agree with analytical values to $2.3\times10^{-16}$ for $10^5$ to $10^6$ markers. Curved-terminal tests and small full-coordinate automatic-differentiation checks also pass. At the production resolution, three prescribed directions satisfy the centered-difference and quadratic-remainder criteria at both initial and returned controls for the oscillatory cellular and two Doswell cases. The constant-start cellular case passes zero of three directions initially and one at the returned control. A separate higher-precision evaluation passes one and three, respectively. This diagnostic identifies sensitivity to arithmetic precision; it was not used in optimization and does not validate the original optimization trajectory.}

{To examine discretization sensitivity independently of optimization, we hold the four returned interface controls fixed. Increasing the material marker count from $100{,}000$ to $400{,}000$ through $200{,}000$, without redistribution, changes terminal length by $11.59\%$, $1.48\%$, $0.139\%$, and $0.00034\%$ in the row order of Table~\ref{tab:optimizer_comparison}. Separately halving the marker time step changes it by $-2.29\%$, $-2.91\%$, $-0.0295\%$, and $0.00265\%$. Each difference is relative to the original $100{,}000$-marker value. A further constant-start cellular check at $400{,}000$ markers gives terminal lengths $318.2533$ and $316.5262$ at the original and halved time steps, a change of $-0.543\%$ relative to the finer marker grid. The oscillatory cellular sequence is nonmonotone, and the constant-start cellular length remains spatially sensitive.}

{With the final transport implementation, all eight returned controls have complete $500^2$ and $1000^2$ evaluations and a separate time-step refinement. Across these comparisons, the terminal mix-norm changes by at most $2.01\%$ and the largest pointwise relative difference along a norm history is $4.08\%$. Every accepted linear solve satisfies its prescribed residual criterion; rejected attempts and the original restart steps are retained in the evaluation records. Independent offline reconstruction covers the saved fields, while the remaining steps rely on runtime residual checks. These tests quantify finite-resolution sensitivity, not continuum convergence.}

\subsection{{Comparison of two- and four-mode controls}}\label{sec:N4}

The preceding cellular experiments use $N=2$ modes. {We now compare two historical controls, one using these modes and the other using $N=4$ modes}, $h_k(x_1,x_2)=\sin(k\pi x_1)\sin(k\pi x_2)$ for $k\in\{1,2,3,4\}$. {The four-mode run was performed} on $\Omega=(0,1)^2$ with $t\in[0,1]$, $M=1000$ time steps, $N_p={100{,}000}$ markers, $\gamma_k=10^{-5}$ for all $k$, and constant initial guess $u_1=u_2=1$, $u_3=u_4=0$. The terminal ratio is $\|u^*\|_{L^2}/R_u = {0.19}${, so the returned control lies strictly inside the admissible ball}. {This subsection retains an earlier $N=2$/$N=4$ pair with its original optimization settings and reevaluates its mix-norms with the final transport implementation. Its $N=2$ reference is separate from the updated controls in Section~\ref{sec:optimal_flows}. Both runs stop on objective stabilization, without meeting the stationarity tolerance. The four-mode production gradient passes one of three prescribed directional checks initially and none at the returned control. We therefore use this pair as a descriptive fixed-control comparison, not as validation of an optimization benefit from additional modes.}

Figure~\ref{fig:interface_Cellular-N4} shows the interface evolution, and Figure~\ref{fig:theta_Cellular-N4} shows the scalar field. {The terminal display length $\mathcal L_{\mathrm{disp}}(T)$ is approximately $4647$, compared with $2236$ for the historical $N=2$ control. At the reported Eulerian resolution, the corresponding terminal mix-norms are $0.01044$ and $0.01648$, with full-horizon decay rates $2.78$ and $2.40$ (Figure~\ref{fig:Cellular_N4_mixnorm-length}). Figure~\ref{fig:Cellular_N4_mixnorm-length}(a) shows the nonmonotone display-length history without an exponential fit.}

{With both controls fixed, increasing the material marker count from $100{,}000$ to $400{,}000$ changes the terminal material lengths by $27.13\%$ for $N=2$ and $177.22\%$ for $N=4$. Their time-integrated squared $L^2$ velocity norms are $37.82$ and $103.80$, respectively. Thus, at the reported resolution, the lower terminal mix-norm for $N=4$ is accompanied by higher velocity effort and greater sensitivity of the material length to marker resolution.}

\begin{figure}[!htbp]
  \centering
  \begin{subfigure}[b]{0.16\textwidth}
    \includegraphics[width=\textwidth]{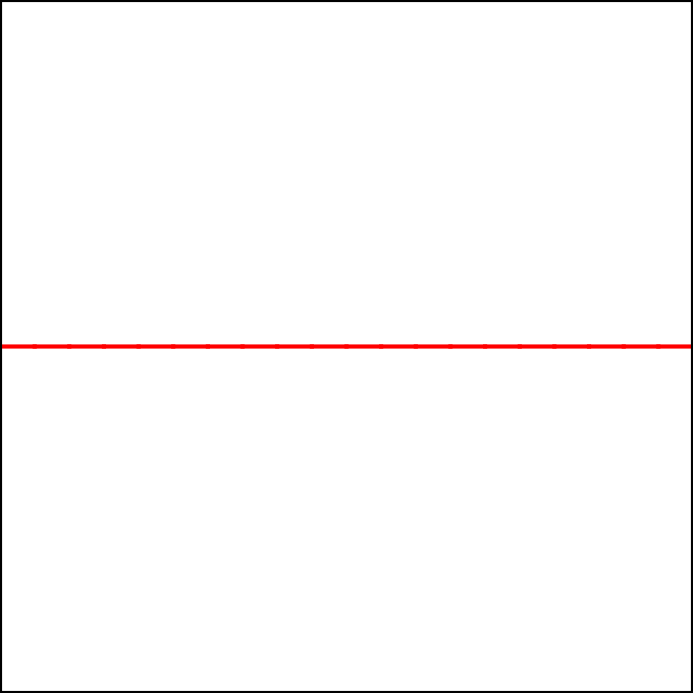}
    \caption{$t=0$}
  \end{subfigure}
  \begin{subfigure}[b]{0.16\textwidth}
    \includegraphics[width=\textwidth]{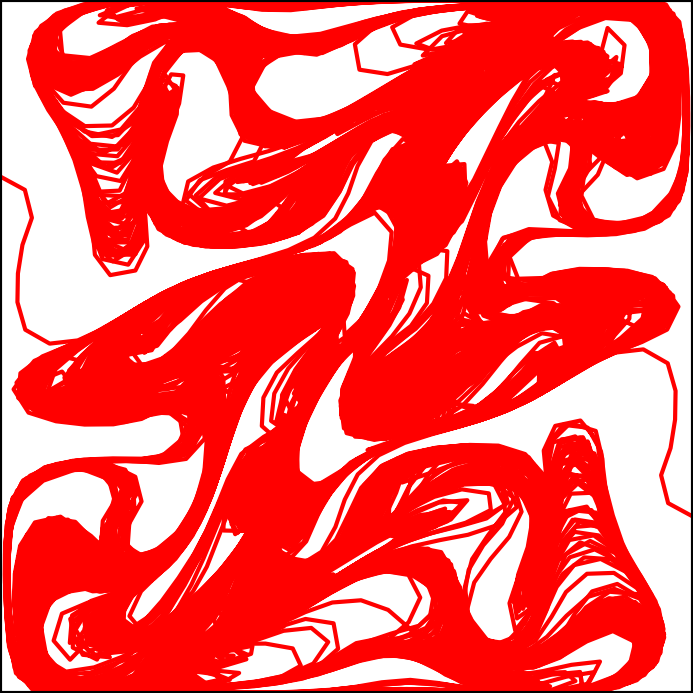}
    \caption{$t=0.2$}
  \end{subfigure}
  \begin{subfigure}[b]{0.16\textwidth}
    \includegraphics[width=\textwidth]{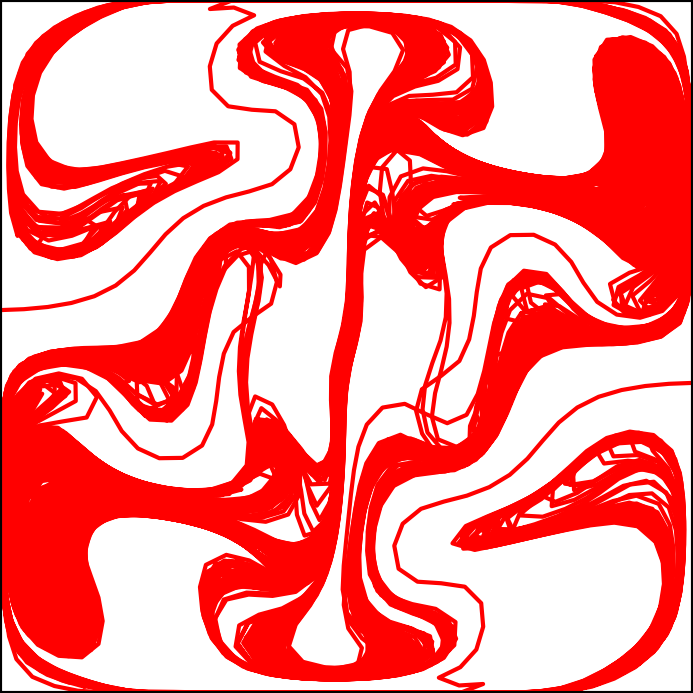}
    \caption{$t=0.4$}
  \end{subfigure}
  \begin{subfigure}[b]{0.16\textwidth}
    \includegraphics[width=\textwidth]{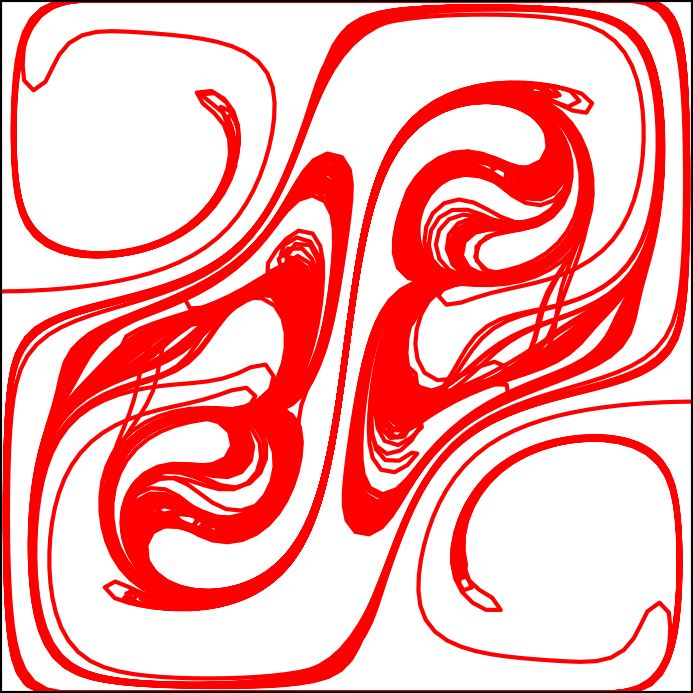}
    \caption{$t=0.6$}
  \end{subfigure}
  \begin{subfigure}[b]{0.16\textwidth}
    \includegraphics[width=\textwidth]{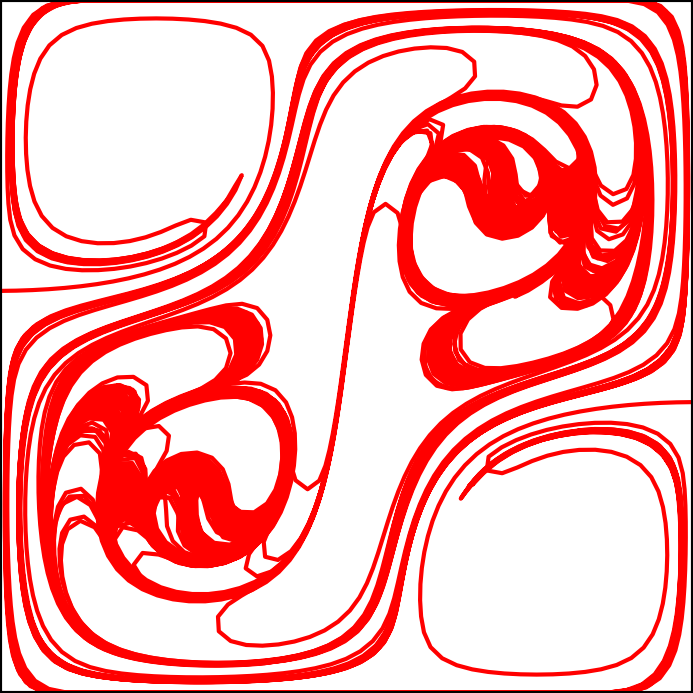}
    \caption{$t=0.8$}
  \end{subfigure}
  \begin{subfigure}[b]{0.16\textwidth}
    \includegraphics[width=\textwidth]{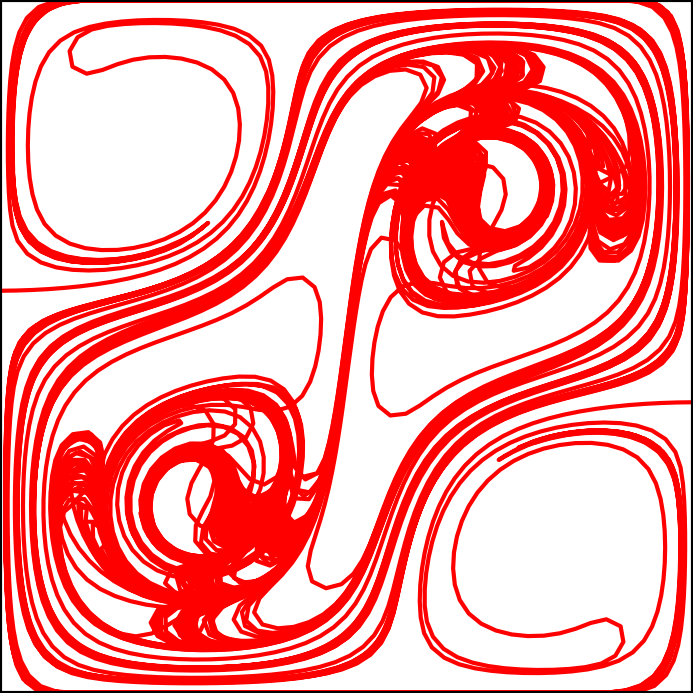}
    \caption{$t=1.0$}
  \end{subfigure}
  \caption{Interface evolution under the {historical four-mode control}.}
  \label{fig:interface_Cellular-N4}
\end{figure}

\begin{figure}[!htbp]
  \centering
  \begin{subfigure}[b]{0.16\textwidth}
    \includegraphics[width=\textwidth]{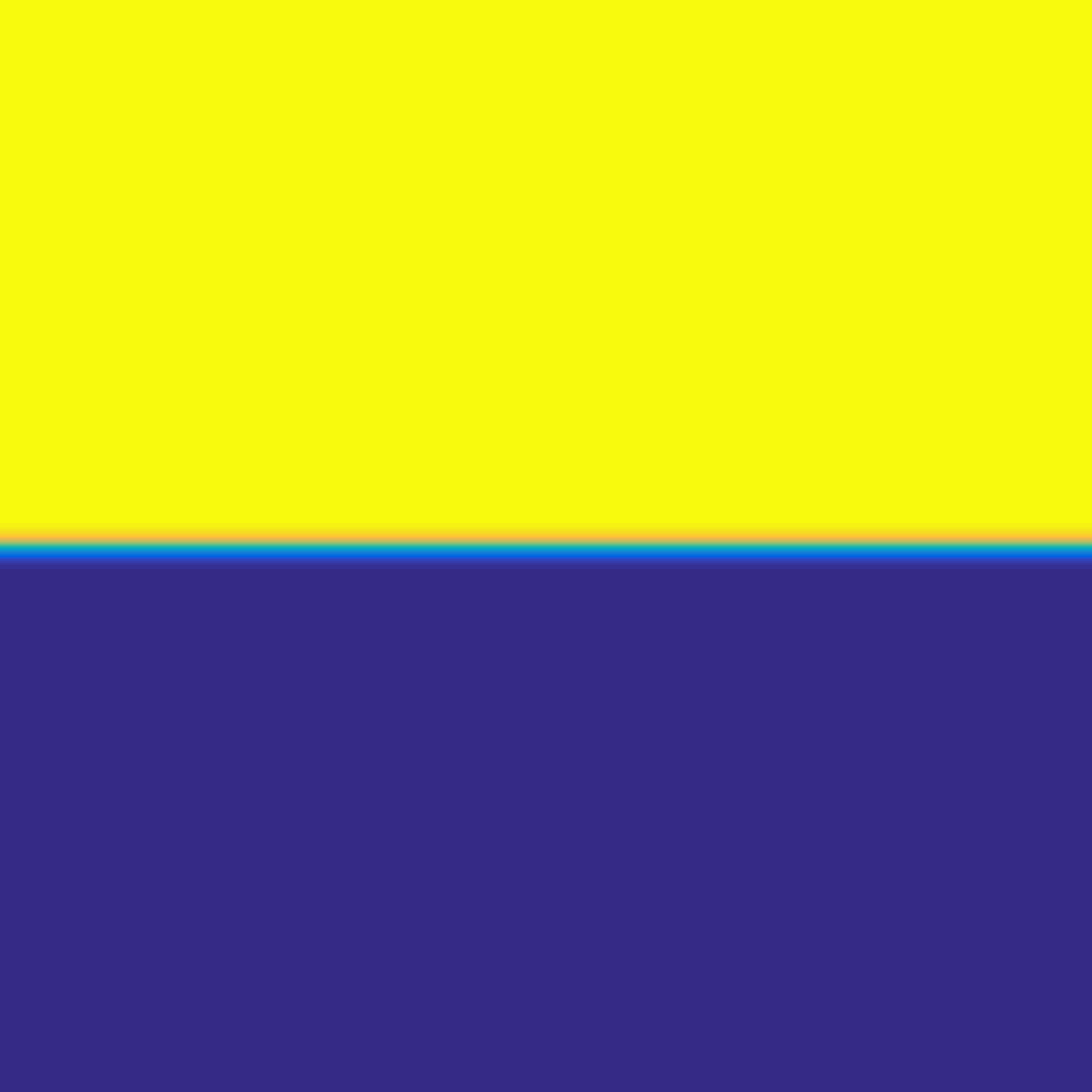}
    \caption{$t=0$}
  \end{subfigure}
  \begin{subfigure}[b]{0.16\textwidth}
    \includegraphics[width=\textwidth]{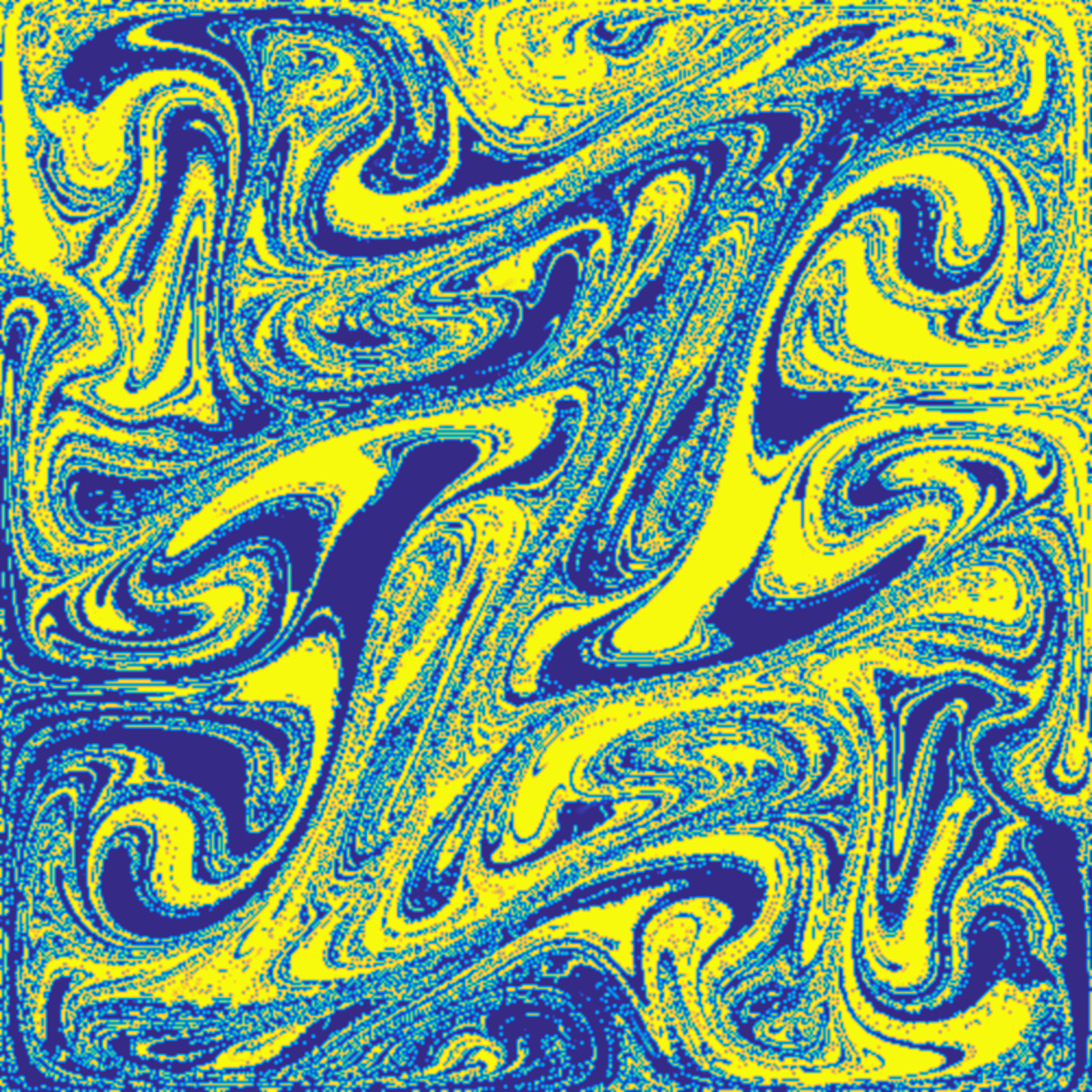}
    \caption{$t=0.2$}
  \end{subfigure}
  \begin{subfigure}[b]{0.16\textwidth}
    \includegraphics[width=\textwidth]{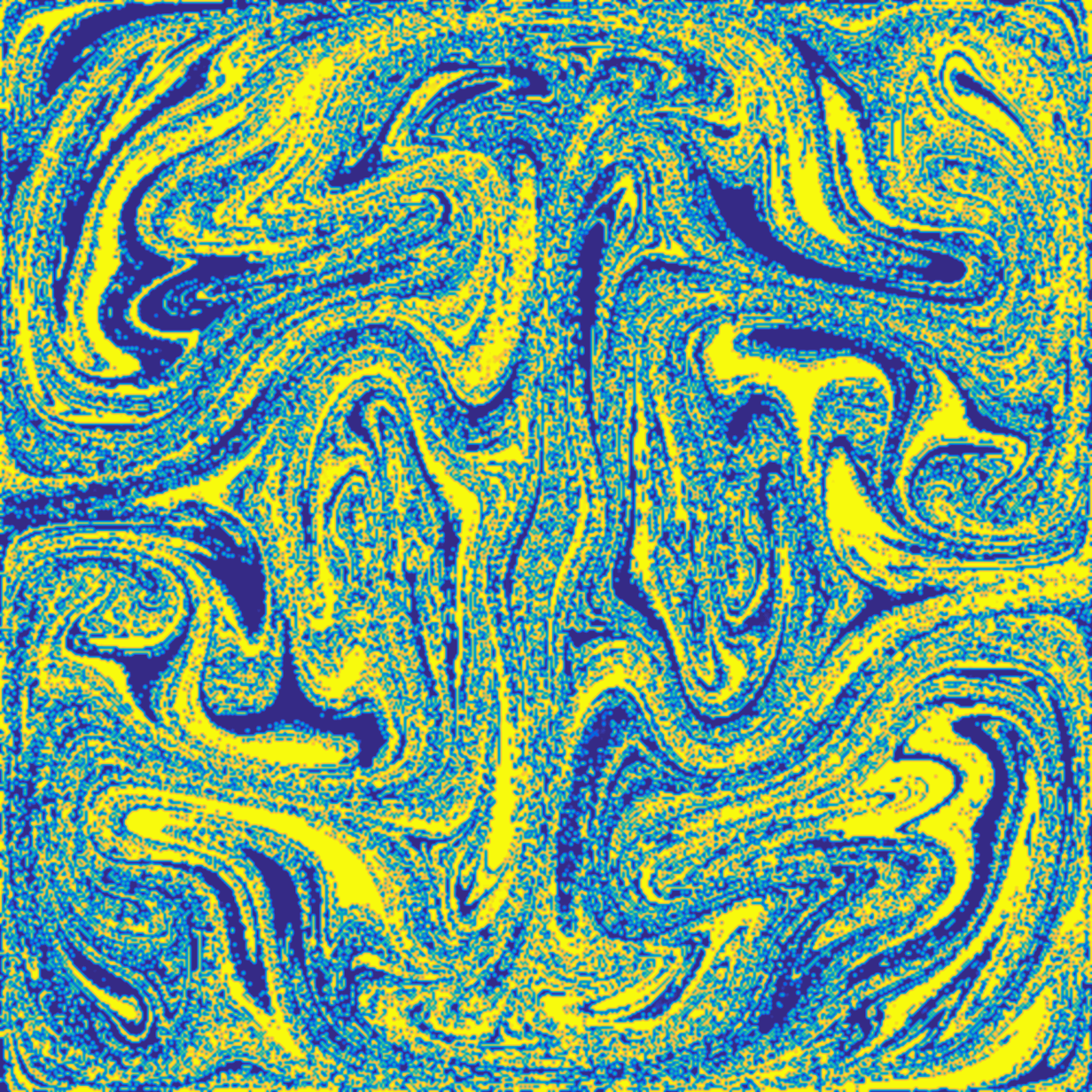}
    \caption{$t=0.4$}
  \end{subfigure}
  \begin{subfigure}[b]{0.16\textwidth}
    \includegraphics[width=\textwidth]{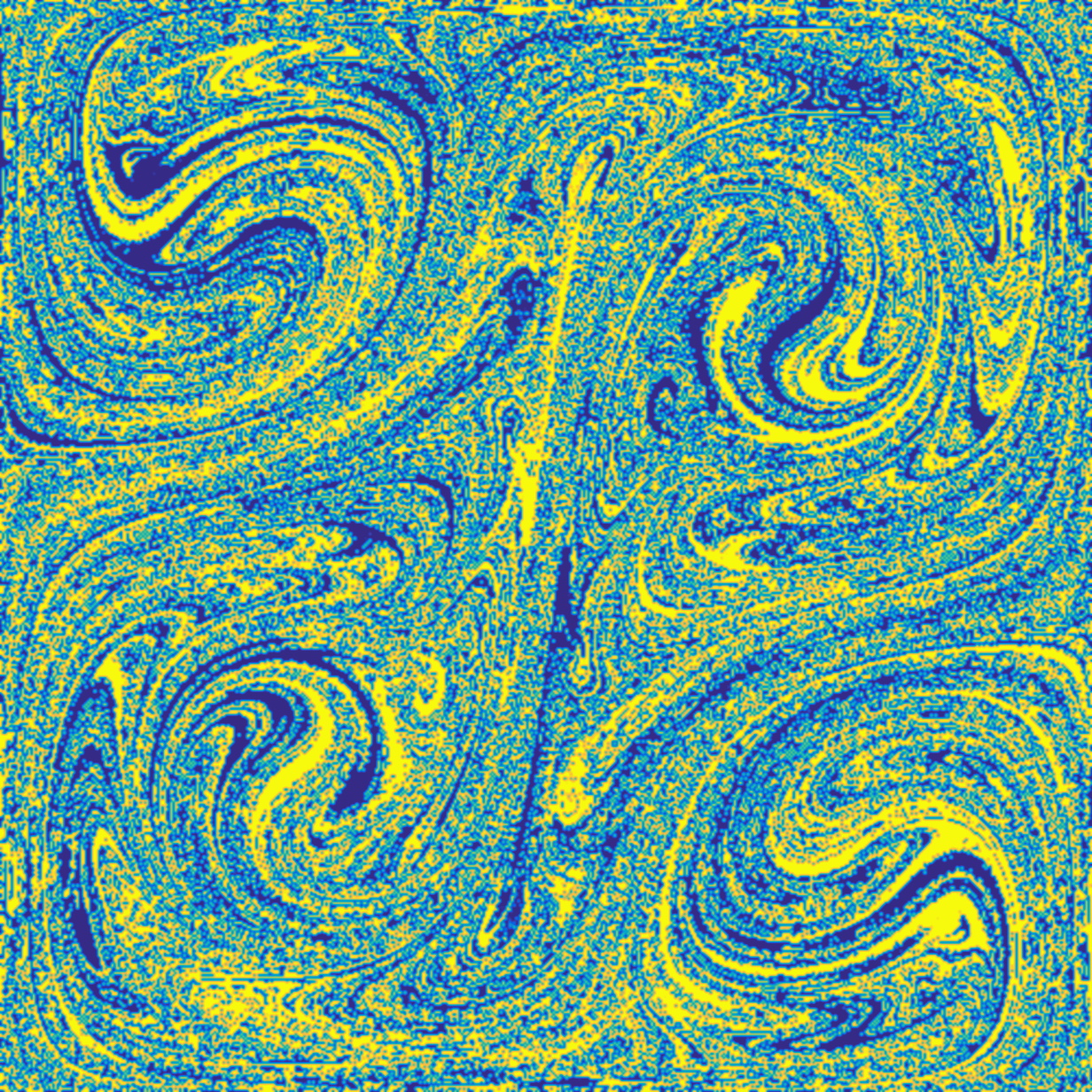}
    \caption{$t=0.6$}
  \end{subfigure}
  \begin{subfigure}[b]{0.16\textwidth}
    \includegraphics[width=\textwidth]{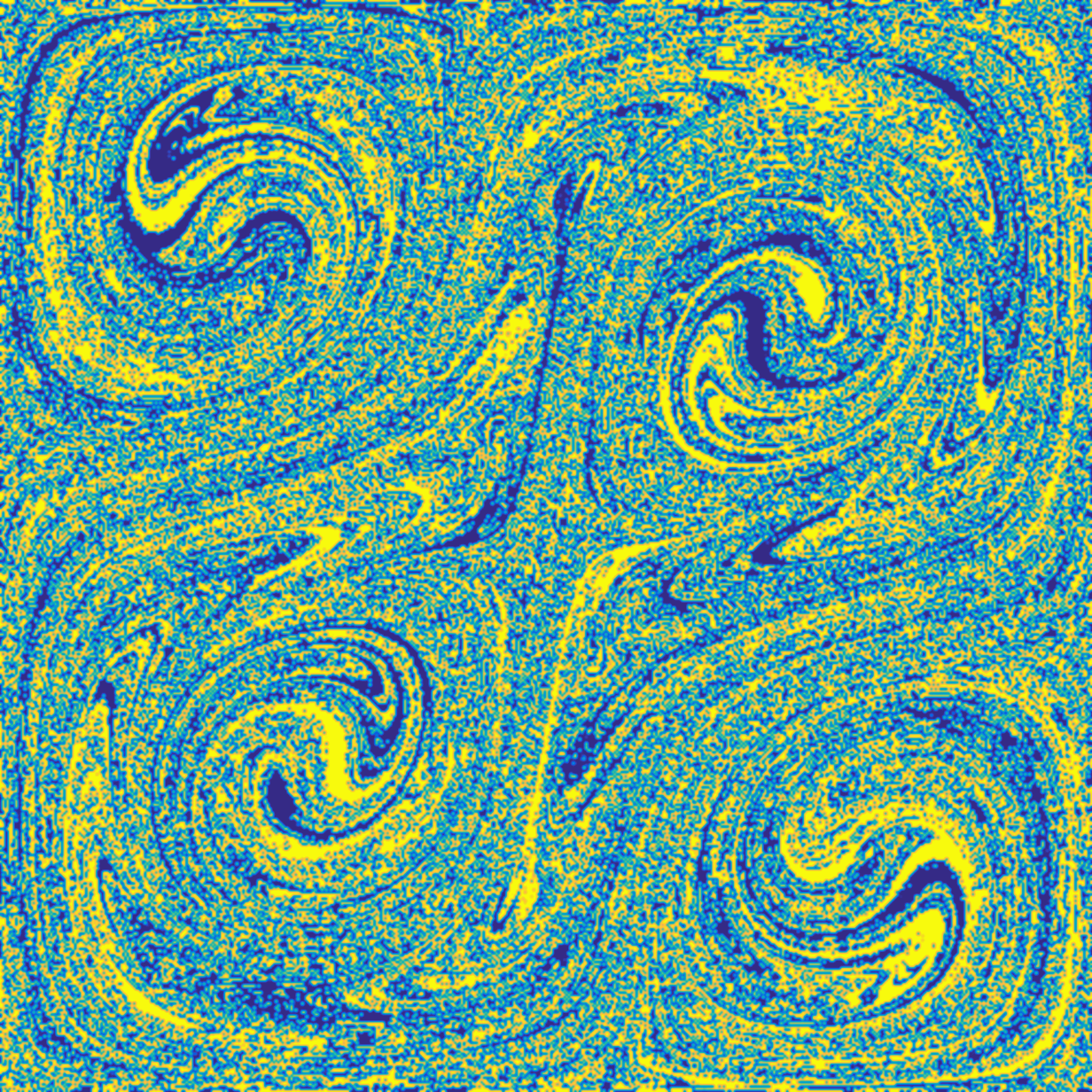}
    \caption{$t=0.8$}
  \end{subfigure}
  \begin{subfigure}[b]{0.16\textwidth}
    \includegraphics[width=\textwidth]{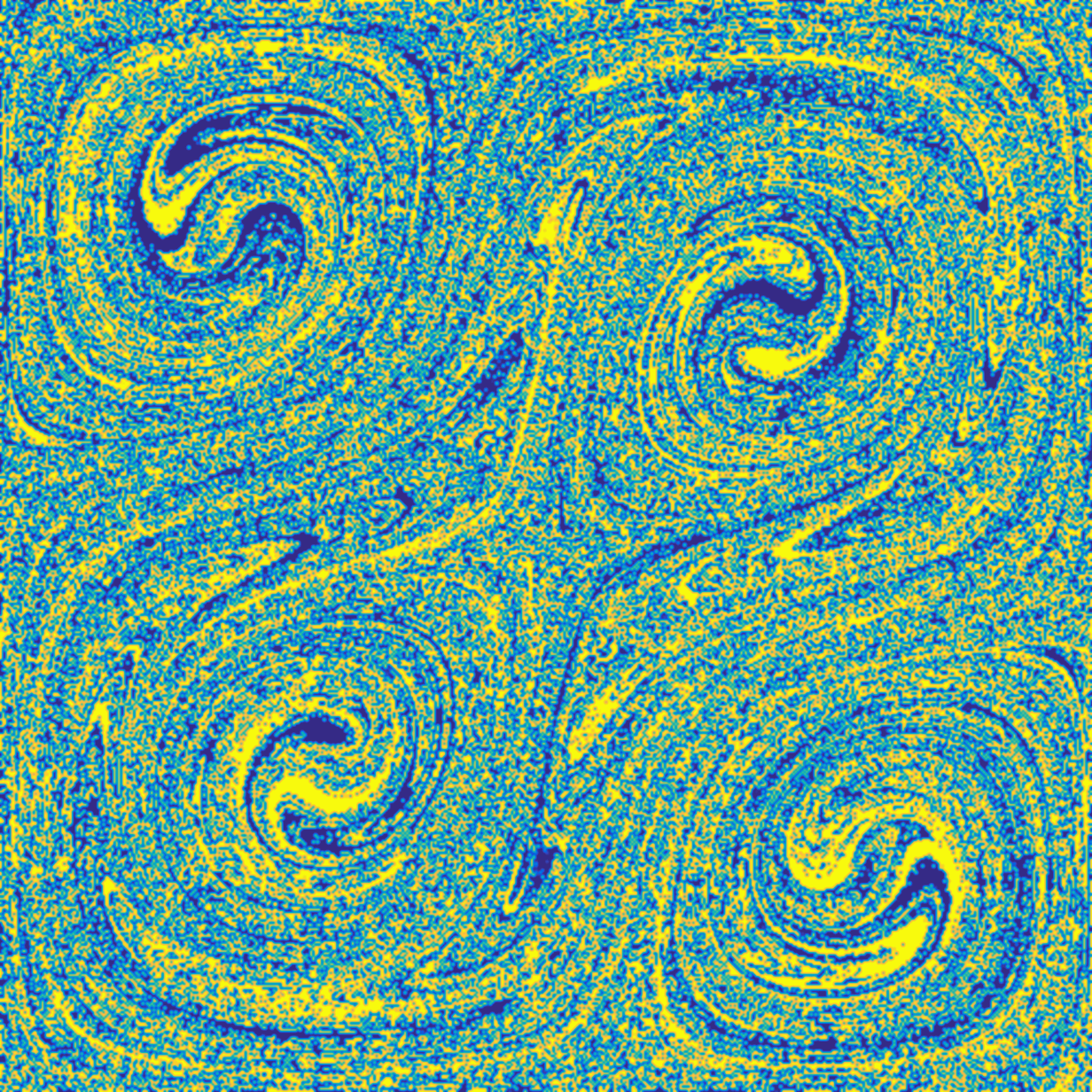}
    \caption{$t=1.0$}
  \end{subfigure}
  \caption{Scalar field {reconstructed by backward characteristics under the historical four-mode control}.}
  \label{fig:theta_Cellular-N4}
\end{figure}

\begin{figure}[!htbp]
  \centering
  \begin{subfigure}[b]{0.32\textwidth}
    \centering
    \includegraphics[width=\textwidth]{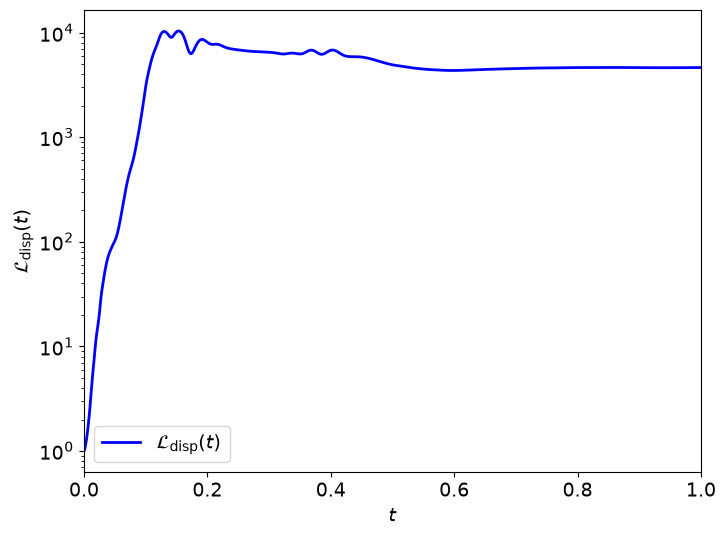}
    \caption{Interface length $\mathcal{L}_{{\mathrm{disp}}}(t)$.}
    \label{fig:Cellular_N4_length}
  \end{subfigure}
  \begin{subfigure}[b]{0.32\textwidth}
    \centering
    \includegraphics[width=\textwidth]{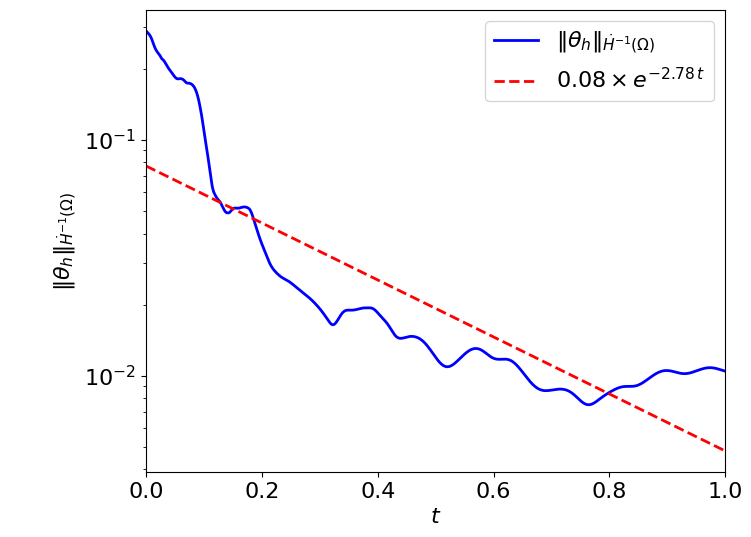}
    \caption{$\dot{H}^{-1}$ mix-norm.}
    \label{fig:Cellular_N4_mixnorm}
  \end{subfigure}
  \begin{subfigure}[b]{0.32\textwidth}
    \centering
    \includegraphics[width=\textwidth]{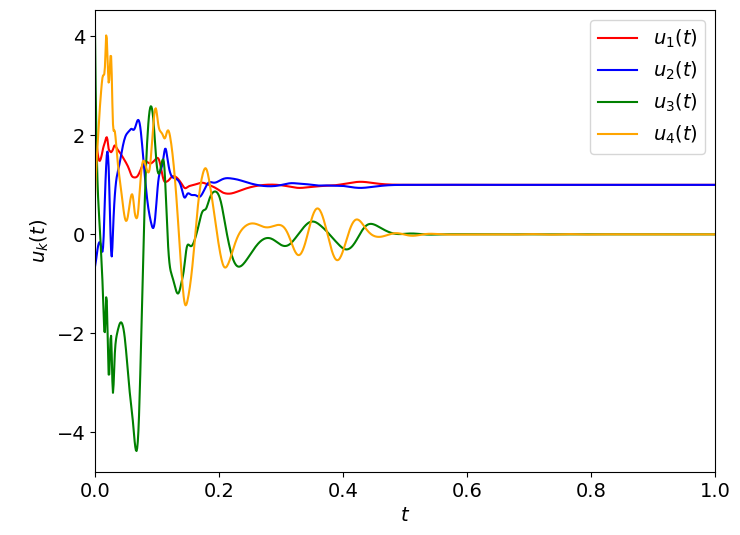}
    \caption{{Returned} $u_1,\dots,u_4$.}
    \label{fig:Cellular_N4_u1u2u3u4}
  \end{subfigure}
    \caption{Quantitative results under the {historical four-mode control}.}
  \label{fig:Cellular_N4_mixnorm-length}
\end{figure}

The controls are shown in Figure~\ref{fig:Cellular_N4_u1u2u3u4}.

\FloatBarrier
\section{Conclusions and perspectives}\label{sec:conclusions}

\subsection{Conclusions}
We proposed a reduced-order optimal control framework whose objective is the length of an advected material interface. Within a finite-dimensional {regular} stream-function ansatz we established existence of a minimizer (Theorem~\ref{prop:H3}) and derived the continuous adjoint and a discrete-consistent gradient (Remark~\ref{rem:discrete_exact}). {The four cellular and Doswell comparisons give lower recorded interface optimization times, while the Sobolev controls attain smaller terminal mix-norms. Different constraints, stopping rules and timing scopes preclude an equal-accuracy speedup claim. These results motivate using reduced-model optimization alongside direct PDE-based optimization, in the spirit of Glowinski's two-grid philosophy~\cite{Glowinski1992}.}

\subsection{Perspectives}
Several directions for future work arise.
\begin{itemize}
\item \textbf{Beyond interface length: sharper Lagrangian surrogates.}
A natural direction is to characterize the $\dot{H}^{-1}$-mixing functional more sharply in terms of the evolving interface, {complementing results on mixing and enhanced dissipation under different flow constraints} \cite{alberti2016exponential, crippa2017cellular, brue2024enhanced}. {Adaptive control-dimension strategies would require verified gradients and adequately resolved material geometry.}

\item \textbf{Mixed fine-forward/reduced-adjoint strategy.} In the spirit of Glowinski's two-grid philosophy~\cite{Glowinski1992}, evolving the scalar on a fine Eulerian grid while computing adjoints on the reduced Hamiltonian model {could reduce sensitivity costs. Such a coupling requires a consistent transfer of PDE sensitivities or a justified surrogate-descent mechanism.}
\end{itemize}

\section*{Acknowledgments}
The authors thank Yongcun Song for helpful discussions.

\bibliographystyle{abbrv}
\enlargethispage{2\baselineskip}
\bibliography{refs_mixing}

\end{document}